\documentclass[oneside,reqno,12pt]{amsart}

\usepackage[T1]{fontenc}
\usepackage{lmodern,times,microtype,comment, esint} % for better looking text
\usepackage[usenames,dvipsnames]{xcolor}
\usepackage{hyperref,graphicx}
\hypersetup{colorlinks,citecolor=Violet,linkcolor=NavyBlue, urlcolor=Blue}

\usepackage[left=1in,right=1in,top=1in, bottom=1in]{geometry}
\usepackage{amsmath,amsfonts,amssymb,enumitem,mathtools,bm}

\usepackage[normalem]{ulem} % for the command \sout (strikeout)

\DeclareMathOperator\Tr{Tr}

\DeclareMathOperator\diag{diag}

\DeclareMathOperator\Hess{Hess}

\DeclareMathOperator\Int{Int}

\newcommand{\tot}{\tfrac{1}{2}} % one half in a small-font frac
\newcommand{\oo}[1]{\tfrac{1}{#1}}
\newcommand{\scl}[2]{\langle #1,#2 \rangle} % scalar product
\newcommand{\sabs}[1]{| #1 |} % absolute value
\newcommand{\abs}[1]{\left| #1 \right|} % absolute value
\newcommand{\babs}[1]{\big| #1 \big|} % absolute value
\newcommand{\ab}[1]{\langle #1 \rangle} % quadratic variation

\newcommand{\set}[1]{\{#1\}} % curly brackets
\newcommand{\sets}[2]{\set{#1\,:\,#2}} % a set with "such that"
\newcommand{\Bset}[1]{\Big\{#1\Big\}} % Big curly brackets
\newcommand{\Bsets}[2]{\Bset{#1\,:\,#2}} % a Big set with "such that"

\newcommand{\ind}[1]{ {\mathbf 1}_{{#1}}} % convex-theoretic indicator of a set
\newcommand{\inds}[1]{ {\mathbf 1}_{\set{#1}}} % indicator of a curly set
\newcommand{\seq}[1]{\set{#1_n}_{n\in\N}} % a sequence indexed by n
\newcommand{\norm}[1]{{||#1||}} % norm enclosure
\newcommand{\ft}[2]{#1\dots#2}
\renewcommand{\ft}[2]{#1,\dots,#2}

\newcommand{\prf}[1]{ \{ #1 \}_{t\in [0,T]}}

\newcommand{\downto}{\searrow}

\renewcommand{\implies}{\rightarrow}

\providecommand{\R}{} \renewcommand{\R}{{\mathbb R}}

\providecommand{\N}{} \renewcommand{\N}{{\mathbb N}}

\newcommand{\PP}{{\mathbb P}}

\newcommand{\EE}{{\mathbb E}}

\newcommand{\FFF}{{\mathbb F}}

\newcommand{\bmu}{\boldsymbol{\mu}}
\newcommand{\bnu}{\boldsymbol{\nu}}

\newcommand{\eps}{\varepsilon}
\newcommand{\ld}{\lambda}
\newcommand{\Ld}{\Lambda}

\newcommand{\vp}{\varphi}

\newcommand{\el}{{\mathbb L}} %l-pees

\newcommand{\lone}{\el^1}
\newcommand{\ltwo}{\el^2}

\newcommand{\linf}{\el^{\infty}}
\newcommand{\eqm}{\stackrel{m}{=}}

\newcommand{\pds}[1]{\frac{\partial}{\partial #1}}
\newcommand{\tpds}[1]{\tfrac{\partial}{\partial #1}}

\newcommand{\tds}[1]{\tfrac{d}{d #1}}

\newcommand{\tddt}{\tds{t}}
\newcommand{\tpddt}{\tpds{t}}

\newcommand{\pdm}[2]{\frac{\partial^2}{\partial #1\partial #2}}

\newcommand{\define}[1]{{\textbf{#1}}}

\newcounter{notecounter}

\newcommand{\efor}{\text{ for }}

\newcommand{\eand}{\text{ and }}

\newcommand{\ewhere}{\text{ where }}

\newcommand{\ewhen}{\text{ when }}

\newcommand{\itos}{It\^ o's}
\newcommand\bsa{{\boldsymbol{a}}}
\newcommand\bsb{{\boldsymbol{b}}}

\newcommand\bsc{{\boldsymbol{c}}}

\newcommand\bsfe{{\boldsymbol{\mathsf{e}}}}

\newcommand\bsf{{\boldsymbol{f}}}

\newcommand\bsg{{\boldsymbol{g}}}

\newcommand\bsh{{\boldsymbol{h}}}

\newcommand\bsfk{{\boldsymbol{\mathsf{k}}}}

\newcommand\bsfl{{\boldsymbol{\mathsf{l}}}}
\newcommand\bsfd{{\boldsymbol{\mathsf{d}}}}
\newcommand\bsfL{{\boldsymbol{\mathsf{L}}}}

\newcommand\bsp{{\boldsymbol{p}}}

\newcommand\bsq{{\boldsymbol{q}}}
\newcommand\sfq{{\mathsf{q}}}
\newcommand\bsfq{{\boldsymbol{\mathsf{q}}}}

\newcommand\bss{{\boldsymbol{s}}}
\newcommand\bsfs{{\boldsymbol{\mathsf{s}}}}

\newcommand\bsu{{\boldsymbol{u}}}

\newcommand\bsv{{\boldsymbol{v}}}

\newcommand\bsw{{\boldsymbol{w}}}

\newcommand\bsx{{\boldsymbol{x}}}

\newcommand\tx{{\tilde{x}}}

\newcommand\bsy{{\boldsymbol{y}}}

\newcommand\bsz{{\boldsymbol{z}}}

\newcommand\bsA{{\boldsymbol{A}}}

\newcommand\tA{{\tilde{A}}}

\newcommand\sC{{\mathcal C}}

\newcommand\tD{{\tilde{D}}}

\newcommand\sF{{\mathcal F}}
\newcommand\bsF{{\boldsymbol{F}}}

\newcommand\bsG{{\boldsymbol{G}}}

\newcommand\sL{{\mathcal L}}
\newcommand\bsL{{\boldsymbol{L}}}

\newcommand\sP{{\mathcal P}}

\newcommand\tT{{\tilde{T}}}

\newcommand\bsW{{\boldsymbol{W}}}

\newcommand\bsX{{\boldsymbol{X}}}

\newcommand\bsY{{\boldsymbol{Y}}}

\newcommand\bsZ{{\boldsymbol{Z}}}

\numberwithin{equation}{section}
\theoremstyle{plain}                % title and number in bold, text italic
\newtheorem{theorem}{Theorem}[section]
\newtheorem{lemma}[theorem]{Lemma}
\newtheorem{proposition}[theorem]{Proposition}
\newtheorem{corollary}[theorem]{Corollary}

\theoremstyle{definition}           % title and number in bold, text normal
\newtheorem{definition}[theorem]{Definition}

\newtheorem{assumption}[theorem]{Assumption}

\theoremstyle{remark}
\newtheorem{remark}[theorem]{Remark}

\DeclareMathOperator{\supp}{supp}

\newcommand{\bsxi}{\boldsymbol{\xi}}
\newcommand{\bsig}{\boldsymbol{\sigma}}

\renewcommand{\tx}{t,\bsx}

\newcommand{\txp}{t',\bsx'}
\newcommand{\txz}{t_0,\bsx_0}

\newcommand{\TR}{[0,T]\times \R^d}
\newcommand{\TBn}{[0,T]\times B_n}

\newcommand{\nvl}{\norm{\bsv}_{\linf}}

\newcommand{\lec}{\leq_C}
\newcommand{\gec}{\geq_C}
\newcommand{\dl}{\delta_{\usig/4}}
\newcommand{\tdl}{\delta_{\osig}}

\newcommand{\tvp}{\tilde{\vp}}
\newcommand{\tchi}{\tilde{\chi}}
\newcommand{\tpsi}{\tilde{\psi}}

\newcommand{\btau}{\bar{\tau}}
\newcommand{\peps}{p_{\eps}}
\renewcommand{\tt}{\tilde{t}}

\def\mathclap#1{\text{\hbox to 0pt{\hss$\mathsurround=1pt#1$\hss}}}
\newcommand{\iintl}[1]{\iint\limits_{\mathclap{#1}}}
\newcommand{\intl}[1]{\int\limits_{\mathclap{#1}} }

\newcommand{\ovps}[1]{\overline{#1}^{\psi}}

\newcommand{\ovrt}{\ovps{\bsv(t,\cdot)}}

\newcommand{\ogr}{\ovps{\bsg}}

\newcommand{\uC}{\underline{C}}
\newcommand{\oC}{\overline{C}}

\newcommand{\usig}{\underline{\sigma}}
\newcommand{\osig}{\overline{\sigma}}

\newcommand{\bmo}{\textrm{bmo}}
\newcommand{\BMO}{\textrm{BMO}}

\newcommand{\bszeta}{\boldsymbol{\zeta}}
\newcommand{\bsld}{\boldsymbol{\ld}}

\renewcommand\sfq{{\mathsf{q}}}
\newcommand\sfs{{\mathsf{s}}}

\newcommand{\Ca}{C^{\set{\alpha_n}}}
\newcommand{\Caa}{C^{\{\alpha_n\}}}

\newcommand{\Caap}{C^{\set{\alpha'_n}}}

\newcommand{\tbsvl}{\hat{\bsv}^{(l)}}

\newcommand{\tbsY}{\hat{\bsY}}

\newcommand{\tbsYl}{\hat{\bsY}^{(l)}}

\newcommand{\bsYm}{\bsY^{(m)}}
\newcommand{\bsYmp}{\bsY^{(m')}}
\newcommand{\bsZm}{\bsZ^{(m)}}
\newcommand{\bsZmp}{\bsZ^{(m')}}

\newcommand{\umo}{\text{uBMO}}
\newcommand{\tbsv}{\hat{\bsv}}

\newcommand{\Ly}{\mathbf{Ly}}
\newcommand{\lLy}{\mathbf{Ly}_{loc}}
\newcommand{\BF}{\mathbf{BF}}
\newcommand{\BFa}{{\mathbf{BF}}_{e}}

\renewcommand{\bsX}{X}
\renewcommand{\bsx}{x}
\renewcommand{\bsb}{b}
\renewcommand{\bsig}{\sigma}
\renewcommand{\bsW}{W}
\newcommand{\one}{\boldsymbol{1}}
\renewcommand{\seq}[1]{\{#1_n\}}
\renewcommand{\bsa}{a}
\renewcommand{\bsxi}{\xi}
\newcommand{\BFckq}{\BF(\seq{C}, \seq{\kappa}, \seq{q})}
\newcommand{\bssa}{\boldsymbol{a}}
\newcommand{\bseta}{\boldsymbol{\eta}}

\let\asum\sum
\renewcommand{\sum}{\textstyle\asum}

\title[Globally solvable BSDE systems]{A class of globally solvable
Markovian quadratic BSDE systems and applications}\thanks{The authors are
grateful to Ying Hu for inspiring discussions, and both referees and the associated editor for their valuable suggestions.
  The second author acknowledges the support
  by the National Science Foundation under Grants No.,
  No.~DMS-1107465 (2012 - 2017) and No.~DMS-1516165 (2015-2018). Any
  opinions, findings and
  conclusions or recommendations expressed in this material are those of the
  author(s) and do not necessarily reflect the views of the National Science
  Foundation (NSF)}

\author{Hao Xing}
\address[Hao Xing]{Department of Statistics,
London School of Economics and Political Science,
London, UK}
\email{h.xing@lse.ac.uk}
\author{Gordan {\v Z}itkovi{\' c}}
\address[Gordan {\v Z}itkovi{\' c}]{Department of Mathematics, The
University of Texas at Austin, Austin, TX, USA}
\email{gordanz@math.utexas.edu}
\begin{document}
\date{\today}
\begin{abstract}
We establish global existence and uniqueness for a wide class of Markovian
systems of backward stochastic differential equations (BSDE) with quadratic
nonlinearities.  This class
is characterized by an abstract structural assumption on the generator,
an a-priori local-boundedness property, and a locally-H\" older-continuous
terminal condition. We present easily verifiable sufficient
conditions for these assumptions and treat several applications, including stochastic
equilibria in incomplete financial markets, stochastic
differential games, and martingales on Riemannian manifolds.
\end{abstract}
\maketitle

\thispagestyle{empty}

\section{Introduction}

\subsection{Backward Stochastic Differential Equations}
Having appeared
first in their linear variant in \cite{Bis73},
backward stochastic differential equations
(BSDE) have been a subject of extensive study since the seminal paper
\cite{Pardoux-Peng}. Given
a time horizon $T\in (0,\infty)$
and a filtered
probability space $(\Omega,\sF,\prf{\sF_t},\PP)$ which satisfies the usual
conditions,
these equations take the following form
\begin{align}
   \label{equ:BSDE-gen}
   \bsY_t = \bsG + \int_t^T \bsf(s,\bsY_s, \bsZ_s)\,ds - \int_t^T \bsZ_s\, dW_s,
\end{align}
where $W$ is a $d$-dimensional $\prf{\sF_t}$-Brownian
motion, $\bsG\in \sF_T$ an $N$-dimensional  random vector and
$\bsf$ a (possibly random) function, called the
generator. A solution to such an equation is a pair $(\bsY, \bsZ)$
consisting of an $N$-dimensional
semimartingale $\bsY$ and an $N\times d$-dimensional adapted process $\bsZ$
which, together, satisfy \eqref{equ:BSDE-gen} pathwise, a.s.

\smallskip

The existence- and uniqueness theory is well developed in the scalar ($N=1$)
setting. It originated with the Lipschitz-generator case treated in
\cite{Pardoux-Peng}, continued in \cite{Lepeltier-SanMartin}
for merely
continuous generators with linear growth, and culminated with
the treatment of quadratic nonlinearities in \cite{Kobylanski} and superquadratic nonlinearities in \cite{DelHuBao11}. A host of
extensions, simplifications, and applications, too numerous to list here,
appeared in the literature since.

\smallskip

On the other hand, systems ($N>1$) of BSDE  - the focus of the present
paper -  pose a greater challenge. Their successful treatment is one
of the most important (and long-standing) open problems in the entire theory,
as mentioned already by Peng in \cite{Pen99}.  While the case of a Lipschitz generator was treated already in \cite{Pardoux-Peng}, the
general, nonlinear, quadratically-growing case is still open. One of
the most well-known general-purpose results has been established in
\cite{Tevzadze}, where the generator has general quadratic growth, but the terminal condition is subject to a restrictive
``smallness'' assumption imposed on its $\linf$-norm.
As is the case in the theory of systems of parabolic PDEs with quadratic nonlinearities,
a smallness assumption is often made and seems to be necessary for
existence in full generality, in absence of any further, structural assumptions; cf. \cite{Str81} and \cite{Chang-Ding-Ye}.
A simple nonexistence example given in \cite{FreRei11} illustrates
this point quite effectively.

\smallskip

Positive results without any smallness assumptions have been established in some special cases. Focusing only
on the general existence results in the multidimensional case pertinent to our
findings, let us mention just a few: \cite{Tan03} deals with linear-quadratic
systems, \cite{CheNam15} builds a structure around the ability to change
the probability measure in the Markovian case, and in \cite{HuTan15}, a
slightly less general, ``diagonally''-quadratic case is treated, but
without the Markovian assumption.

\smallskip

Our present work was motivated not only by the unresolved status of the
basic existence and uniqueness problems for quadratic systems of BSDE,
but also by a number of applications such systems have.
Indeed, in addition to their innate
mathematical interest, BSDE appear in numerous applications, including
stochastic representations for partial differential
equations, optimal stochastic control and  stochastic games (see, e.g.
\cite{El-Karoui-Hamadene}, \cite{Cheridito-et.al},
\cite{Espinosa-Touzi},
and \cite{Kramkov-Pulido}).
Moreover, as has been shown in \cite{Kardaras-Xing-Zitkovic}, arguably the
most important open problem in stochastic financial economics -
namely, the so-called
incomplete-market equilibrium problem -  can be reduced to a quadratic system
of BSDE (which we solve in the present paper).
Quadratic systems of BSDE also appear in geometry, most prominently in
the study of harmonic functions between manifolds and the construction of
martingales on curved spaces, (see,
e.g., \cite{Darling}, \cite{Bla05} and \cite{Bla06}).

\subsection{Our contributions - the main result.} We focus on
a Markovian framework, where the randomness in the generator and the
terminal condition is supplied by a (forward) $d$-dimensional
non-degenerate diffusion $X$. Our terminal condition is of the form
$\bsG=\bsg(\bsX_T)$ and the generator $\bsf(t,X_t,\bsY_t, \bsZ_t)$ grows at
most quadratically in $\bsZ$. We formulate and work with a novel structural
condition on $\bsf$ which requires the existence of what we term a
\emph{Lyapunov function}.
Loosely speaking, a
Lyapunov function $h$ has the property that $h(\bsY)$ is a ``strict''
submartingale, a-priori, for any solution $\bsY$ of \eqref{equ:BSDE-gen}
(actually, we consider a pair of functions, but we
focus on only one of them in this introduction).
Under the quadratic-growth assumption,
Lyapunov functions always exist in the $1$-dimensional case and can be
found in the class of exponential functions;
this explains the success of the exponential transform in the
$1$-dimensional setting. The multidimensional case appears to be much more
difficult, but as we show, widely applicable sufficient conditions can be
given.

\smallskip

Our main result states that when a Lyapunov function exists and an
additional a-priori local-boundedness condition holds, the equation
\eqref{equ:BSDE-gen} admits a Markovian solution as soon as $\bsg$
belongs to an appropriately-defined local H\" older space, \emph{without
any assumptions of the ``smallness'' type} on the terminal condition, the
driver, or the time horizon.  Moreover, under an additional mild assumption,
this solution turns out to be unique in a wide class of stochastic
processes.

\smallskip

In contrast to the bulk of the literature on multidimensional BSDE, we
rely on deep analytic results for systems of PDEs and combine them with
probabilistic techniques. More precisely, we use powerful ideas first
introduced to study regularity theory for systems of parabolic PDEs, most
notably the so-called \emph{partial-regularity} estimates and the
\emph{hole-filling} technique of \cite{Widman1971},  initially developed
for elliptic systems and later extended to parabolic systems of PDEs in
\cite{Str81}, which was later used in \cite{Bensoussan-Frehse}.  Partial-regularity and hole-filling
techniques can be seen as a replacement for order-based
arguments involving, for instance,  the maximum principle (comparison
principle) or the related notion of a viscosity solution; such methods,
unfortunately, fail miserably in the multidimensional case (see
\cite{HuPen06}). However we represent this analytic theory entirely in a
probabilistic fashion. This allows us to implement the
hole-filling technique only assuming the existence of a Lyapunov function.
This strategy not only decouples the hole-filling technique from specific
structural conditions on the nonlinearity such as the ``smallness"
condition in \cite{Str81} and the structural condition in
\cite{Bensoussan-Frehse}. It also links naturally to the notion of
\emph{geodesical convexity} in the studies of martingales on manifolds (see
Section \ref{subsec:mart}). The probabilistic treatment also reduces some
technical estimates from its the analytic counterpart. In particular, by replacing
integration-by-parts techniques by \itos{} formula, we 
bypass the estimates on derivatives of
Gaussian transition densities present in \cite{Str81} and
\cite{Bensoussan-Frehse}.

\smallskip

A major difficulty in adopting the techniques from the theory of systems of
PDE lies in
the choice of the regularity class of the Markov representatives, i.e.,
functions $\bsv$ such that $\bsY_t=\bsv(t,X_t)$ is a solution. On one hand,
the classical notion of a weak solution - typically a starting
point for any regularity analysis in the PDE world - is too weak for us; indeed, the
very definition of a solution to a BSDE requires $\bsY$ to be a
semimartingale (see, however, \cite{Barles-Lesigne},
\cite{Bally-Matoussi}, \cite{Lejay}, and \cite{Matoussi-Xu} for
developments in Lipschitz systems). On the opposite end of the spectrum,
a classical $C^{1,2}$-solution would, indeed, guarantee the
semimartingale property of $\bsv(t,X_t)$, but one can hardly expect that
kind of regularity from a solution to a nonlinear system. In the
one-dimensional case, the situation is fully understood - Markov
representatives of solutions to Markovian BSDE in dimension $1$ are
viscosity solutions to the associated quasilinear PDEs (see \cite{ParPen92},
\cite{Kobylanski}, \cite{BriHu08}
or \cite{DelHuBao11}). The multidimensional case, again, presents major
difficulties: unless the system is very weakly coupled (only in its zero-th order
terms), there is no natural notion of a viscosity solution and there is no
corresponding characterization of the class of semimartingale functions (see,
however, \cite{ChiMan97} for a related result in the Brownian case).
However, in many applications, the automatic semimartingale
property is especially useful, as it allows us to perform so-called
``verification'' directly and without additional assumptions or the
invocation of the dynamic-programming principle.

\smallskip

The way we overcame these difficulties in the present paper is by:
1) approximating our system by a
sequence of well-behaved systems, 2) combining analytic methods described
above with probabilistic ones to obtain adequate uniform estimates on these
approximations, and 3) showing that the passage to the limit preserves the
semimartingale property (as well as the equation itself), while relying mostly on probabilistic arguments. This way, we obtain a solution of the form
$\bsY_t=\bsv(t,X_t)$, $\bsZ_t = \bsw(t,X_t)$, where $\bsv$ is locally H\" older
continuous, $\bsv(t,X_t)$ is a semimartingale, and $\bsw$ is the weak
Jacobian of $\bsv$. his strategy bypasses regularity and
pointwise estimates on $\bsw$, which is typically needed to establish a PDE
solution in more analytical approaches. While we are still far from 
complete understanding of the appropriate class of functions  to
replace viscosity solutions in the multidimensional case, we feel that the functions with
above properties constitute a promising first step.

\subsection{Our results - sufficient conditions and examples.}
As a complement to our main existence/\-unique\-ness theorem, we provide 
a sufficient condition for the existence of Lyapunov functions - termed
the Bensoussan-Frehse (BF) condition - as well as a somewhat simpler
sufficient condition for a-priori boundedness. The (BF) condition, a list
of algebraic conditions on various terms
in the generator, is a slight generalization of the structural condition discovered by Bensoussan and
Frehse in \cite{Bensoussan-Frehse}. We add a term of sub-quadratic growth and also allow for a small 
``error'' around the structure, thus incorporating both ``smallness'' and Bensoussan and Frehse's
structural condition  into a single condition. 
Generators in many solvable Markovian BSDE
systems described in the prior literature satisfy our condition (BF). Nevertheless (BF) may not be suitable for our Example \ref{subsec:mart} below where the Lyapunov function is constructed by geometric argument.

\medskip

We illustrate our main results and the sufficient conditions with a number
of examples. Our first example shows that the stochastic equilibria exist
and are unique in
a class of incomplete financial market models, with heterogeneous
``exponential'' agents. Next, we construct a class of martingales on
differential manifolds with connections under fewer assumptions than
before. Then, we treat two non-zero sum stochastic games,
namely, a game with ``cooperation and hindrance'', and a risk-sensitive
stochastic differential game; we show that Nash equilibria exist in both.
Our final example focuses on a different aspect of our results and treats a
one-dimensional equation.

\subsection{Structure of the paper.} After this introduction, we describe
the setting and state our main theorem and various sufficient conditions for
its assumptions in section 2. Section 3 contains examples, while the proofs
are divided between two sections: section 4 deals with H\" older
boundedness and
contains the bulk of the partial-regularity and hole-filling
estimates, while all the other proofs are collected in section 5.

\subsection{Notation and conventions}
For a scalar function $u$, $D u$ denotes its ($\R^d$-valued) spatial gradient, interpreted as
a row vector, while $D^2 u$ denotes its (spatial) Hessian matrix.
Individual spatial partial derivatives of are denoted
by subscripts, i.e., $D_i u= \pds{x^i} u$ and $D_{ij} u$ stands for
$\pdm{x^i}{x^j} u$.
In the
vectorial ($\R^N$-valued) case - which we mark by bold symbols -  $D \bsu$ is
understood as the $\R^{N\times d}$-Jacobian matrix.

We will also have occasion to evaluate bilinear forms on
function gradients; for a $d\times d$ matrix $S$ we write $\scl{D u}{Dw}_S =
\sum_{ij} D_i u D_j w S_{ij} = Du S Dw^\top$. More generally, for a square
matrix $S$ and two matrices $A$ and $B$ of appropriate dimensions, we write
$\scl{A}{B}_S$ for the matrix $A S B^\top$.
The Frobenious product of matrices
is denoted by $A:B$, i.e., for square matrices $A$ and $B$ of the same
dimension, we define $A:B = \sum_{ij} A_{ij} B_{ij} = \Tr (A^\top B)$. The
Frobenius norm of a square matrix $A$ is given by $\sqrt{A:A}$.

The notation $\abs{\cdot}$  is used both for the Lebesgue measure of a subset
of $\R^d$, as well as for the Euclidean norm in any Euclidean space; it should
be interpreted as the Frobenius norm, in case its argument is a matrix. For
$(\tx) \in \R\times \R^d$, we use an anisotropic norm, namely, we set
$\abs{(\tx)} = \max(\sqrt{\abs{t}}, |\bsx|)$. The closed ball of radius $R$
around $\bsx$ in $\R^d$ is denoted by $B_R(\bsx)$. In the
special case when $\bsx=0$ and $R=n\in\N$, we use simply $B_n$.

The notation $\norm{\cdot}$ will be reserved for infinite-dimensional spaces. More
specifically, unless defined otherwise, $\norm{\cdot}$ stands for the
$\ltwo$-norm, both on the underlying probability space, and on an appropriate
domain.

For $r\in \N$ and a (generally matrix-valued) process $(\bsZ_u)_{u\in
[t,T]}$, we write $\bsZ\in \sP^r$ if
$\int_t^T \abs{\bsZ_u}^r\, du<\infty$, a.s. The stochastic
integral $\int \bsZ_s \, d W_s$ of $\bsZ\in\sP^2$ with
respect to a vector  Brownian motion
$\bsW$, defines a vector-valued process whose $i$-th component is given by
$\sum_j \int \bsZ^{ij}_s\, d W^j_s$. We write $\bsZ\in \bmo$ if
$\sup_{\tau\in \mathcal{T}} \norm{\EE_{\tau}[\int_\tau^T |\bsZ_u|^2
du]}_{\linf}<\infty$, where $\mathcal{T}$ is the set of $[0,T]$-valued
stopping times and $\EE_{\tau}$ denotes the conditional expectation $\EE[\cdot | \mathcal{F}_\tau]$ with respect to $\mathcal{F}_\tau$.
The notation $dF \eqm \alpha$
means $F - \int_0^\cdot \alpha_s\, ds$      is a local martingale. Standard
localization techniques and boundedness of processes involved can be used to show that all local martingales in the
sequel can be treated as martingales effectively, therefore we will treat them as such without
explicit mention.

For $\alpha \in (0,1]$, a compact subset $K$ of $\R^d$ and
a function $\bsv:\TR\to\R^N$, the H\" older seminorm $[\bsv]_{\alpha,K}$ is
defined by
\begin{align}
\label{equ:holder-semi}
[\bsv]_{\alpha,K} = \sup_{(\txp)\ne (\tx) \in [0,T]\times K}
\frac{\abs{\bsv(\txp) - \bsv(\tx)}}{\abs{(\tx) - (\txp)}^{\alpha}}.
\end{align}
Sequences are denoted by curly brackets $\set{\cdot}$. The index $n\in\N$
or $m\in\N$ is usually omitted and will always be clear from the context.

\section{Main results}

\subsection{The setup, standing assumptions and key concepts}

\subsubsection{The driving diffusion}
\label{sse:diff}
We work on a probability space $(\Omega, \mathcal{F}_T, \PP)$, on which a
$d$-dimensional Brownian motion $(\bsW_t)_{t\in [0,T]}$ is defined. With
$\FFF=(\sF_t)$ denoting the argumented
filtration generated by $\bsW$, we
consider the stochastic
differential equation
\begin{equation}\label{equ:X}
 d\bsX_t = \bsb(t, \bsX_t) dt + \bsig(t, \bsX_t) d\bsW_t,
 \end{equation}
 where
 \begin{enumerate}
  \item
  the drift vector $\bsb: \TR \to \R^d$ is
  uniformly
  bounded,
  \item the dispersion matrix $\bsig: \TR \to \R^{d\times d}$
  is symmetric and
  there exist a constant $\Lambda>0$ such that
  $\Ld \abs{z}^2 \geq \abs{z \bsig(t,\bsx)}^2 \geq \tfrac{1}{\Ld}\abs{z}^2$, for
  all $(t,\bsx)\in \TR$ and all $z\in \R^d$, and
  \item
      there exists a constant $L$ such that,
      for all $t\in[0,T], \bsx, \bsx'\in \R^d$, we have
      \[
       \abs{\bsb(t, \bsx)-\bsb(t, \bsx')} + \abs{\bsig(t, \bsx) - \bsig(t,
       \bsx')} \leq L\abs{\bsx-\bsx'}.
      \]
 \end{enumerate}
These conditions ensure, in particular, that for each $(\tx)\in\TR$, there
exists a unique strong solution $(\bsX^{\tx}_u)_{u \in [t,T]}$ of
\eqref{equ:X}, defined on $[t,T]$, such that $\bsX^{\tx}_t=\bsx$. For
notational reasons, we extend $\bsX^{\tx}$ by setting $\bsX^{\tx}_u =
\bsx$, for $u\in [0,t)$, and denote by $\PP^{\tx}$ its law on the canonical
space $C^d[0,T]$.

\subsubsection{Markovian and H\" olderian Solutions}
Given $b_0\in \R^d$ and a sequence $\seq{\alpha}$ in $(0,1]$,
a sequence $\{\bsv^m\}$ is said to be \define{bounded in} $\Caa_{loc, b_0}(\TR)$
if there exists a sequence $\seq{c}$
  of positive constants,
  such that, for all $m,n\in\N$,
\[ \norm{\bsv^m}_{C^{\alpha_n}([0,T]\times B_n(b_0))} =\norm{\bsv^m}_{\linf([0,T]\times B_n(b_0))} + [\bsv^m]_{\alpha_n,B_n(b_0)} \leq c_n.\]
We write $\bsv\in \Caa_{loc, b_0}$ if the constant sequence $\{\bsv\}$ is bounded
in $\Caa_{loc, b_0}$. If the sequence $\{c_n\}$ is uniform for all $b_0 \in \R^d$, we say $\bsv\in \Caa_{loc}$.
A completely analogous construction yields the family of local H\" older
spaces $\Ca_{loc, b_0}(\R^d)$ and $\Ca_{loc}$ over $\R^d$ instead of $\TR$.
Various spaces of continuously (non-fractionally) differentiable
functions are defined in the standard manner.
\begin{definition}[A Markovian solution to BSDE]\label{def:Mark-BSDE}
Given Borel functions $\bsf:\TR \times \R^N \times \R^{N\times d} \to \R^N$ and
$\bsg:\R^d \to \R^N$,
 a pair ($\bsv$, $\bsw$) of Borel functions with the domain $\TR$ and
 co-domains $\R^N$ and $\R^{N\times d}$, respectively,
 is a called a \define{Markovian solution} to
the \define{system}
\begin{align}
\label{equ:BSDE1}
  d\bsY_t = - \bsf(t, \bsX_t,\bsY_t, \bsZ_t)\, dt + \bsZ_t \bsig(t, \bsX_t)\, d\bsW_t,\quad \bsY_T=\bsg
  (\bsX_T),
\end{align}
of
\define{backward stochastic differential equations}
if, for all $(\tx)\in\TR$,
\begin{enumerate}
  \item
  $\bsY^{\tx} := \bsv(\cdot, \bsX^{\tx})$ is a continuous process,
$\bsZ^{\tx} := \bsw(\cdot, \bsX^{\tx})\in \sP^2$, and \newline
$\bsf(\cdot,\bsX^{\tx}, \bsY^{\tx}, \bsZ^{\tx})\in \sP^1$,
\item For all $t' \in [t,T]$, we have
\[
  \bsY^{\tx}_{t'}= \bsg(\bsX^{\tx}_T) + \int^T_{t'} \bsf(u, \bsX^{\tx}_u, \bsY^{\tx}_u, \bsZ^{\tx}_u) du -\int^T_{t'} \bsZ^{\tx}_u \bsig(u, \bsX^{\tx}_u) d\bsW_u, \text{ a.s.}
\]
\end{enumerate}
A Markovian solution $(\bsv, \bsw)$ to \eqref{equ:BSDE1} is said to be
\define{bounded} if $\bsv$ is bounded,  \define{continuous} if $\bsv$ is continuous,
\define{locally H\" olderian} if $\bsv\in \Caa_{loc, b_0}$, for some $b_0 \in \R^d$ and some sequence
$\seq{\alpha}$ in $(0,1]$, and a \define{$\bmo$-solution} if $\bsZ^{\tx}\in\bmo$ for all
$(\tx)\in\TR$.
\end{definition}

\begin{remark}
For Markovian BSDE, it is customary to consider the generator
$\tilde{\bsf}(t, \bsx, \bsy, \bsz \bsig)$ instead of our $\bsf(t,\bsx, \bsy, \bsz)$.
Due to our assumptions on $\bsig$, these are equivalent and we maintain the generator as $\bsf$
for notational convenience later on.
\end{remark}
\subsubsection{Lyapunov functions}
The key condition in our main result below concerns the existence of
sequence of functions which we term the Lyapunov functions.  We abbreviate $\bsa
= \bsig\bsig^{\top}$ and define
$\scl{\bsz}{\bsz}_{\bsa(\tx)}=\bsz \bsig(\tx) (\bsz
\bsig(\tx))^{\top}$.

\begin{definition}[Lyapunov functions]
\label{def:Lyapunov}
Let $\bsf:\TR\times \R^N \times \R^{N\times d} \to\R^N$ be a Borel function and let
$c>0$ be a constant. A pair $(h,k)$ of nonnegative functions, with
$h\in C^2(\R^N)$ and $k$ Borel,
is said to be a
\define{$c$-Lyapunov pair} for $\bsf$ if
$h(\boldsymbol{0})=0, Dh(\boldsymbol{0})=\boldsymbol{0}$, and
\begin{align}
        \label{equ:global-Lyapunov}
   \tot   D^2 h(\bsy) : \scl{\bsz}{\bsz}_{\bsa(\tx)} -
   Dh (\bsy) \bsf (t, \bsx, \bsy,
   \bsz) \geq
   \abs{\bsz}^2 - k(\tx),
\end{align}
for all $(t, \bsx, \bsy, \bsz)\in \TR \times \R^N \times
\R^{N\times d}$, with $\abs{\bsy}\leq c$. We write $(h,k) \in
\Ly(\bsf,c)$.

\smallskip

Given $b_0\in \R^d$ and a sequence $\seq{c}$ of positive constants,
a pair $(\seq{h}, \seq{k})$ of sequences of nonnegative functions, with $h_n\in
C^2(\R^N)$ and  $k_n$ Borel,
is called a \define{local $\seq{c}$-Lyapunov pair for
$\bsf$}, if
$h_n(\boldsymbol{0})=0, Dh_n(\boldsymbol{0})=\boldsymbol{0}$
and
\begin{align}
        \label{equ:Lyapunov}
   \tot   D^2 h_n(\bsy) : \scl{\bsz}{\bsz}_{\bsa(\tx)} -
   Dh_n (\bsy) \bsf (t, \bsx, \bsy,
   \bsz) \geq
   \abs{\bsz}^2 - k_n(\tx),
\end{align}
for all $n\in\N$, $(t, \bsx, \bsy, \bsz)\in \TBn(b_0) \times \R^N \times
\R^{N\times d}$, with $\abs{\bsy} \leq
c_n$. We write $(\seq{h}, \seq{k}) \in \lLy(\bsf, \seq{c})$.
\end{definition}
\begin{remark}
$\,$
\begin{enumerate}
\item
Suppose that the process $\bsY$ has a semimartingale decomposition as in
\eqref{equ:BSDE1} (i.e., solves the BSDE system) and satisfies the
bound
$\abs{\bsY}\leq c$. A function
$h$ for which \eqref{equ:global-Lyapunov} holds
has the property that
$h(\bsY_t)$ is a
semimartingale with the finite variation part dominating (in the
increasing order) the process $\int_0^{\cdot} \big( \abs{\bsZ_u}^2 -
k(u,\bsX_u)\big)\, du$. The function $k$ will often be constant, but certain
applications require more flexibility.
If one wants to deal with unbounded $\bsY$, a
layer of localization
- expressed through the dependence on $n$ and the sequence $\seq{c}$ in the
local version - is necessary.

\item
It is interesting to note that in the scalar case
($N=1$), and when the generator
$\bsf$ grows at most quadratically in $\bsz$, it is essentially sufficient
to look for  Lyapunov
pairs with $h(y) = \exp( \alpha y )$, for large enough
$\alpha$. As we shall see below, this no longer works in the vector case,
which
leads to nontrivial
constructions of Lyapunov pairs under specific structural conditions.

\item Let $(\bsv, \bsw)$ be a bounded solution to \eqref{equ:BSDE1} whose
generator $\bsf$ admits a $\norm{\bsv}_{\linf}$-Lyapunov pair $(h,k)$ with
$k$ bounded. Item (1), together with boundedness of $\bsv$ and $k$, implies
that $Z = \bsw(\cdot, X)\in \bmo$. Hence $(\bsv, \bsw)$ is a
$\bmo$-solution. 

\end{enumerate}
\end{remark}

\subsection{A uniform estimate}
The first main result of the paper, contained
in Theorem \ref{thm:abstract} below, provides an abstract stencil for a uniform estimate for a family of BSDE systems under several assumptions, most
notable of which is the existence of a
Lyapunov pair, uniform for all systems in the family.
Sufficient conditions for these assumptions and examples will be given
shortly.
\begin{theorem}[Uniform estimate]\label{thm:abstract}
Let $\{\bsf^m\}$ and $\{\bsg^m\}$ be sequences of
Borel functions $\bsf^m:\TR \times \R^N \times \R^{N\times d} \to \R^N$ and
$\bsg^m:\R^d \to \R^N$ such that, for each $m\in\N$,
 the BSDE system
\begin{align}
\label{equ:BSDEm}
  d\bsY_t = - \bsf^m(t, \bsX_t,\bsY_t, \bsZ_t)\, dt + \bsZ_t \bsig(t,
  \bsX_t)\, d\bsW_t,\quad \bsY_T=\bsg^m(\bsX_T),
\end{align}
admits a Markovian solution
$(\bsv^m,\bsw^m)$.

Suppose that there exist $b_0\in \R^d$ and  sequences $\seq{M}, \seq{c}$ in $[0,\infty)$, $\seq{\alpha}$
in $(0,1]$ and $\seq{q}$
with $q_n > 1+d/2$, such that
\begin{enumerate}
\item \label{asm:g} ($C^{\alpha}_{loc}$-regularity of the terminal condition)
The sequence $\{\bsg^m\}$ is bounded in $\Ca_{loc, b_0}$.
\item \label{asm:v} (A-priori continuity and local uniform boundedness) For
all $m,n\in\N$,
$\bsv^m$ is continuous on $\TR$ and
\[ \abs{\bsv^m(\tx)}\leq
c_n, \quad \text{for all $(\tx)\in \TBn(b_0)$.}\]
 \item \label{asm:q} (Local uniform quadratic growth)
 For each $n\in \N$, there exist functions $\{k^m_n\}$ such that
 \[
 \sup_{m\in\N} \norm{k^m_n}_{\el^{q_n}([0,T]\times B_n(b_0))} < \infty \quad \text{and} \quad
  \abs{\bsf^m(\tx,\bsy,\bsz)} \leq M_n \big( \abs{\bsz}^2 + k^m_n(\tx)
   \big),
 \]
 for all  $m\in\N$,  $(\tx) \in \TBn(b_0)$, $|\bsy|\leq c_n$,
 and $\bsz\in\R^{N\times d}$.
 \item \label{asm:h} (Local Lyapunov pair)
 There exist functions $\{h_n\}$ such that
  $(\seq{h},
 \seq{k^m})$ is a local $\seq{2c}$-Lyapunov pair for $\bsf^m$ for each $m\in\N$, i.e. $(\seq{h},
 \seq{k^m}) \in \lLy(\bsf^m, \seq{2c})$.
\end{enumerate}

\smallskip

  Then, the sequence $\{\bsv^m\}$ is bounded in $\Caap_{loc, b_0}$, for some
  $\seq{\alpha'}$ in $(0,1]$. Moreover, for each $n$, the
  H\"{o}lder seminorm $[\bsv^m]_{\alpha'_n, B_n(b_0)}$ depends only on $d, N, T, \Lambda, L, \norm{b}_{\linf}, \norm{\sigma}_{\linf}, \alpha_n, M_n, c_n$, $[\bsg^m]_{\alpha_n, B_n(b_0)}, h_n$, and $\sup_m\norm{k^m_n}_{\mathbb{L}^{q_n}([0,T]\times B_n(b_0))}$.
\end{theorem}

\begin{remark}\label{rem:a priori}
$\,$
\begin{enumerate}
\item 
The sequence $\seq{h}$ in condition (4) above is chosen uniformly for all $\{\bsf^m\}$. Therefore the inequality \eqref{equ:Lyapunov} is satisfied for all $\bsf^m$ and $k^m_n$. It is without loss of generality to have $|\bsz|^2$ on the right-hand side of \eqref{equ:Lyapunov}, since any positive constant $\delta_n$ before $|\bsz|^2$ can be normalized to $1$ after scaling $h_n$ and $k^m_n$ by $1/\delta_n$.
\item Applying Theorem \ref{thm:abstract} to a constant
 sequence (i.e. all $\{\bsf^m\}$ and $\{\bsg^m\}$ are the same for
 different $m$), we obtain an a priori estimate for a continuous Markovian solution of a single system: let $(\bsv, \bsw)$ be a continuous Markovian solution \eqref{equ:BSDE1} whose data $\bsf$ and $\bsg$ satisfy assumptions in Theorem \ref{thm:abstract}, then $\bsv$ is locally
 H\"{o}lderian, i.e., $\bsv\in \Caa_{loc, b_0}$. 
\end{enumerate}
\end{remark}

One of the advantages of our probabilistic approach is that
the uniform H\" older bound on $\{\bsv^m\}$ in Theorem \ref{thm:abstract}
is sufficient to establish the
existence result in Theorem \ref{thm:existence}
below. To make a connection 
with a typical analytic treatment of related PDEs, where regularity
and bounds of $\bsw$ need to be obtained, we provide some pertinent
information in the following remark.
\begin{remark}\label{rem:w reg}
$\,$
 \begin{enumerate}
 \item 
\label{rem:further-bounds}
Without structural conditions on $\bsf$, uniform $\linf$-bounds for \emph{systems} do
not always lead to gradient bounds, as evidenced by the following example
due to E.~Heinz.
Consider the following quadratic system of PDE:
\begin{align*}
   v^i_t - v^i_{xx} &= v^1 \Big( (v_x^1)^2 + (v_x^2)^2\Big), \quad i=1,2.
\end{align*}
For any $m\in\N$, $v^1=\cos(mx)$ and $v^2=\sin(mx)$ is a (stationary)
solution, but clearly, $\norm{\nabla v}_{\linf}=m$ cannot be controlled by
$\norm{v}_{\linf}=1$ (and a universal constant independent of $m$).
For a general system of the form \[  \partial_t \bsv - \tot \Delta \bsv +
\bsf(t,x,\bsv,\nabla \bsv)=0,\]
a local estimate of $\norm{\nabla \bsv}_{\linf}$ is established in \cite[Theorem
6.1]{LadSolUra67} in the case when
$f$ satisfies a condition of the form
\[\abs{f(t,x,v,p)}\leq \Big[\eps\abs{v} + P(\abs{p}, \abs{v}) \Big]
(1+\abs{p})^2\]
for some sufficiently small $\eps>0$ and $P(\abs{p},\abs{v}) \to 0$ as
  $\abs{p}\to\infty$.
When $f$ has at most linear growth in $p$, the same
local estimate is established in \cite{Del03} using a probabilistic
techniques. 

\item When $\bsv$ is H\"{o}lder continuous and there exists $k_n \in \mathbb{L}^q([0,T]\times B_n(b_0))$ with $q>1+d/2$ such that 
\[
 |\bsf(t,x,\bsy, \bsz)| \leq M_n \big(|z|^2 + k_n(t,x)\big),
\] 
for all $(t,x)\in [0,T]\times B_n(b_0)$, $\bsy\in \R^N$, and $\bsz\in
  \R^{N\times d}$, then \cite[Proposition 5.1]{Bensoussan-Frehse} used
  regularity theory of elliptic systems in \cite{Fre88} to show that
  $\bsv\in W^{2,1}_{q}([0,T]\times B_n(b_0))$. In particular, when $q>2+d$,
  Sobolev embedding Theorem (see \cite[Lemma 3.3]{LadSolUra67}) implies
  that $\bsw$, as the  weak Jacobian of $\bsv$, is H\"{o}lder continuous on
  $[0,T]\times B_n(b_0)$.
\end{enumerate}
\end{remark}

\subsection{Existence and uniqueness}

A direct consequence of the uniform estimate in Theorem \ref{thm:abstract}
is the existence of a Markovian solution to the system \eqref{equ:BSDE},
whose data $(\bsf, \bsg)$ are approximated by a sequence $\{\bsf^m,
\bsg^m\}$.

\begin{theorem}\label{thm:existence} (Existence by approximation)
 Let $\bsf:\TR\times \R^N \times \R^{N\times d} \to \R^N$ and
 $\bsg:\R^d\to\R$ be a pair of Borel functions. Assume that there
 exist sequences $\{\bsf^m\}$ and $\{\bsg^m\}$ which
 satisfy the assumptions of Theorem \ref{thm:abstract} and
 \begin{align}
\label{equ:approx-f-g}
 \lim_{m\to \infty} \bsf^m(t, \bsx, \bsy^m, \bsz^m)= \bsf(t, \bsx, \bsy, \bsz)\quad
\eand \quad \lim_{m\to\infty} \bsg^m(\bsx) = \bsg(\bsx),
\end{align}
for all
$(t, \bsx, \bsy, \bsz) \in \TR \times \R^N\times \R^{N\times d}$ and  all sequences $\bsy^m\rightarrow \bsy$ and $\bsz^m \rightarrow \bsz$.
 Then the system
\begin{align}
\label{equ:BSDE}
  d\bsY_t = - \bsf(t, \bsX_t, \bsY_t, \bsZ_t)\, dt + \bsZ_t \bsig(t,
  \bsX_t)\, d\bsW_t,\quad \bsY_T=\bsg(\bsX_T),
\end{align}
admits a locally H\"{o}lderian solution $(\bsv,\bsw)$ such that $\bsv$ is a locally uniform
limit of a subsequence of $\set{\bsv^m}$ in Theorem \ref{thm:abstract}, and
$\bsw$ is the weak Jacobian of $\bsv$ on $(0,T)\times \R^d$.
\end{theorem}

The solutions produced in Theorem
\ref{thm:existence} are not necessarily unique, even when the solutions to the
approximating equations are. Indeed, one only needs to consider the case
where $\bsf\equiv 0$ and
where $\{\bsg^m\}$ is a sequence of bounded and smooth approximations to the
function $\bsg$ appearing in Tychonov's non-uniqueness theorem (see
\cite[p.~171]{Joh78}) for the heat equation.

As we shall see below, these pathologies disappear under appropriate
conditions on  $\bsf$ and $\bsg$. When the H\"{o}lder norm of $\bsg$ does not depend on $b_0$, $\bsf$ does not depend on $\bsy$ and
satisfies additional regularity assumption in $\bsz$, uniqueness is recovered.
Two Markovian solutions, $(\bsv,\bsw)$ and
$(\bsv',\bsw')$, are considered equal if $\bsv(\tx)=  \bsv'(\tx)$ for all
$(\tx)\in \TR$ and
$\bsw=\bsw'$, a.e., with respect to the Lebesgue measure on $\TR$.

\begin{theorem}[Uniqueness] \label{thm:unique}
Suppose that 
\begin{enumerate}
 \item $\bsg\in \Caa_{loc}\cap \mathbb{L}^\infty$ for some sequence $\{\alpha_n\}$ in $(0,1]$;
 \item $\bsf$ is continuous, does not depend on $\bsy$, and there exists
$M \geq 0$ such that
  \[ 
  |\bsf(t, x, \bsz)| \leq M\big(1+ |\bsz|^2\big) \quad \text{and} \quad
  \abs{\bsf(t, \bsx, \bsz) - \bsf(t, \bsx, \bsz')} \leq M (\abs{\bsz} +
  \abs{\bsz'})\abs{\bsz-\bsz'},   
  \]
  for all $(t,\bsx) \in [0,T]\times \R^d, \bsz, \bsz'\in \R^{d\times N}$;
  \item There exists a (global) Lyapunov pair
$(h,k)\in\Ly(\bsf,c)$ with $k$ bounded for some $c>0$.
\end{enumerate}
Then \eqref{equ:BSDE1} admits at most one
continuous solution $(\bsv,\bsw)$ with
$\norm{\bsv}_{\linf}\leq c$.
\end{theorem}

\subsection{A sufficient condition for existence and uniqueness}
This section provides explicit conditions on the generator $\bsf$ and
the terminal condition $\bsg$ such that assumptions in Theorems
\ref{thm:abstract}, \ref{thm:existence}, and \ref{thm:unique} hold for \eqref{equ:BSDE1}. While the proof depends on
the abstract Theorem \ref{thm:abstract} above, we state it in a
self-contained form to make it more accessible to a reader interested in
its applications.

We start with a structural condition on the generator $\bsf$.
To the best of our knowledge, a version of it was first formulated in
\cite{Bensoussan-Frehse}. We present here a
generalization including   a subquadratic term; a further generalization
will be discussed in Remark \ref{rem:BF} below.
We interpret $\bsz\in\R^{N\times d}$
as an $N\times d$-matrix, and use $\bsz^j$ to denote its
$j$-th row, $j=\ft{1}{N}$. In the vector case the superscript $j$ denotes
the $j$-th component.
\begin{definition}[The Bensoussan-Frehse (BF) condition]
\label{def:BF}
 We say that a continuous function $\bsf:\TR\times \R^N \times \R^{N\times d} \to \R^N$
 satisfies the \define{condition (BF)} if
 it admits a decomposition of the form
 \begin{equation}\label{equ:BF}
 \bsf(t, \bsx, \bsy, \bsz) = \diag (\bsz \, \bsfl(t, \bsx, \bsy, \bsz)) + \bsfq(t,
  \bsx, \bsy, \bsz) + \bsfs(t, \bsx, \bsy, \bsz) + \bsfk(t, \bsx),
 \end{equation}
such that the functions
 $\bsfl:[0,T]\times \R^d \times\R^N \times  \R^{N\times d} \to \R^{d\times N}$  and
 $\bsfq, \bsfs, \bsfk:[0,T]\times \R^d \times \R^N \times \R^{N\times d} \to \R^N$
  have following property:
 there exist $b_0\in \R^d$ and two  sequences $\seq{C}$ and $\seq{q}$ of positive constants with
 $q_n>1+d/2$, and a sequence $\seq{\kappa}$ of functions $\kappa_n:[0,\infty)\to
 [0,\infty)$ with $\lim_{w\to\infty} \kappa_n(w)/w^2=0$ such that, for each
 $n\in\N$ and all
$(\tx, \bsy, \bsz)\in \TBn(b_0) \times \R^N \times \R^{N\times d}$  we have
 %where $\text{diag}(\cdot)$ is the diagonal of a matrix, $\bsfg, \bsfh,
 %\bsfl$ are functions from $[0,T]\times \R^d\times \R^{N\times d}$ to
 %$\R^{d\times n}$, $\R^n$, $\R^n$, respectively, $\bsfk$ maps $[0,T]\times
 %\R^d$ to $\R^n$, and they satisfy following conditions:
 \begin{align*}
& \abs{\bsfl(\tx, \bsy, \bsz)} \leq C_n(1+ \abs{\bsz}), && \text{ (quadratic-linear)} \\
&  \abs{\bsfq^i(\tx, \bsy, \bsz)} \leq C_n \big(1+ \sum_{j=1}^i \abs{\bsz^j}^2\big),\quad i=\ft{1}{N}, &&
  \text{ (quadratic-triangular) } \\
& \abs{\bsfs(\tx, \bsy, \bsz)} \leq \kappa_n( \abs{\bsz}),  &&
 \text{ (subquadratic) } \\
& \,\,\bsfk\in \mathbb{L}^{q_n}([0,T]\times B_n),
  && \text{ ($\bsz$-independent) }
\end{align*}
In that case, we write $\bsf\in \BFckq$
\end{definition}
The (BF) conditions are simple enough to be easily checked in applications,
but also strong enough to yield the following result which will play a
major role in the existence theorem below:
\begin{proposition}[Existence of Lyapunov pairs under condition (BF)]
\label{thm:BF-Lyap}
 Let $\seq{c}$ is
 an arbitrary sequence of positive constants, and $\bsf$ a function in
 $\BFckq$.
 \begin{enumerate}
     \item There exists a local $\seq{c}$-Lyapunov pair $(\{h_n\}, \{k_n\})$ for
 $\bsf$. Furthermore, the same pair $(\{h_n\},\{k_n\})$ is a local
 $\seq{c}$-Lyapunov pair for any other function $\bsf'\in\BFckq$.
\item
 If, additionally, the sequences $\seq{C}$, $\seq{q}$ and $\{\kappa_n\}$ are
 constant (in $n$), then, for each $c$, a (global) $c$-Lyapunov
 pair for $\bsf$ exists.
 \end{enumerate}
\end{proposition}
Another ingredient necessary to
guarantee the existence of a solution to \eqref{equ:BSDE1} is a-priori
boundedness. We remind the
reader that a set of non-zero vectors $\bssa_1, \dots, \bssa_K$ in $\R^N$ (with $K > N$) is said to
\define{positively span} $\R^N$, if, for each $\bssa\in \R^N$ there exist
nonnegative constants $\ld_1,\dots, \ld_K$ such that
\[ \ld_1 \bssa_1 + \dots + \ld_K \bssa_K = \bssa.\]
The following two well-known characterization (see \cite{Dav54}), presented
here for reader's convenience, make
positively-spanning sets easy to spot: (1)
Non-zero vectors $\bssa_1,\dots, \bssa_K$ positively span $\R^N$ if for every
$\bssa\in \R^N\setminus\{\boldsymbol{0}\}$ there exists $k\in \set{1,\dots, K}$ such that
$\bssa^{\top} \bssa_k>0$. (2) If non-zero vectors $\bssa_1,\dots, \bssa_K$ already
span $\R^N$, then they positively span $\R^N$ if $\boldsymbol{0}$ admits a
nontrivial
positive representation, i.e., if there exist nonnegative $\ld_1,\dots,
\ld_K$, not all $0$, such that $\ld_1\bssa_1+\dots+ \ld_K \bssa_K=0$.
\begin{definition}[The a-priori boundedness (AB) condition]
\label{def:AB}
We say that $\bsf$ satisfies the \define{condition (AB)} if
 there exist a deterministic function $l\in
 \lone
 [0,T]$,
 and a set $\bssa_1,\dots, \bssa_K$
 which positively spans $\R^N$, such that
            \begin{align}
            \label{equ:AB}
            \bssa_k^{\top} \bsf (\tx, \bsy, \bsz) \leq   l(t) +  \tot
            \sabs{\bssa_k^{\top} \bsz}^2, \quad
            \text{ for all $(\tx, \bsy, \bsz)$ and }
            k=1,\dots, K.
\end{align}
We say that $\bsf$ satisfies the \define{weak condition (AB)} - abbreviated
as \define{(wAB)} - if there exist Borel functions $\bsL_k:\TR\times \R^{N\times d}
\to \R^{d}$, for $k=1, \dots, K$,  such that $|\bsL_k(t,\bsx, \bsz)|\leq C (1+ |\bsz|)$ for some constant $C$ and
 \begin{align}
            \label{equ:wAB}
            \bssa_k^{\top} \bsf (\tx,\bsy, \bsz) \leq   l(t) +  \tot
            \sabs{\bssa_k^{\top} \bsz}^2+ \bssa^\top_k \bsz \bsL_k(t,\bsx, \bsz), \quad
            \text{ for all $(\tx,\bsy,\bsz)$ and }
            k=1,\dots, K.
\end{align}
\end{definition}
\begin{remark}\label{rem:wAB}
  The constant $\tot$ in \eqref{equ:AB} is simply a convenient choice for later
  use; it can easily be replaced by any other constant by scaling.
  Furthermore, conditions (AB) and (wAB) are invariant under invertible linear transformation of $\R^N$. More precisely,
  suppose that $\bsf$ satisfies $(wAB)$ with $l \in \lone[0,T]$, the
  positively-spanning set
  $\bssa_1,\dots, \bssa_K$ and the functions $\{\bsL_k\}$,  and that
  $\Sigma:\R^N\to\R^N$ is an invertible linear
  map. Then the generator of the transformed system, namely $\tilde{\bsf}(\tx, \tilde{\bsy},
  \tilde{\bsz}):= \Sigma \bsf (\tx, \bsy, \bsz)$,
  satisfies $(wAB)$ with the same $l$,  $\tilde{\bsL}_k(t,\bsx,\tilde{\bsy}, \tilde{\bsz}) = \bsL_k(t,\bsx,\bsy, \bsz)$, and  transformed (but still
  positively spanning) set $(\Sigma^{-1})^{\top} \bssa_k$, $k=1,\dots, K$.
\end{remark}

\begin{theorem}\label{thm:sufficient}
$\,$

\noindent(Existence under (BF)$+$(AB))
 Suppose that
 $\bsf$ satisfies
 conditions (BF) and (AB), and that
 $\bsg\in \Ca_{loc, b_0}$ for some $b_0$ and it satisfies
 $\lim_{\abs{\bsx}\to\infty} \abs{\bsg(\bsx)}/ \abs{\bsx}^2 = 0$.
 Then the system \eqref{equ:BSDE1} admits a locally H\"{o}lderian solution
 $(\bsv, \bsw)$, i.e., $\bsv\in \Caa_{loc, b_0}$ for some sequence $\{\alpha'_n\}$ in $(0,1]$. When $\bsg$ is bounded, the condition (AB) can be replaced by $(wAB)$ and $(\bsv, \bsw)$ is a bounded $\bmo$-solution. 
 
 \medskip

 \noindent (Uniqueness under (BF)$+$(wAB)) Suppose that 
 \begin{enumerate} 
 \item $\bsg\in \Caa_{loc}\cap \linf$ for some sequence $\{\alpha_n\}\in (0,1]$;
 \item (wAB) is satisfied, and (BF) is satisfied with the constants $\seq{C}$ and functions $\seq{\kappa}$ independent of $n$;
 \item $\bsf$ does not depend on $\bsy$, $\bsf(\cdot,\cdot,0)$ is bounded, and there exists a constant $M$ such that $\abs{\bsf(t, \bsx, \bsz) - \bsf(t, \bsx, \bsz')} \leq M (\abs{\bsz} +
  \abs{\bsz'})\abs{\bsz-\bsz'}$ for all $t\in [0,T]$, $\bsx\in \R^d$, $\bsz,\bsz'\in \R^{N\times d}$.
 \end{enumerate}
 Then the solution $(\bsv,\bsw)$ is unique in the class of bounded continuous solutions.
 \end{theorem}

\begin{remark} \label{rem:BF} Here are two extensions of Theorem \ref{thm:sufficient}
which, for the sake of simplicity of presentation, we did not put into its
statement. They will be proved, however, along with Theorem
\ref{thm:sufficient}, below.
\begin{enumerate}
\item
When $\bsg$ is bounded,
the conclusions of Theorem \ref{thm:sufficient} hold if the equality in \eqref{equ:BF} holds only
approximately, namely if, for each $n\in \N$, there exists a sufficiently small $\eps_n$ such that
\begin{align}
\label{equ:BF-approx}
 \babs{
  \bsf(t, \bsx, \bsy, \bsz) - \text{diag}(\bsz \bsfl(t, \bsx, \bsy, \bsz)) - \bsfq(t,
  \bsx,  \bsy, \bsz) - \bsfs(t, \bsx, \bsy, \bsz)  - \bsfk(t,
  \bsx)}\leq \eps_n \abs{\bsz}^2
\end{align}
holds on $[0,T]\times B_n(b_0)\times \R^N\times \R^{N\times d}$.
How small this $\eps_n$ needs to be
depends on the constants $\seq{C}$ in the condition (BF), on
$\norm{\bsv}_{\linf}$
(which, in turn, depends on $\norm{\bsg}_{\linf}$ and the functions and constants
appearing in condition (wAB)), as well as the universal constants ($\Lambda, T, d,
N$, etc.). In general, it is possible to obtain an explicit expression for an estimate
of $\eps_n$ by keeping track of the explicit values of the constants
involved in the proof, but we do not pursue that here.
The case in which such an explicit expression may prove to be useful is when $\bsfl=\boldsymbol{0}$, $\bsfq=\bsfs=\boldsymbol{0}$ and $\bsfk=\boldsymbol{0}$, i.e.
when $\bsf$ is of general structure, but satisfies a smallness assumption.
This case
allows for an especially simple treatment; indeed, to construct a global
Lyapunov pair, it suffices to pick
\[ h(\bsy) = \tot \abs{\bsy}^2, k \equiv 0,  \text{ so that
 $D^2 h = I_{d}$ and $Dh(\bsy) =  \bsy$.}\]
 Then,
\[ \tot D^2 h : \scl{\bsz}{\bsz}_{\bsa}  = \tot \sum_{i=1}^D \abs{\bsz^i
\sigma}^2 \geq \tot \Ld^{-1} \sum_{i=1}^D \abs{\bsz^i}^2, \quad \text{ but }
\quad
 Dh \bsf \leq \eps \abs{\bsy} \abs{\bsz}^2.\]
Therefore, it suffices to
require $\eps < (4\Ld \norm{\bsv}_{\linf})^{-1}$ so that $(h,0)\in \Ly(\bsf, 2 \|\bsv\|_{\linf})$.
This recovers the situation in \cite{Str81} where solutions to parabolic
systems of PDEs were constructed under a parallel ``smallness'' condition.

\item
 Suppose that some component of $\bsg(X_T)$,
say the $j$-th, has a bounded
Malliavin derivative and $\bsf^j$ does not depend on
$\bsz^i$ for $i\neq j$. It is known  then
(see
\cite{Briand-Elie} and \cite{Cheridito-Nam}  for sufficient conditions on
$\bsX, f^j$ and $g^j$) that
 the $j$-th component $\bsZ^j$ of the
solution is bounded, too.
 In this case, Theorem \ref{thm:sufficient} still holds if  any locally bounded
function of $\bsz^j$ is added to the right-hand side of
\eqref{equ:AB} and \eqref{equ:wAB}.
\end{enumerate}
\end{remark}

\section{Examples}\label{sec:examples}
We illustrate the strength of our results by considering four different
classes of BSDE systems arising from game theory, geometry, mathematical
economics and mathematical finance. Proofs of all statements are postponed
until section 5.

\subsection{Incomplete stochastic equilibria}

The existence and properties of equilibrium (market-clearing) asset-price
dynamics in financial markets is one of the central problems in financial
economics and mathematical finance. While the so-called complete market case has
been fully understood, the incomplete market case has been
open since early 1990s.
A stochastic equilibrium among $N$ heterogeneous agents in
incomplete markets has been considered in \cite{Kardaras-Xing-Zitkovic}.
There the filtration is generated by a $2$-dimensional Brownian motion
$\bsW= (B, B^\bot)$, where the first component drives the price of a
tradable asset but both components can determine the size of agents' random
endowment. Preference of agents are modeled by exponential utilities with
heterogenous risk-tolerance coefficients. An equilibrium is a pair
consisting of an asset-price process and
agents' trading strategies such that every agent maximizes the expected
utility from trading and random endowment, meanwhile supply equals to
demand (market clears), cf. \cite[Definition 1.1]{Kardaras-Xing-Zitkovic}.

In this setting, \cite{Kardaras-Xing-Zitkovic} considered the following system of quadratic BSDE:
\begin{equation}\label{equ:equilibrium}
 d\bsY_t = \bmu_t dB_t + \bnu_t dB^\bot_t + \big(\tfrac{1}{2} \bnu_t^2 - \tfrac{1}{2} \bsA[\bmu_t]^2 + \bsA[\bmu_t] \bmu_t \big)\, dt, \quad \bsY_T = \bsG,
\end{equation}
where $\bsA[\bmu] = \sum_{i=1}^N \alpha^i \mu^i$ for a sequence of
constants $(\alpha^i)$ with $\alpha^i\in (0,1)$ and $\sum_{i=1}^N
\alpha^i=1$.
It is proved in \cite[Theorem 1.6]{Kardaras-Xing-Zitkovic} that equilibria one-to-one correspond to solutions of \eqref{equ:equilibrium} with $(\bmu,
\bnu)\in \bmo$. Moreover in an equilibrium, each component of $\bsY$ represents the certainty equivalence of each agent. However, when it comes to the existence and uniqueness of
solutions, certain ``smallness-type" of conditions need to be assumed;
either $\norm{\bsG}_{\linf}$ is sufficiently small or $T$ is sufficiently
small, cf. \cite[Corollaries 2.6 and 2.7]{Kardaras-Xing-Zitkovic}. In the
Markovian setting, existence of solutions was also established for
sufficiently small $T$ in
\cite{Zit12} in a similar model, and \cite{Zha12} and
\cite[Theorem 3.1]{ChoLar14}.

The following result establishes global existence and uniqueness of equilibrium in a Markovian setting with bounded random endowment.
Here, $X$ is the solution of \eqref{equ:X} with $\bsW=(B, B^\bot)$,
$b: [0,T]\times \R^d \rightarrow \R^d$, and $\sigma: [0,T]\times \R^d
\rightarrow \R^{d\times 2}$ satisfying conditions (1)-(3) after
\eqref{equ:X}.

\begin{theorem}[Existence and uniqueness of incomplete stochastic equilibria]\label{thm:equilibrium} Suppose that the terminal condition
is of the form $\bsG=\bsg(X_T)$ for some $\bsg\in \Caa_{loc}
\cap \linf$. Then the system \eqref{equ:equilibrium} admits a unique bounded continuous
solution. Consequently, an incomplete stochastic
equilibrium in the setting of \cite{Kardaras-Xing-Zitkovic} exists and
is unique in the class of equilibria in which
each agent's certainty-equivalence process is a continuous function of
time and the state $X$.
\end{theorem}
\begin{remark}
 When $\bsg$ is of merely subquadratic growth, the system
 \eqref{equ:equilibrium} still admits a locally H\"{o}lderian solution, but
 the martingale part associated to this solution may not have enough
 integrability to be identified with an equilibrium.
\end{remark}

\subsection{Martingales on manifolds}\label{subsec:mart}
It is well-known that semimartingales can be defined on arbitrary
differentiable manifolds, but that martingales require additional
structure, namely that of a connection (if one wants a Brownian
motion, one needs a full Riemannian metric). We refer the reader to the
books \cite{Eme89} and \cite{Hsu02} for more details.

In the flat (Euclidean) case, martingales are
easily constructed from their terminal values by a simple process of
filtering, i.e.,  computing conditional expectation. When the underlying
filtration is Brownian, one can,
additionally, build this martingale from the given Brownian motion via the
martingale representation theorem; this
amounts to a solution to a linear system of BSDE.

If the geometry is not flat, one cannot simply filter anymore, but, as it
turns out, the problem can still be formulated in terms of a system of
BSDE. This system, however, is no longer linear and the existence of its
solution has been a subject of extensive study (see, e.g.,
\cite{Darling}, \cite{Bla05} and \cite{Bla06}).

Before we write down this system, we set the stage by assuming that a
$d$-dimensional Brownian motion $W$ is given,
and that the target space is an $N$-dimensional
differentiable manifold $M$, without boundary,
endowed with an affine connection. This connection,  $\Gamma$, is  described in coordinates by its
Christoffel symbols $\Gamma^k_{ij}$; we assume these are all Lipschitz on
compact sets, but not necessarily differentiable (as we will not be needing
the concept of curvature).

The martingale property on a manifold with a connection $\Gamma$ can be
formulated in many ways - we prefer to give the one that resembles a
characterization in the flat case; we say that a continuous $M$-valued
semimartingale $\bsY$ is a
\define{$\Gamma$-martingale} (with respect to the natural
filtration of $W$) if
\[ f(\bsY_t) - \tot \int_0^t  \Hess f(d\bsY_t, d\bsY_t), t\in [0,T],\]
is a local martingale for each smooth real-valued $f$. Here $\Hess f$ is the
(covariant) Hessian of $f$, i.e.,  a $(0,2)$-tensor, given in our coordinate chart by
\[ (\Hess f)_{ij}(\bsy)   = D_{ij} f(\bsy) -  \Gamma^k_{ij}(\bsy) D_k
f(\bsy).\] We refer the reader to \cite[p.~23]{Eme89} for the definition of
quadratic variation with respect to a $(0,2)$-tensor field (such as $\Hess$) on
a manifold.  \itos{} formula immediately implies that $\bsY$ is a
$\Gamma$-martingale if its coordinate representation
admits the following semimartingale decomposition
\begin{equation}
\label{equ:BSDE-Darling}
dY^k_t = - \tot \sum_{i,j=1}^d \Gamma^k_{ij}(\bsY_t) (\bsZ^i_t)^{\top} \bsZ^j_t\, dt +
\bsZ^k_t\, dW_t,\ k=1,\dots, N,
\end{equation}
where, as usual, $\bsZ^j$ denotes the $j$-th row of the $N\times
d$-matrix-valued process $\bsZ$.

For simplicity, and without too great a loss of generality, we assume
that the given terminal value $\bsG$ of the  martingale we want to construct is
of the form $\bsg(W_T)$. Furthermore, we assume that the image of $\bsg$ is
localized in the following way: there exists a convex and  compact set $M_0$, covered by
the image
$V \subseteq M$ of a single chart,
with coordinates $\bsy=(y_1,\dots, y_N)$, such that $\bsg(x) \in M_0$, for
all $x\in R^d$.
This way, we can work in a single coordinate chart, as if
$M$ itself were an open set of $\R^N$ and, in fact, assume that $M=\R^N$.
Also since we only care about the connection in a neighborhood of $M_0$, we
assume that the Christoffel symbols are globally Lipschitz.

As in \cite{Darling}, we make the following assumption on the geometry of $M$
around the image of $\bsg$:
\begin{assumption}[Double convexity]
\label{asm:geometry}
 There exists a convex function $\phi \in C^2(\R^N)$
 such that
 \begin{enumerate}
     \item $M_0=\phi^{-1}((-\infty,0])$, and
     \item $\Hess \phi$ is nonnegative definite ($\phi$ is
     geodesically convex), and
     strictly positive definite on some neigborhood of $M_0$.
 \end{enumerate}
\end{assumption}

Applying Theorem \ref{thm:existence} to the current setting, we obtain the following result.
\begin{proposition}\label{pro:Darling}
If $\bsg\in \Ca_{loc}(\R^d)$ and
Assumption \ref{asm:geometry} holds, there exists a $\Gamma$-martingale
$\prf{\bsY_t}$ with $\bsY_T= \bsg(W_T)$ which takes values in
$M_0$, for all $t\in [0,T]$.
\end{proposition}

\begin{remark}\
\begin{enumerate}
\item
While the detailed proof of Proposition \ref{pro:Darling} above is
postponed until Section 5, we comment, briefly, on the interpretation of
Lyapunov pairs in this, special, case. What makes it especially convenient
is the fact that
the driver $\bsf$ depends on $\bsZ$ only
through the symmetric matrix $\bsZ^{\top} \bsZ$. A simple computation shows
that
$(h,0)$ is a $c$-Lyapunov pair if (and only if) the matrix
\[ \big( D^2_{ij} h(\bsy)  - \sum_k D_k h(\bsy)\,  \Gamma^k_{ij}(\bsy)
\big)_{ij},\ i,j =1,\dots, d\]
is strictly positive definite for all $\abs{\bsy}\leq c$. Equivalently,
$\Hess h (\bsy) \succeq 0$, i.e., $h$ is (geodesically)
strictly convex (see, e.g.,
Chapter 3 of \cite{Udr94} for a detailed discussion of convexity on
Riemannian manifolds). This characterization fits perfectly with our
interpretation of Lyapunov functions as ``submartingale'' functions.
\item Unlike in the flat case, where convex functions abound, the very
existence of (geodesically) convex functions depends on geometric properties on $M$. We do not go into
details, but note that smooth nontrivial global convex functions always exist on
complete, simple-connected Riemanninan
manifolds of nonpositive sectional curvature (Cartan-Hadamard manifolds);
cf. \cite{Kendall}. In the general case, one
can always find a convex function locally, but it is not hard to see that
compact Riemannian manifolds, e.g., never admit nonconstant global convex functions. We refer the
reader to \cite{Udr94} for a thorough treatment of geodesic convexity.
\item The condition of
double convexity has been imposed in \cite{Darling} to
construct not-necessarily-Markovian martingales with values in manifolds
with connections. Our construction not only recovers some of the results from
\cite{Darling} in the
Markovian case, but also gives a partial positive answer to
Conjecture 7.2., p.~1257.  Indeed, Proposition \ref{pro:Darling} does not
require $M_0$ to have \emph{doubly convex geometry}, i.e., does not
assume conditions 5.1 and 6.1 in \cite[Theorem 7.1]{Darling}.
\item Without too much work, Proposition \ref{pro:Darling} can be extended
in several directions. First, the flat
Brownian motion $W$ can be replaced by a Brownian motion on a Riemannian
manifold (with metric $g$); indeed, one simply needs
to solve a modified version of BSDE \ref{equ:BSDE-Darling}
driven by a driftless diffusion whose dispersion coefficient $\sigma$
relates to the underlying metric as $\sigma \sigma^{\top} = g^{-1}$.
Moreover, under appropriate growth conditions,
the linear connection $\Gamma$ can be replaced by a nonlinear one, of the
form $\Gamma=\Gamma(\bsy, \bsZ^{\top} \bsZ)$.
\item Seen as a map between manifolds, a function  which transforms a
Browninan motion into a martingale is called {harmonic}. Our BSDE
\eqref{equ:BSDE-Darling}
corresponds to the parabolic system introduced by \cite{EelSam64} in order
to show that, under certain geometric conditions, harmonic maps exist within each
homotopy class (see \cite[Chapter 10]{Aub98a} for a detailed treatment of
this fascinating problem).
\end{enumerate}
\end{remark}

\subsection{A Stochastic Game with Cooperation or Hinderance}
Our next example concerns a finite-horizon stochastic differential game which is inspired
by a bounded-domain discounted game treated in
\cite{Bensoussan-Frehse-Nplayer}. For simplicity of presentation, we assume
there are only 2 players whom we call Player $1$ and Player $2$.
On a $d$-dimensional Brownian filtration, these players choose two
$\R^d$-valued processes, namely  $\bmu$ and $\bnu$ in $\bmo$, as their
respective controls. These
affect the state $\bsX$ through its drift in the following way:
\[
 d\bsX^{(\mu, \nu)}_t = \big(\bsb(\bsX^{(\mu, \nu)}_t) + \bmu_t +
 \bnu_t\big) dt +
 d\bsW^{(\mu, \nu)}_t, \quad \bsX^{(\mu, \nu)}_0=\bsx,
\]
where $\bsb: \R^d \rightarrow \R^d$ is a bounded Lipschitz vector field, and
$\bsW^{(\mu, \nu)} = \bsW - \int_0^\cdot(\bmu_u+\bnu_u) du$ is a Brownian
motion under the probability measure $\PP^{(\mu, \nu)}$ defined via
$d\,\PP^{(\mu, \nu)}/d\,\PP = \mathcal(\int (\bmu_u+ \bnu_u)^\top
d\bsW_u)_T$. Given a constant $\theta$ - which we term the
\define{cooperation penalty} - and integrable-enough functions $h^i, g^i: \R^d
\rightarrow \R$, $i=1,2$, the cost of player $i$ with the initial state
$\bsx$ at time $t=0$ is defined as
\[
 J^i(0, \bsx, \mu, \nu) = \EE^{(\bsx,\mu, \nu)} \Big[\int_0^T
 \big(h^i(\bsX_u) + \tfrac12 |\bmu_u|^2 + \theta \bmu_u^\top \bnu_u\big)
 dt + g^i(\bsX_T)\Big],\quad i=1,2,
\]
where the expectation is taken with respect to $\PP^{(\bsx, \mu, \nu)}$.
It is clear from its form how large positive values of the parameter $\theta$
incentivize the players to push in opposing directions, while the large
negative values motivate them to cooperate.
A \define{Nash equilibrium} between these two players is a
pair $(\hat{\bmu}, \hat{\bnu})$ of controls with the property that, for any $\bmu,
\bnu\in \bmo$, we have
\begin{equation}\label{equ:Nash-equ}
 J^1(0, \bsx, \hat{\bmu}, \hat{\bnu}) \leq J^1(0, \bsx, \bmu, \hat{\bnu})
 \quad \text{and} \quad J^2(0, \bsx, \hat{\bmu}, \hat{\bnu}) \leq J^2(0,
 \bsx, \hat{\bmu}, \bnu),
\end{equation}
and $(J^1, J^2)(\cdot, \cdot, \hat{\bmu}, \hat{\bnu})$ is called the value of this equilibrium.

We recast the problem as a BSDE system by introducing the
Lagrangians of the two players:
\[
 L^1(\bmu, \bnu, \bsp) = \tfrac12 |\bmu|^2 + \theta \bmu \cdot \bnu +
 \bsp^1 (\bmu+ \bnu) \quad \text{and} \quad L^2(\bmu, \bnu, \bsp) =
 \tfrac12 |\bnu|^2 + \theta \bmu \cdot \bnu + \bsp^2 (\bmu+ \bnu),
\]
where $\bsp^i$ is $i$-th row vector of $\bsp$.
When $\theta\neq \pm 1$, the minimizers are given by
\begin{equation}\label{equ:munu*}
 \hat{\bmu}(\bsp)= \tfrac{\theta}{(1+\theta)(1-\theta)}(\bsp^1 + \bsp^2) -
 \tfrac{1}{1-\theta} \bsp^1 \quad \text{and} \quad \hat{\bnu}(\bsp) =
 \tfrac{\theta}{(1+\theta)(1-\theta)}(\bsp^1 + \bsp^2) -
 \tfrac{1}{1-\theta} \bsp^2.
\end{equation}
Setting $L^i(\bsp) = L^i(\hat{\bmu}(\bsp), \hat{\bnu}(\bsp), \bsp)$ and $\bsL =
(L^1, L^2)^{\top}$, we pose the following BSDE:
\begin{align}\label{equ:BSDE-BF}
 d\bsY_t = -\bsf(\bsX_t, \bsZ_t) dt + \bsZ_t d\bsW_t,& \quad \bsY_T = \bsg(\bsX_T),
\ewhere
\bsf(\bsx, \bsz) = \bsh(\bsx) + \bsL(\bsz),
\end{align}
 with the state process given by $d\bsX_t = \bsb(\bsX_t) dt + d\bsW_t$ on the (augmented) filtration generated by $\bsW$. The following result establishes a unique bounded H\"{o}lderian solution, which corresponds to a Nash equilibrium.

\begin{proposition}\label{pro:BF-game}
 Assume that $\bsh \in \linf$ and
 $\bsg\in\Caa_{loc} \cap \linf$. When $-1\neq \theta \leq 1/2$ or
 $\theta>1$, the equation
 \eqref{equ:BSDE-BF} admits a unique bounded continuous solution
 $(\bsv, \bsw)$. Moreover the pair $(\hat{\bmu}(\bsZ), \hat{\bnu}(\bsZ))$,
 where $\bsZ=\bsw(\cdot, \bsX)$, is in $\bmo$ and enacts a Nash equilibrium
 with the value $\bsv$.
\end{proposition}

\subsection{Risk-sensitive nonzero-sum stochastic games}
Next, we consider a risk-sensitive stochastic game between $2$ players
studied in \cite{El-Karoui-Hamadene}. Let $U$ and $V$ be two compact metric
spaces. Player $1$ (resp. player $2$) chooses a $U$-valued (resp.
$V$-valued) control process $\bmu$ (resp. $\bnu$), which affects the state
$X$ in the following way:
\[
 dX^{(\mu, \nu)}_t = b(t,X^{(\mu, \nu)}_t, \bmu_t, \bnu_t) \, dt +
 \sigma(t,X^{\mu, \nu}_t) \, d\bsW^{(\mu, \nu)}_t,
\]
where $b: [0,T]\times \R^d\times U\times V\rightarrow \R^d$ is a bounded
measurable vector field, $\sigma$ satisfies conditions (2) and (3) after
\eqref{equ:X}, and $X$ is understood as the unique weak solution of the
previous stochastic differential equation. Given measurable functions $h^i:
[0,T]\times \R^d\times U\times V\rightarrow \R$ and $g^i: \R^d \rightarrow
\R$ with enough integrability, the cost of player $i$ with the initial
state $x$ at time $t=0$ is defined as
\[
 J^i(0, x, \bmu, \bnu)= \EE^{(x,\mu,\nu)}\Big[\exp\Big(\int_0^T h^i(u, X_u,
 \bmu_u, \bnu_u)\, du + g^i(X_T)\Big)\Big], \quad i=1,2.
\]
The problem is to find a Nash equilibrium $(\hat{\bmu}, \hat{\bnu})$
satisfying \eqref{equ:Nash-equ}.
To solve it, we define the Hamiltonian function
\begin{align*}
 &H^1(t,x, \bmu, \bnu) = \bmu \sigma^{-1}(t,x) b(t,x,\bmu, \bnu) + h^1(t,
 x, \bmu, \bnu) \quad \text{and}\\
 &H^2(t,x,\bmu, \bnu) = \bnu \sigma^{-1}(t,x) b(t,x,\bmu, \bnu) + h^2(t, x,
 \bmu, \bnu),
\end{align*}
and assume the \emph{generalized Issac's condition} holds, i.e., there
exists two measurable functions $\hat{\bmu}(t,x, \bsz)$ and $\hat{\bnu}(t,
x, \bsz)$ such that
\begin{align*}
 &H^1(t,x, \bsz^1, \hat{\bmu}(t,x,\bsz), \hat{\bnu}(t,x, \bsz)) \leq
 H^1(t,x, \bsz^1, \bmu, \hat{\bnu}(t,x, \bsz)) \quad \text{and}\\
 &H^2(t,x,\bsz^2, \hat{\bmu}(t,x,\bsz), \hat{\bnu}(t,x, \bsz))\leq
 H^2(t,x,\bsz^2, \hat{\bmu}(t,x, \bsz), \bnu),
\end{align*}
for any $(t,x,\bsz, \bmu, \bnu)\in [0,T]\times \R^d\times \R^{2\times d} \times U\times V$. Denote $\hat{H}^1(t,x,\bsz) = H^1(t,x, \bsz^1, \hat{\bmu}(t,x,\bsz), \hat{\bnu}(t,x, \bsz))$ and $\hat{H}^2(t,x,\bsz) = H^2(t,x,\bsz^2, \hat{\bmu}(t,x,\bsz), \hat{\bnu}(t,x, \bsz))$. We consider the following system of BSDE:
\begin{equation}\label{equ:EH-game}
 dY^i_t = - \big(\hat{H}^i(t, \bsX_t, \bsZ_t) + \tfrac12 |\bsZ^i_t|^2\big) \,
 dt + \bsZ^i_t d\bsW_t, \quad Y^i_T = g^i(\bsX_T), \, i=1,2.
\end{equation}

\begin{proposition}\label{pro:EH-game}
 Assume that $g^i\in \Caa_{loc}\cap \linf$, $h^i\in \linf$, and  $\hat{H}^i$ is continuous, for $i=1,2$. Moreover $b$ has at most linear growth in $(\bmu, \bnu)$, and $(\hat{\bmu}, \hat{\bnu})$ has at most linear growth in $\bsz$, both uniformly in $(t,x)$. Then \eqref{equ:EH-game} admits a unique bounded continuous solution $(\bsv, \bsw)$. Moreover $(\hat{\bmu}(\cdot, \cdot, \bsw), \hat{\bnu}(\cdot, \cdot,\bsw))$ is a Nash equilibrium with value $(\exp(v^1), \exp(v^2))$.
\end{proposition}

\begin{remark}
A solution to \eqref{equ:EH-game} was constructed in
\cite[Theorem 5.3]{El-Karoui-Hamadene} in the case of a bounded $b$ and a bounded (but not
necessarily continuous) terminal condition $\bsg$. When $\bsg$ is locally H\"{o}lder, our result shows that the solution is also locally H\"{o}lder (cf. Remark \ref{rem:a priori} part (2)). Moreover when $\bsg$ is of merely subquadratic and $b$ is bounded, our result still ensures the existence of locally H\"{o}lderian solution to \eqref{equ:EH-game}.

The system \eqref{equ:EH-game} belongs to the \emph{diagonally quadratic}
 class studied recently in \cite{HuTan15}, whose Theorem 2.7 implies the
 existence of a unique bounded solution of \eqref{equ:EH-game} with
 non-Markovian bounded terminal condition.
\end{remark}

\subsection{A scalar example with unbounded coefficients}
\renewcommand{\bsZ}{Z} % w and Z will be reverted back
\renewcommand{\bsw}{w} % to bold after this example
Given continuous functions $f, g: \R^d \to \R$, with $g\in \Caa_{loc} \cap \linf$,
but $f$ possibly
unbounded, we consider the BSDE
\begin{equation}\label{equ:single}
 dY_t = - \tfrac12 f(\bsX_t) |\bsZ_t|^2 dt + \bsZ_t\, d\bsW_t, \quad Y_T = g(\bsX_T).
\end{equation}
An equation of this type played a central role in a recent solution
\cite{Cetin-Danilova}
of a long-standing
open problem of \cite{Sub91}.
Since the ``coefficient'' $f$ in front of the quadratic nonlinearity is
unbounded, the generator of \eqref{equ:single} does not satisfy the standard
quadratic growth bound in $\bsz$ (as presented, e.g., in \cite{Kobylanski}).

Our Theorem \ref{thm:existence} implies that \eqref{equ:single} admits a
bounded locally H\"{o}lderian solution. Indeed, consider a sequence
$\{f^m\}$ of bounded Lipschitz
approximations  of $f$ such that $\lim_{m\to \infty} f^m(\bsx) =
f(\bsx)$ for any $\bsx\in \R^d$, and the approximating BSDE
\begin{equation}
\label{equ:apB}
 dY^m_t = -\tfrac12 f^m(\bsX_t) (|\bsZ^m_t|^2 \wedge m) dt + \bsZ^m_t \,
 d\bsW_t, \quad Y^m_T= g(X_T), m\in\N.
 \end{equation}
Standard Lipschitz theory implies that \eqref{equ:apB} admits a unique
bounded continuous solution $(v^m, \bsw^m)$. Moreover,
the generator of \eqref{equ:apB} satisfies the condition (BF) of Definition
\ref{def:BF} with
 $\bsfl^m=\bsfs^m=\bsfk^m=0$; the component
 $\bsfq^m$ satisfies the
quadratic-triangular growth condition on each $B_n$ uniformly in $m$.
By Proposition \ref{thm:BF-Lyap} above, $\bsf$ admits a local $\seq{c}$-Lyapunov
sequence, for each $\seq{c}$.
To establish a-priori boundedness, we rewrite \eqref{equ:apB} as
\[
 dY^m_t = \bsZ^m_t \big[-\tfrac12 f^m(\bsX_t) \tfrac{|\bsZ^m_t|^2 \wedge m}{|\bsZ^m|^2} (\bsZ^m_t)^\top dt + d\bsW_t\big], \quad Y^m_T = g(\bsX_T).
\]
Since $f^m$ and $g$ are bounded, a simple measure-change argument implies
that $\|v^m\|_{\linf}\leq \|g\|_{\linf}$. Therefore it is enough to pick
a local $(\norm{g}_{\linf})$-Lyapunov pair to establish the existence of a
bounded locally H\"{o}lderian solution by Theorem
\ref{thm:existence}.
It is worth noting that our uniqueness results do not apply in this case.
In fact, as far as we know, no general-purpose uniqueness result is known
for BSDE of this type.
\begin{remark} The techniques of the present paper, geared towards systems
of equations, have limited impact in the one-dimensional case where
powerful methods based on comparison principle apply. To
illustrate that point, we note that
the existence of a Markovian solution for \eqref{equ:single} can also be
established using a localization technique of \cite{Briand-Hu} or from a
forward point of view, as in   \cite{Barrieu-El-Karoui}, as follows.
 With $f^{m,p}(\bsx) = (-p) \vee f(\bsx) \wedge m$ for
$m,p\in \N$, the approximating BSDE
\[
 dY^{m,p}_t = -\tfrac12 f^{m,p}(\bsX_t) |\bsZ^{m,p}_t|^2 dt + \bsZ^{m,p}_t \, d\bsW_t, \quad Y^{m,p}_T= g(X_T)
\]
admits a unique bounded continuous solution $(v^{m,p}, \bsw^{m,p})$, cf.
\cite[Theorems 2.3 and 3.7]{Kobylanski}. Define the exit time $\tau_n
=\inf\{u\geq 0\,:\, \bsX_u \notin B_n\}$. Comparison theorem for quadratic
BSDE implies $v^{m,p+1}\leq v^{m,p}\leq v^{m+1, p}$. It then follows from
the monotone stability of quadratic BSDE (cf. \cite[Proposition
2.4]{Kobylanski}) that $Y^{m,p}_{\cdot \wedge \tau_n}= v^{m,p}(\cdot,
\bsX^{\tau_n}_\cdot)$ increasingly converges to some process $Y^{p}_{\cdot
\wedge \tau_n}$ as $m\rightarrow \infty$, and $Y^{p}_{\cdot \wedge
\tau_n}$ decreasingly converges to $Y_{\cdot \wedge \tau_n}= v(\cdot,
\bsX^{\tau_n}_\cdot)$, for some function $v$, as $p\rightarrow \infty$.
The convergence of $Y^{m,p}_{\cdot \wedge \tau_n}$ to $Y_{\cdot \wedge
\tau_n}$ is also uniform and $Z^{m,p}_{\cdot \wedge \tau_n}$ also
converges to some $Z_{\cdot \wedge \tau_n}$ in \bmo, cf. \cite[Theorems
4.5 and 4.7]{Barrieu-El-Karoui}. Sending $n \rightarrow \infty$, we
obtain a solution to \eqref{equ:single}.
\end{remark}
\renewcommand{\bsZ}{\boldsymbol{Z}}
\renewcommand{\bsw}{\boldsymbol{w}}

\section{Proof of Theorem \ref{thm:abstract}} \label{sec:uniform est}

Within this proof, all the constants $T, d, N, \Ld, L, \norm{b}_{\linf}, \norm{\sigma}_{\linf}$, and functions
$\{h_n\}$, which define the setting or appear in the assumptions of Theorem
\ref{thm:abstract} will be
thought of as \emph {global variables}; any function of them will
be treated as a constant, which we call an \define{universal constant}. For
quantities dependent on additional parameters, we write, e.g., $C=C(\psi)$
to signal that, in addition to the global
variables mentioned above, $C$ also depends on $\psi$.
In Hardy's manner, universal, constants will always be denoted by the
letter $C$ which may change from line to line, and they are always positive.
To increase readability, we use the notation $\leq_C$ as follows
 \begin{align*}
   `a \lec b' \quad\text{ stands for }\quad`a\leq C\, b'.
 \end{align*}

Furthermore,
we fix both $m$ and $n\in\N$, and removing them almost entirely from
the notation throughout this section. It is important to note, however, that
our treatment of $m$ and $n$ will be different.
One one hand, since we are after uniform estimates on the entire sequence
$\set{\bsv^m}$, we do not allow any of our constants to depend on $m$. (We will see later that the dependence on $\{k^m_n\}$ is through its $\mathbb{L}^{q_n}$-norm which is assumed to be bounded uniformly in $m$.)
On the other hand, all our analysis in this section will be restricted locally to the ball $B_n(b_0)$. Therefore
$n$ \emph{is
be added, temporarily, to the list of
universal constants}
 and all the estimates below will depend on it
implicitly. Hence, for the time being, the
conditions of Theorem
\ref{thm:abstract} is localized to $\bsx \in B_n(b_0)$ and we
simply assume, for the reminder of this section,
that conditions of \eqref{asm:g} - \eqref{asm:h},   with the centre $b_0$ of the ball $B_n(b_0)$ and
indices $m, n$ removed,  are satisfied. In particular, we assume that
there exists constants $\alpha, b, c, \ell, M>0$ and $q>1+d/2$ such that
\begin{equation}\label{eq:abc const} \norm{\bsv}_{\linf([0,T]\times B_n)}\leq c \quad \text{and}\quad \norm{\bsg}_{C^{\alpha}( B_n)} \leq b.
\end{equation}
There exists a Lyapunov pair $(h,k)\in \Ly(\bsf, 2c)$ on $B_n$ such that
\begin{align}
\label{equ:f}
\|k\|_{\mathbb{L}^q([0,T]\times B_n)} \leq \ell \quad \text{and} \quad
\abs{\bsf(t,\bsx,\bsy,\bsz)}\leq M( \abs{\bsz}^2+ k(\tx)),
\end{align}
for all
$(\tx) \in [0,T]\times B_n$, $|\bsy|\leq c$, and $\bsz\in\R^{N\times d}$.
The $\alpha, b, c, \ell, M$ and $q$ are also added to the list of universal constants. Throughout this section, the dependence on $b_0$ is only through $\alpha, b, c, \ell, M$ and $q$.

\subsection{A ``testing'' Lemma}
As we already mentioned,
for each initial condition, the SDE \eqref{equ:X} admits a unique strong
solution $\bsX^{\tx} = (\bsX^{\tx}_u)_{u\in [t,T]}$. For notational convenience in several proofs below, we allow $\bsX$ to start from negative time, i.e. $t\leq 0$. Therefore we extend $\bsb$ and $\bsig$ via
\[
 \bsb(t, \bsx) = \bsb(0, \bsx) \quad \eand \quad \bsig(t, \bsx) = \bsig(0, \bsx) \quad \text{for } t\leq 0.
\]
These extended coefficients still satisfy conditions (1)-(3) after
\eqref{equ:X}, ensuring the existence of the unique strong solution, which
is still denoted by $\bsX^{\tx} = (\bsX^{\tx}_u)_{u\in [t,T]}$. Its
infinitesimal generator is given by
\begin{equation}\label{equ:generator}
 \sL = \sum_{i=1}^d \bsb_i(\cdot,\cdot)\, D_i + \tot \sum_{i,j=1}^d \bsa
 (\cdot,\cdot)\, D_{ij}.
 \end{equation}
Parameterized by $(\tx)$, the laws of these solutions constitute a Markov
family $(\PP^{\tx})_{(\tx)\in (-\infty,T] \times \R^d}$ of probability measures on $C([0,T]\to
\R^d)$. (Even through the canonical process may start from negative time, we only focus on its trajectory on $[0,T]$.)
In a minimal notational overload, we use $\bsX$ for the
coordinate map on $C([0,T]\to\R^d)$ and set throughout
\begin{align}
\label{equ:YZ}
  \bsY_u = \bsv(u, \bsX_u) \quad \eand \quad \bsZ_u = \bsw (u,\bsX_u) \quad \text{for } u\in[t\vee 0, T].
\end{align}
A $C^{1,2}$-function $\vp:[0,T]\times \R^d \to [0,1]$
is said to be \define {testable}  if its support is contained in
$[0,T]\times B_n$
and we have $\Gamma_{\vp}<\infty$, where
\begin{align*}
   \Gamma_{\vp} := \sup_{(\tx)\, :\, \vp(\tx) > 0} \Big( \abs{ \tpddt \vp} +
   \abs{D \vp\, \bsb} + \tot \abs{D^2 \vp : \bsa}
   + \oo{\vp} \abs{D\vp}^2 \Big).
\end{align*}
\begin{lemma} \label{lem:testing}
With $c$ as in \eqref{eq:abc const},
there exists a universal constant $C>0$ such that, for each $\bsc\in \R^N$
with $\abs{\bsc}\leq c$, all $t'\in [t\vee 0, T]$,  $\bsx\in\R^d$,
and any testable $\vp$, we have
     \begin{multline*}
     \EE^{\tx} \Big[\int_{t'}^T \ind{\{\vp(u,\bsX_u) =1\}}\abs{\bsZ_u}^2\, du\Big]
     \lec
     \Gamma_ {\vp}\,
     \EE^{\tx}\Big[\int_{t'}^T \ind{\{\vp(u,\bsX_u) \in (0,1)\}} \abs{\bsY_u -
     \bsc}^2\, du\Big]\, +\\
      +
      \EE^{\tx}\Big[\int_{t'}^T \ind{\{\vp(u,\bsX_u) >0\}}\,  k(u, \bsX_u)\, du\Big]
      +  \EE^{\tx}[\ind{\{\vp(T,\bsX_T)>0\}}
      \abs{\bsY_T - \bsc}^2],
     \end{multline*}
where the expectation $\EE^{\tx}$ is with respect to $\PP^{\tx}$.
\end{lemma}
\begin{proof}
We overload the notation by writing $\vp$ for both the process $\vp
(\cdot,\bsX)$ and the function $\vp$;
similarly, having fixed $\bsc\in\R^N$ with $\abs{\bsc}\leq c$, we write
$h^c$ both for the function
$h(\cdot - \bsc)$ and the process $h(\bsY-\bsc)$.
We define the product process $F= \vp h^c=
\vp(\cdot,\bsX) h(\bsY-\bsc)$ and write down
the semimartingale decompositions (under any $\PP^{\tx}$)
\[ d \vp = (\tfrac{\partial}{\partial u}\vp + \sL \vp)\, du +
 D\vp \, \bsig\, d\bsW\]
 and
\begin{align}
\label{equ:ito-H-Y}
    d h^c =
   \Big( \tot  D^2h^c :  \scl{\bsZ}{\bsZ}_{\bsa}
- Dh^c\, \bsf \Big)\, du
    + Dh^c\, \bsZ\, \bsig \, d\bsW.
\end{align}
Reminding the reader that $dF\eqm \alpha$ means  that
$F - \int_0^\cdot \alpha_s\, ds$ is a local martingale, we conclude that
\begin{align}
\label{equ:dF}
    dF &\eqm \vp
     \Big(  \tot D^2 h^c: \scl{\bsZ}{\bsZ}_{\bsa} - Dh^c\, \bsf \Big) + h^c
     (\tfrac{\partial}{\partial u}\vp + \sL \vp)
     + \scl{Dh^c\bsZ}{D\vp}_{\bsa}.
\end{align}
The $C^2$-regularity of the function $h$ and the fact that $h(\boldsymbol{0})=Dh(\boldsymbol{0})=\boldsymbol{0}$
imply that there exists a constant $C>0$ which depends only on $h$ and $c$ such that
\begin{align}
\label{equ:sq}
h^c (\bsy) \leq C |\bsy - \bsc |^2 \quad \eand \quad \abs{D h^c (\bsy) }^2
\leq C |\bsy - \bsc|^2 \quad
\text{
for all $\bsy$ with $\abs{\bsy}\leq  c$.}
\end{align}
The fact that $(h,k)\in \Ly(\bsf, 2c)$
coupled with the boundedness of $a$ and the fact that
$\abs{\bsY- \bsc}\leq 2 c$,
imply that the
right-hand side of \eqref{equ:dF} above is bounded from below by
%\begin{align}
%\label{equ:CS-sigma}
    %\scl{\bsa}{\bsb}_{\bsa} \geq_C -\abs{\scl{\bsa}{\bsb}_{\Id}}= - \abs{\bsa
    %\bsb^\top},\text {
    %for all (rows) } \bsa,\bsb\in
    %\R^d,
%\end{align}
%where the universal constant $C$ depends only $\Ld$,
\begin{multline*}
       \vp (\abs{\bsZ}^2-k) - C \abs{\bsY-\bsc}^2 \abs{\tpddt \vp +
     \sL \vp}
     - \sum_{il} \Big(  \oo{2} \vp  (Z_{il})^2
     +  C \oo{2\vp} (D_i h^c\, D_l \vp)^2 \Big) \ind{\vp >0} \geq \\ \geq
     \oo{2} \vp \abs{\bsZ}^2 - \vp k - C \abs{\bsY - \bsc}^2 \Big( \abs
     {\tfrac{\partial}{\partial u}\vp +
     \sL
     \vp} +
     \oo{\vp} \abs{D \vp}^2\Big) \ind{\vp>0}.
\end{multline*}
It remains take the expectation and use boundedness of $F$ (implied by the boundedness of
$\vp$ and $\bsY$ on $[0,T]\times B_n$, as well as continuity of $h^c$) and its positivity.
\end{proof}

\subsection{First consequences of the regularity of transition densities}
It follows from the conditions imposed on $\bsig$ and $b$ (see
\cite[Theorem 3.2.1, p.~71]{StrVar06}) that
the Markov family $(\PP^{\tx})$
admits a family of transition
densities \[ p(\tx; \txp), \quad \ t'\in [t, T], \bsx'\in \R^d.\] Moreover,
they satisfy the
following fundamental estimate (known as the Aronson's estimate): there
exist constants $\usig, \osig>0$, as
well as
$\uC, \oC>0$, depending only on the $\linf$- and ellipticity bounds on $\bsb$ and $\bsig$,
such that,
for all $0\leq t < t'\leq T$ and all $\bsx,\bsx' \in \R^d$, we have
\begin{align}
\label{equ:Aronson}
   \tfrac{\uC}{ (t-t')^{d/2}}\,  e^{-\tfrac{r}{2\usig^2}}
\leq p (\tx;\txp) \leq
   \tfrac{\oC}{(t-t')^{d/2}} \, e^{-\tfrac{r}{2\osig^2}}
   \ewhere
r=\tfrac{\abs{\bsx'-\bsx}^2}{{t'-t}}.
\end{align}
\begin{remark}
Under our assumptions (cf. conditions (1)-(3) after \eqref{equ:X}),
the upper bound in \eqref{equ:Aronson} can be obtained by the parametrix
method (see
\cite[equation (6.12), p. 24]{Friedman}). The lower bound was first
obtained in the paraboloid $|\bsx'-\bsx|^2 \leq \text{const} (t'-t)$ by
\cite[equation (4.75)]{Ilin-et-al}, then extended globally by a standard
chaining argument. When  $\bsb$ and $\bsig$ are only measurable and $\sL$
is in the divergence form,
\eqref{equ:Aronson} was obtained by \cite[Theorem 1]{Aro67}.
\end{remark}
The first consequence of the estimates \eqref{equ:Aronson} is the following
uniform boundedness result:
\begin{lemma}
\label{lem:Z-bound}
 There exists a universal constant $C>0$ such that
 \[ \EE^{\tx}\left[ \int_t^T \big(\abs{\bsZ_u}^2 + k(u,\bsX_u)\big)
 \inds{ \bsX_u \in B_n}
    \, du\right] \leq C,\]
   for all $\tx\in\TR$.
 \end{lemma}
 \begin{proof}
Let $\chi$ be a testable function,  such that, for all $t\in [0,T]$,
$\chi(\tx)=1$ for $ \bsx \in B_{n-1}$
 and $\chi(\tx)=0$ for $\bsx\in B_n^c$.  According to Lemma
\ref{lem:testing} and the boundedness of $\bsv$ on $[0,T]\times B_n$, there exists a universal constant $C=C(\chi)$ such that,
uniformly over $(\tx)\in \TR$ we have
\[ \EE^{\tx} \Big[\int_t^T \abs{\bsZ_u}^2 \inds{\bsX_u\in B_{n-1}}\, du\Big] \lec
1+ \EE^{\tx} \Big[\int_t^T k(u, \bsX_u) \inds{\bsX_u \in B_n}\, du\Big].\]
By H\" older's inequality with $1/q+1/q'=1$, we have
\begin{multline*}
\EE^{\tx} \Big[\int_t^T k(u,\bsX_u)\inds{\bsX_u\in B_n} \, du\Big]
=
\int_t^T \int_{B_n} k(u,\bsxi) p(\tx, u;\bsxi)\, du\, d\bsxi \leq \\
\leq
\norm{k}_{\el^{q}([0,T]\times B_n)}
\Big(\int_t^T \int_{B_n} p(\tx, u;\bsxi)^{q'} \, du\, d\bsxi\Big)^{1/q'}.
\end{multline*}
The proof is completed once we ues the upper bound in \eqref{equ:Aronson}
and the fact that $q'<1+2/d$ to obtain
\begin{align*}
 \int_t^T \int_{B_n} p(t, \bsx; u, \bsxi)^{q'}\, du\, d\bsxi & \lec \int_t^T
 (u-t)^{-\tfrac{d}{2}(q'-1)} \Big(\int_{B_n} (u-t)^{-\tfrac{d}{2}}
 e^{-\tfrac{q'|\bsxi-\bsx|^2}{2\overline{\sigma}^2 (u-t)}} d\bsxi\Big)\, du\\
 &\lec \int_t^T (u-t)^{-\tfrac{d}{2}(q'-1)} \, du\leq
 T^{1-\tfrac{d}{2(q-1)}}.\qedhere
\end{align*}
 \end{proof}
The uniform bound of Lemma \ref{lem:Z-bound} helps provide the following
fundamental relation between $\bsw$ and $\bsv$.
\begin{lemma}
\label{lem:weak-Jacobian}
$\bsw$ is the weak (spatial) Jacobian $D \bsv$ of $\bsv$ on $(0,T)\times
\R^d$.
\end{lemma}
\begin{proof}
Given $\epsilon \in (0, T/2)$ and
the testable function $\chi$ from the proof of Lemma
\ref{lem:Z-bound}, we define
\[ \tbsv(\tx) = \bsv(t, \bsx) \chi(\tx) \quad \efor t, \bsx\in [0, T] \times \R^d,\]
and let the
sequence $\set{\tbsvl}$ (with $l\geq 1/\epsilon$) of approximations to $\tbsv$ be given by
\[ \tbsvl(t,\bsx) = l \int_{t}^{t+1/l} \EE^{\tx}[ \tbsv(u, \bsX_u)
]\, du \quad \efor \tx\in [0, T-\epsilon]\times \R^d.\]
The functions $(\tbsvl)$ are uniformly bounded (by $c$, in fact), and,
thanks to smoothness of the transition densities of $\bsX$,  each $\tbsvl$ is $C^{1,2}$-differentiable.
Moreover, as one readily checks, we have
\[ (\tddt + \sL) \tbsvl(t,\bsx) = l \Big( \EE^{\tx}[
\tbsv(t+1/l, \bsX_{t+1/l})] - \tbsv(\tx) \Big). \]
Having fixed a pair $(\tx)\in [0, T-\epsilon]\times \R^d$, we apply \itos{} formula to
$\tbsv(\cdot,\bsX^{\tx})$, use the boundedness of $\bsv$ on $[0, T]\times
B_n$ in the second inequality below, and recall the second inequality in
\eqref{equ:f} in the last inequality, to obtain
\begin{align*}
\abs{ (\tddt + \sL) \tbsvl( t,\bsx)}
& \lec
l \,\EE^{\tx}\Big[\int_t^{t+1/l} |\bsv(u, \bsX_u)|\abs{(\tfrac{d}{dt} +
\sL) \chi(u, \bsX_u)} du\Big]\\ & \qquad + l\, \EE^{\tx}\Big[\int_t^{t+1/l} \chi(u,
\bsX_u) |\bsf(u, \bsX_u, \bsY_u, \bsZ_u)| du\Big]\\ & \qquad
+ l\, \EE^{\tx}\Big[\int_t^{t+1/l} \abs{D\chi(u, \bsX_u)} \abs{\bsZ_u}du\Big]\\
&\lec1+
l\, \EE^{\tx}\Big[\int_t^{t+1/l}
\chi (u,\bsX_u) \abs{\bsf(u,\bsX_u, \bsY_u, \bsZ_u)}\, du\Big] \\ & \qquad +
\EE^{\tx}\Big[ \int_t^{t+1/l} \inds{\chi(u,\bsX_u)>0} \abs{\bsZ_u}^2\, du\Big]\\
&\lec  1 + l\, \EE^{\tx} \Big[\int_t^{t+1/l} (k(u, \bsX_u) + \abs{\bsZ_u}^2)
\inds{\bsX_u \in B_n}\, du\Big],
\end{align*}
for a universal constant $C=C(\chi)$. The Markov property of the family
  $(\PP^{\tx})$ now implies that, with
  $\bsZ'_s =\bsw(s,\bsX'_s)$, where
  $\bsX'$ denotes the coordinate
  process inside the $\PP^{u,\bsX_u}$-expectation, we have
\begin{equation}
\label{equ:fvb}
\begin{split}
\EE^{\tx} &\Big[\int_t^{T-\epsilon}
\abs{ (\tfrac{d}{dt} + \sL) \tbsvl( u,\bsX_u)}
\, du\Big] \\ &\lec 1 +
\EE^{\tx} \left[
\int_t^{T-\epsilon} \EE^{u,\bsX_u}\Big[ l \int_u^{u+1/l}
\Big(\abs{\bsZ'_s}^2 + k(s,\bsX'_s) \inds{\bsX'_s \in B_n}\Big)\, ds\, \Big]
du\right]
\\
&=  1+ \EE^{\tx}  \Big[\int_t^{T-\epsilon} l \int_{u}^{u+1/l}
\big(\abs{\bsZ_s}^2 + k(s,\bsX_s)\big)\inds{\bsX_s\in B_n}\, ds\,
\, du  \Big]\\
&\lec 1+ \EE^{\tx}\Big[ \int_t^{T-\epsilon}
\Big(\abs{\bsZ_s}^2 + k(s,\bsX_s)\Big)\inds{\bsX_s\in B_n}\, ds\Big]
\leq C, \quad \text{for all } \ell,
\end{split}
\end{equation}
where the last inequality follows from Lemma \ref{lem:Z-bound};  the constant $C$ obtained above is also uniform for all $\tx \in [0,T-\epsilon]\times \R^d$.

Continuity of $\bsv$ implies that $\hat{\bsv}$ is also continuous, and,
hence,
uniformly continuous on compacts. For any $\tilde{\epsilon}>0$,
there exists $\delta = \delta(\tilde{\epsilon}, n)\in (0,1)$ such that
\[
 \abs{\hat{\bsv}(t, \bsx) - \hat{\bsv}(t', \bsx')} \leq \tilde{\epsilon}
 \quad \text{if}\quad |(t, \bsx) - (t', \bsx')|\leq \delta \eand |\bsx- b_0|\leq
 n+1.
\]
The difference above vanishes if $|(t, \bsx)-(t', \bsx')|\leq \delta$ and
$|\bsx- b_0|>n+1$, since,
in that case, $\bsx, \bsx'\notin \text{supp} \chi$.
Therefore, using the boundedness of $\hat{\bsv}$ and the upper bound in \eqref{equ:Aronson}, for $(\tx)\in [0,
T-\epsilon]\times \R^d$ we have
\begin{align*}
 \abs{\hat{\bsv}^{(l)}(\tx) - \hat{\bsv}(\tx)}
 \leq & \iintl{[t, t+1/l]\times \R^d} \inds{|\bsxi-\bsx|>\delta}
 \abs{\hat{\bsv}(u, \bsxi) -\hat{\bsv}(t, \bsx)} l\, p(t, \bsx; u, \bsxi)
 \,du\,
 d\bsxi\\
 &+ \iintl{[t, t+1/l] \times \R^d} \inds{|\bsxi- \bsx|\leq \delta,
 |\bsx|\leq n+1} \abs{\hat{\bsv}(u, \bsxi) -\hat{\bsv}(t, \bsx)} l\, p(t,
 \bsx; u, \bsxi) \,du\, d\bsxi\\
 \lec& \, l \int_0^{1/l}\int_{|\bsxi|\geq \delta} u^{-d/2}
 e^{-\tfrac{|\bsxi|^2}{2\overline{\sigma}^2 u}}\, d\bsxi\, du + \tilde{\epsilon}\\
 \lec&\, l \int_0^{1/l}(1-\Phi(\delta/\sqrt{ \overline{\sigma}^2 u})) \,du+ \tilde{\epsilon},
\end{align*}
where $\Phi(\cdot)$ is the distribution function for standard normal.
Note that $\lim_{u\downarrow 0} \Phi(\delta/\sqrt{ \overline{\sigma}^2
u})=1$. The last expression is less than $2\tilde{\epsilon}$, for
sufficiently large $l$, uniformly for $\tx \in [0,T-\epsilon]\times \R^d$.
Since the choice of $\tilde{\epsilon}$ is arbitrary, the previous estimates
implies the uniform convergence of $\{\hat{\bsv}^{(l)}\}$ to $\hat{\bsv}$ on
$[0, T-\epsilon]\times \R^d$.

Setting $\tbsYl = \tbsvl(\cdot, \bsX^{0,\bsx})$ and $\tbsY =
\tbsv(\cdot,\bsX^{0, \bsx})$, we use the uniform convergence of
$\set{\hat{\bsv}^{(l)}}$ to obtain $\tbsYl \to \tbsY$, uniformly. Applying \itos{}
formula to $\sabs{\tbsY-\tbsYl}^2$ and using \eqref{equ:fvb}, we obtain
\[
 \EE^{0, \bsx}\big[\ab{\tbsYl - \tbsY}_{T-\epsilon}\big] \lec \Big\|\sup_{u\in [0, T-\epsilon]}|\tbsYl_u - \tbsY_u|\Big\|_{\mathbb{L}^\infty},
\]
which converges to $0$.
This means that
\[ \lim_l \int_0^{T-\epsilon} \int_{B_n} g^2_l(u,\bsxi) \, p(0, \bsx ;u, \bsxi)\, du\,
d\bsxi  = 0, \quad \text{where } g_l= \abs{
  (\chi\bsw + \bsv D\chi -   D \tbsvl) \bsig }.
\]
By the lower bound in \eqref{equ:Aronson}, the density $p(0, \bsx;\cdot,\cdot)$ is
bounded away from $0$ on
$[\epsilon,T-\epsilon]\times B_n$ and $\bsig$ is uniformly elliptic. Therefore,
$D\tbsvl \to
\chi \bsw + \bsv D \chi$ in $\ltwo([\epsilon, T-\epsilon]\times B_n)$, as $l\to\infty$. Coupled with the
fact that $\tbsvl \to \tbsv$ also in $\ltwo([\epsilon, T-\epsilon]\times B_n)$, this implies that
$\chi \bsw + \bsv D\chi $ is the weak Jacobian of $\bsv \chi$.  The
statement follows by noting that for each compact $K\subseteq \R^d$, $\chi
\bsw+ \bsv D\chi = \bsw$ and $\bsv\chi = \bsv$ on $[0,T]\times K$, for large enough $n$, and the choice of $\epsilon$ is arbitrary.
\end{proof}

\subsection{Uniform local estimates}
We now choose and fix $R\in (0, 1/4]$
and a pair $(\txz) \in [0,T]\times B_{n-1}$. It is important to note that
\emph{none} of the constants in the sequence of lemmas in the next two subsections depends on the choice of $(\txz)$ and $R$.

We will use the point $ (\txz)$ (or only
$\bsx_0$) as the
origin throughout the proof and dilate the coordinate system around it  with
the change of variables $ (\tx) \to ( \tau, \bsxi)$, given by
\begin{align*}
     \bsx = \bsx_0 + R \bsxi \quad  \eand \quad t = t_0 +  R^2 \tau.
 \end{align*}
 Most balls, cylinders, etc.~in the sequel will be centered around $(t_0,\bsx_0)$
 (or $\bsx_0$) and their dimensions (radius, etc) will have much nicer
 expressions in the $(\tau,\bsxi)$-coordinates, so we introduce the following
 notation:
 \begin{align*}
   \beta(\rho) &= \sets{\bsx\in\R^d}{ \abs{\bsx-\bsx_0} \leq \rho R}, \eand \\
    \gamma(\theta; \rho) &= \Bsets{ (\tx)\in\TR }{ t_0
    \leq t \leq t_0+ \theta  R^2, \abs{\bsx-\bsx_0}\leq \rho R},
 \end{align*}
 for the ball $\beta$ and the parabolic cylinder $\gamma$.
Typically, a function $\tilde{\vp}: \R \times \R^d \rightarrow \R$ will be defined in $
(\tau, \bsxi)$-coordinates, and then its counterpart
\begin{align}
\label{equ:vp-tvp}
     \vp (\tx) = \tilde{\vp}\left(\tfrac{t-t_0}{
     R^2}, \tfrac{\bsx- \bsx_0}{R} \right),
\end{align}
restricted to $\TR$,
is used in computations. A similar notation will be used for functions of $\bsx$
only or for subsets of $[0,T] \times \R^d$ or $\R^d$ (identified with their indicators).
In the same spirit, we set
 $\tT=(T-t_0)/( R^2)$.

 Within this proof, $\int$ denotes the integral over $R^d$, while $\iint$ stands for
the integral over $[t_0,T]\times \R^d$. If the domain of integration is
notationally further
restricted, as e.g., in $\iint_D$, the integral is taken over $([t_0,T]\times \R^d)\cap D$ (or $\R^d\cap D$ in the spatial case). Similarly, in order to avoid
repeated explicit minimization with $T$, we assume that temporal variables
cannot take
values above $T$, so that, for example, the interval $[t_0,t_0+4 R^2]$
coincides with $[t_0,T]$, when $t_0+4  R^2>T$. Meanwhile, as we mentioned before, $t$ in $\PP^{t,x}$ is allowed to be negative.

We continue with some consequences of \eqref{equ:Aronson} which will be
used in the sequel.
 Given the  origin $(t_0, \bsx_0)$ and the radius $R$ fixed above,
 we introduce the following  shortcut
 \begin{align}
 \label{equ:p-tau}
   \peps(\tx) = p(t_0 - \eps R^2, \bsx_0; \tx), \quad \efor \eps\geq 0,
 \end{align}
 for the transition density, and state
several useful estimates where
 the functions  $\Delta_{\sigma}, \delta_{\sigma}$ are given by
\begin{align}
\label{equ:deltas}
\Delta_{\sigma}(\eps) = \exp(-\oo{2\sigma^2} \eps^{-1}) \quad \eand \quad
  \delta_{\sigma} (\eps) = \eps^{-d/2} \Delta_{\sigma}(\eps),
\end{align}
and extended to $\eps=0$, $\eps=+\infty$ by continuity.
We also define the positive (universal) constant
$\eps_0$ by
\begin{align}
\label{equ:eps-z}
\eps_0 = \inf\sets{ \eps>0 }{\tdl(\eps) = 1/2 } \wedge \min(1, (\osig^2 d)^{-1}).
\end{align}
Note that $\delta_{\overline{\sigma}}(\epsilon)$ is increasing on $[0, (\overline{\sigma}^2 d)^{-1}]$ and decreasing on $[(\overline{\sigma}^2 d)^{-1}, \infty)$ with $\delta_{\overline{\sigma}}(0)=0$. Therefore the definition of $\epsilon_0$ implies $\delta_{\overline{\sigma}}(\epsilon)\leq \delta_{\overline{\sigma}}(\epsilon_0)\leq 1/2$ for $\epsilon \in [0, \epsilon_0]$. The role of $\epsilon_0$ will be clear in Proposition \ref{pro:hol-fil-2} below.
\begin{lemma}
\label{lem:Gauss} There exists a universal constant $C>0$ such that, for all
$(\tx)\in \gamma(4,4)$ and
$\eps\in (0,\eps_0/2]$, we have
\begin{align}
\label{equ:G1}
      \peps(\tx) &\lec R^{-d} &&\ewhen (\tx)\not\in
      \gamma(1/4;1/2), \\
\label{equ:G2}
     \peps(\tx) &\gec R^{-d} \, \delta_{\usig/4} (\eps_0/2+\eps ) &&\ewhen
      t \geq t_0 + \eps_0 R^2/2, \\
\label{equ:G3}
     \peps(\tx) &\lec R^{-d} \, \delta_{\osig} (\eps_0/2+\eps) &&\ewhen
    t< t_0+ \eps_0 R^2/2 \eand
    \bsx\not\in \beta(1).
 \end{align}
\end{lemma}
\begin{proof}
We set $\btau = \eps+ (t-t_0)/R^2 >0$,  $\bsxi =
(\bsx-\bsx_0)/R$ and
$r = \btau/\abs{\bsxi}^2 \in (0,\infty]$,   and note
that the bounds in the density estimates
\eqref{equ:Aronson} can be represented in two forms (with the outer one
holding only
for $\bsxi\ne 0$):
\begin{align}
\label{equ:alt-A}
    \abs{\bsxi}^{-d} \delta_{\usig} ( r )  =
    \btau^{-\tfrac{d}{2}} \Delta_{\usig} ( r )  \lec
   R^d\, \peps(\tx)  \lec
     \btau^{-\tfrac{d}{2}} \Delta_{\osig} ( r )=
    \abs{\bsxi}^{-d} \delta_{\osig} ( r ),
\end{align}
which will be used throughout the proof.

- Inequality \eqref{equ:G1}:
 Both $\delta_{\sigma}$ and $\Delta_{\sigma}$ are
bounded by a constant $C$ on $
[0,\infty]$, so, by the right-hand side of \eqref{equ:alt-A}, we have
\begin{align*}
 R^d \peps(\tx) \lec
    \max( \abs{\bsxi}, \sqrt{\btau})^{-d},
\end{align*}
which, in turn, implies \eqref{equ:G1} since $\max(\abs{\bsxi}, \sqrt{\btau})
\geq 1/2$ on $\gamma(1/4;1/2)^c$.

- Inequality \eqref{equ:G2}:
Under the conditions of \eqref{equ:G2}, we have $r \geq \btau/16$, and so,
by monotonicity of $\Delta_{\usig}$ and \eqref{equ:alt-A}, we have
\begin{align*}
   R^d\, \peps(\tx) \gec
    \btau^{-d/2} \Delta_{\usig} ( r ) \gec
    \delta_{\usig} ( \btau/16 ) \gec \delta_{\usig/4}(\btau).
\end{align*}
The function $\delta_{\usig/4}$ attains it maximum at $16( \usig^2 d)^{-1}$, and
is nondecreasing to the left of it and nonincreasing and positive to the right.
Since $\eps_0/2+\eps \leq \btau $, in the case that $\btau \leq 16( \usig^2 d)^
{-1}$,
we have
$\delta_ {\usig/4} (\btau) \geq \delta_
{\usig/4}(\eps_0/2+\eps)$.
On the other hand, if  $\btau \in (16( \usig^2 d)^
{-1},4 \tfrac{1}{2}]$, we
have
\begin{align*}
    \delta_{\usig/4}(\eps_0/2+\eps) \leq \delta_{\usig/4}(16(\usig^2 d)^{-1})
    % = \tfrac{\delta_{\usig}(16(\usig^2 d)^{-1})}{\delta_{\usig}(5)} \delta_
    % {\usig/4} (5)
    \lec \delta_{\usig/4}(4 \tfrac12) \leq \delta_{\usig/4}(\btau).
\end{align*}
Both alternatives lead to \eqref{equ:G2}.

- Inequality \eqref{equ:G3}:
The conditions of  \eqref{equ:G3} translate into $\btau \leq \eps_0/2+\eps \leq
(\osig^2
d)^{-1}$, as well as $r\leq \btau$ (since $\abs{\bsxi}\geq 1$).
The function $\delta_{\osig}$
is nondecreasing on $[0,(\osig^2 d)^{-1}]$, so the \eqref{equ:alt-A} implies
 \[
   R^d\, \peps(\tx)   \lec \delta_{\osig} ( r ) \leq \delta_
   {\osig}
   (\eps_0/2+\eps). \qedhere\]
\end{proof}

An operational form of Lemma \ref{lem:testing}, stated in Proposition
\ref{pro:testing} below, employs a particular testing function $\vp$,
obtained via \eqref{equ:vp-tvp} from
a function $\tvp:\R \times \R^d \to [0,1]$ in the class $C^{1,2}(\R\times \R^d)$ satisfying
\begin{align*}
\tvp(\tau,\bsxi) =
1, \text{ when } \tau \leq 1 \text{ and } \abs{\bsxi} \leq 1,
\ \tvp(\tau,\bsxi) = 0, \text{ when } \tau \geq 4 \text{ or } \abs{\bsxi} \geq 2,
\end{align*}
and $\tvp(\tau,\bsxi) \in (0,1)$ otherwise.
By making sure that the $\tvp$ decreases fast enough (quadratically, for example)
towards its $0$-level set, one can also guarantee the boundedness of $\abs{D
\tvp}^2/\tvp$. For the corresponding function $\vp$, expressed in the original
coordinates (as defined in \eqref{equ:vp-tvp}), we easily check that, relative to the set $[t_0, T]\times \R^d$,
\begin{align*}
    \set{\vp=1} = \gamma(1;1) \quad \eand \quad \set{\vp=0} = \overline{ \gamma( 4; 2)^c},
\end{align*}
and
that the quantity $\Gamma_{\vp}$ is independent of the choice of $(t_0,\bsx_0)$. Due to $R\leq 1/4$, we have $|D \vp \bsb| \lec \tfrac{1}{R^2} |D\tvp \bsb|$. Therefore $\Gamma_\vp$
satisfies
\begin{equation}\label{equ:Gamma-vp}
\Gamma_{\vp}  \lec  \tfrac{1}{R^2} \Gamma_{\tvp}.
\end{equation}
Finally, the support of $\vp(t,\cdot)$ is a subset of $B_n$, this follows from $\bsx_0\in
    B_{n-1}$,  $\text{supp }\tilde{\vp}(\tau,\cdot)= B_2$, and $R\leq 1/4$.
\begin{lemma}\label{lem:peq}
For $p_\epsilon$ defined in \eqref{equ:p-tau},
there exists a universal constant $C>0$ such
that, with
$\oo{q'} = 1- \oo{q}$ we have
\begin{align}
\label{equ:mit}
 \iint_{\gamma(4;2)} k p_\epsilon &\lec
   (2R\wedge \sqrt{T})^{2-\tfrac{2+d}{q}}.
\end{align}
\end{lemma}
\begin{proof}
By
\eqref{equ:Aronson},
with $t_\epsilon=t_0 - \eps R^2$,
we have
\begin{multline*}
 \iint_{\gamma(4;2)} \peps^{q'} \lec
   \intl{[t_0,t_0+ 4R^2]} (t-t_\epsilon)^{-\tfrac{d}{2}
  (q'-1) }
 \Big(  \intl{\beta(2)} (t-t_\epsilon)^{-\tfrac{d}{2}} e^{-\tfrac{q'|\bsx-
 \bsx_0|^2}
 {2\overline {\sigma}^2 (t-t_\epsilon)}}\, d\bsx\Big)\,  dt\lec \\
   \lec  \intl{[t_0,t_0+4R^2]} (t-t_\epsilon)^{-\tfrac{d}{2}(q'-1)} dt
   \leq \intl{[t_0, t_0+4R^2]} (t-t_0)^{-\tfrac{d}{2}(q'-1)} dt
   \lec
   (2R\wedge \sqrt{T})^{2-\tfrac{d}{q-1}}.
  \end{multline*}
  The previous inequality, combined with H\"{o}lder's inequality in the
  form $\iint k p_\epsilon \leq \norm{k}_{\mathbb{L}^q} (\iint
  p_\epsilon^{q'})^{1/q'}$, establishes the statement.
\end{proof}
Reminding the reader that the constant $c$ is defined in \eqref{eq:abc const}, we state
the following result which is a combination of Lemma
\ref{lem:testing}, applied with the testing function $\vp$ introduced
above, and Lemmas  \ref{lem:weak-Jacobian}, \ref{lem:peq}, together with \eqref{equ:Gamma-vp}.

\begin{proposition} \label{pro:testing}
There exists a universal constant $C>0$ such that for all
$\eps\geq 0$
and all $\bsc$ with $\abs{\bsc}\leq c$,  we have
\begin{equation}\label{equ:Dv est}
     \iintl{\gamma(1;1)} \,\abs{D \bsv}^2 \peps
      \lec
      R^{-2}
     \iintl{ \gamma(4;2) \setminus \gamma(1; 1)}
     \abs{\bsv - \bsc}^2 \peps
      + \inds{ \tilde{T} \leq 4}  \intl{\beta(2)}
      \abs{\bsg - \bsc}^2 \peps( T, \cdot) + (2R\wedge \sqrt{T})^{2-\frac{2+d}{q}}.
\end{equation}
\end{proposition}

\subsection{A weighted Poincar\'e inequality and Struwe's lemma}
Next we state a weighted Poincar\'{e} inequality in Lemma \ref{lem:Por} below. Let $\tD$
be
a Lipschitz domain (nonempty open connected set) in $\R^d$, and let $D$ be its
translate/dilate as described around \eqref{equ:vp-tvp}. Similarly, let
$\tilde{\chi}$
be a \define{weight} function, i.e., such that $\tilde{\chi} \in \linf(\tilde{D})$
and $\int_{\tilde{D}} \tilde{\chi}>0$, and let $\chi$ be its translated/dilated
version.
Given a function $u\in\lone(D)$, we define its $\chi$-\define{average}
\begin{align*}
    \bar{u}^{\chi}_D = \frac{1}{\textstyle \int_D \chi} \int_D u\chi.
\end{align*}
The special case $\chi=1$ is denoted simply by $\bar{u}_D$.
For a vector-valued function $\bsu$, the same notation is used, but with
averaging is applied component-wise. When the domain $D$ is omitted, it is assumed
that $D=\Int\supp \chi$.

In keeping with the notational philosophy of the proof, $\tilde{D}$ and
$\tilde{\chi}$  are thought of as prototypes, and $D$ and $\chi$ as the family
of their homothetic copies, indexed by $t_0,\bsx_0$, and $R$. As above,
the main message behind our results below is that estimates can be made
independently (or explicitly dependently) of those indices. Here,
$\norm{\cdot}$ denotes the $\ltwo$-norm on $D$ and $H^1(D)$ the
 Sobolev space $W^{1,2}$ on $D$.
\begin{lemma}\label{lem:Por} Given
$\tilde{D}$ and $\tilde{\chi}$ as above, there exists a universal constant $C=C
(\tilde{D}, \tilde{\chi})$ such that for all
$u\in H^1(D)$ we have
     \begin{align*}
         \norm{
             u - \bar{u}^{ \chi}_{D}
         }^2
         \lec  R^2 \, \norm
         { Du}^2,
     \end{align*}
     \end{lemma}
\begin{proof}
For $w\in \ltwo(\tD)$, by the Cauchy-Schwarz inequality, we have
   \begin{align*}
   \textstyle
       \norm {w -\tfrac{1}{|\tilde{D}|}\int_{\tD}  w}\, \norm{\tchi - \tfrac{1}{|\tilde{D}|}\int_{\tD} \tchi} \geq
       \abs{ \int_{\tD} w \tchi -
       \tfrac{1}{|\tilde{D}|}\int_{\tD} w \int_{\tD} \tchi} = \abs{\int_{\tD} \tchi} \abs{ \overline
       {w}^{\tchi}_{\tD} - \overline{w}_{\tD}}.
   \end{align*}
  If, additionally, $w\in H^1(\tD)$, then,   combining the previous inequality and Poincar\' e's inequality,
  we have
   \begin{align*}
       \norm{w - \overline{w}_{\tD}^{\tchi}}^2 &\leq
       2(\overline{w}_{\tD} - \overline{w}_{\tD}^{\tchi})^2 +
       2\norm{w - \overline{w}_{\tD}}^2 \leq
       C \norm{ w-\overline{w}_{\tD}}^2\leq
       C \norm{\nabla w}^2,
   \end{align*}
   with $C$ depending only on $\tD$ and $\tchi$. It remains to set
$w(\bsxi) = u(\bsx_0 + R \bsxi)$.
\end{proof}

Next,
let us  pick a (weight function) $\tpsi:\R^d \to [0,1]$
such that
\begin{align}\label{equ:psi}
    \tpsi(\bsxi) =
    0, \efor \abs{\bsxi} \leq \tot\text{ or } \abs{\bsxi} \geq 4, \quad
    \tpsi(\bsxi) = 1, \efor 1\leq \abs{\bsxi} \leq 2,
\end{align}
and
    $\tpsi(\bsxi)\in (0,1)$, otherwise,
and consider its version $\psi$ in the $\bsx$-coordinates.
\begin{lemma}
\label{lem:Por-2}
When $d\geq 2$, there exists a universal constant $C=C(\tpsi)$ such
that, for any $u\in {H}^1_{loc}
(\R^d)$ we have
\begin{align*}
    \int_{\beta(2)\setminus \beta(1)} \abs{u-\overline{u}^{\psi} }^2 \lec
    R^2 \,
    \int_{\beta(4)\setminus \beta(1/2) } \abs{D u}^2.
\end{align*}
The same inequality holds when the domain of the left integral is replaced
by $\beta(2)$ and that of right one is replaced by $\beta(4)$, but $d$ is allowed to be $1$ in this case.
\end{lemma}
\begin{proof}
      With $\tilde{D}$ such that $D= \beta(4)\setminus \beta(1/2)$, we observe
      that $\supp \psi \subseteq D$ and that $\beta(2) \setminus \beta(1)  \subseteq D$. Therefore, applying Lemma
      \ref {lem:Por}, we have
      \[
      \int_{\beta(2)\setminus \beta(1)} \abs{ u - \overline{u}^{\psi}}^2
      \leq
      \int_{D} \abs{ u - \overline{u}^{\psi}_{D} }^2
      \lec    R^2 \int_{D} \abs{ D u}^2.
      \]
      The proof is the same when $\beta(2)\setminus \beta(1)$ is replaced by $\beta(2)$ and $D= \beta(4)$.
\end{proof}
\begin{remark}\label{rem:d=1}
 When $d=1$, the set
 $\beta(4)\setminus \beta(1/2)= [-4, -\tfrac12] \cup [\tfrac12, 4]$ is not connected, and, in fact,
the statement of Lemma \ref{lem:Por-2} does not hold.
To see that, it is enough to consider
$u=1$ on $[-4, -\tfrac12]$ and $u=2$ on $[4, \tfrac12]$. Then
$\int_{\beta(4)\setminus \beta(1/2)} |Du|^2 =0$, but
$1<\overline{u}^\psi<2$ implies $\int_{\beta(2)\setminus
\beta(1)}|u-\overline{u}^\psi|^2>0$.

\smallskip

On the other hand, let us argue that we can assume, without loss of
generality, that $d\geq 2$ in Theorem \ref{thm:abstract}. Indeed, suppose that we have established
Theorem \ref{thm:abstract} for $d\geq 2$, but we are
facing a situation where $d=1$. In this case, we simply embed our
one-dimensional problem into a two-dimensional one. More precisely, we
define the new state process, perhaps on an enlarged probability space, as
$\hat{X}=(X,B)$, where $B$ is a Brownian motion independent of $W$. The so-obtained
coefficients $\hat{b}=(b,0)$ and $\hat{\sigma}=\diag(\sigma,1)$ satisfy all the
necessary assumptions. Furthermore, the functions $\hat{\bsv}^m(t,x,x') =
\bsv^m(t,x)$ and $\hat{\bsw}^m(t,x,x') = (\bsw^m(t,x), \boldsymbol{0})$ is a Markovian solution to the system \eqref{equ:BSDEm} on the enlarged probability space. The similarly defined
$\hat{\bsf}^m$ and $\hat{\bsg}^m$ satisfy the assumptions
of Theorem \ref{thm:abstract}. In particular, for the Lyapunov pair in condition (4), the inequality \eqref{equ:Lyapunov} (with $\bsf$ and $k_n$ replaced by $\bsf^m$ and $k^m_n$ respectively) is satisfied for all $\hat{\bsz} = (\bsz, \boldsymbol{0})\in \R^{N\times 2}$. Therefore Theorem \ref{thm:abstract} implies that $\{\hat{\bsv}^m\}$ is uniformly locally H\"{o}lderian in its first two variables. As a result, Theorem \ref{thm:existence} produces
a locally
H\" olderian solution $(\hat{\bsv}, \hat{\bsw})$ on the extended space. It remains to use
locally uniform convergence and Lemma \ref{lem:weak-Jacobian}
to conclude that $\hat{\bsv}$ does not depend on the
additional coordinate $x'$ and that the second column of $\hat{\bsw}$
vanishes. Therefore $\bsY= \hat{\bsv}(\cdot, \bsX, 0)$ and $\bsZ=\hat{\bsw}^1(\cdot, \bsX, 0)$ are adapted to the original filtration and solves the original system.
\end{remark}

We consider $\tilde{\psi}$ as a global variable for the remainder of the proof.
Consequently, the dependence of universal constants on it will be
suppressed in sequel. The following Lemma generalizes an important result of
Struwe (see \cite[Lemma 4., p.~134]{Str81})
\begin{lemma}
\label{lem:Struwe}
There exists a universal constant $C$ such that
\begin{align*}
  \babs{\overline{\bsv(t_2,\cdot )}^{\psi} -
  \overline{\bsv(t_1,\cdot )}^{\psi}
  }^2
  \lec
   R^{-d} \iint_{ \gamma(\theta_2,4)\setminus (\gamma(\theta_1,4) \cup \gamma
  (\theta_2, 1/2) )} \abs{D\bsv}^2
  +  R^{ 2 - \tfrac{2+d}{q}},
\end{align*}
for
all $ 0 \leq \theta_1 \leq \theta_2 \leq 4$, where
$t_i = t_0+ \theta_i  R^2$, $i=1,2$.
\end{lemma}

\begin{proof} We fix $\eps\in (0,1]$, and,
reminding the reader that $\peps = p(t_0 - \eps R^2, \bsx_0, \cdot,\cdot)$,
set $\eta = \psi/\peps$.
\itos{} formula and Lemma \ref{lem:weak-Jacobian}, applied to the product
$F_t = \eta(t,\bsX_t) v^i(t,\bsX_t)$ (with $\bsX=\bsX^{t_0-\eps
R^2,\bsx_0}$)
yields
\begin{align*}
d (\eta v^i) \eqm - \eta f^i + v^i (\tpddt \eta + \sL \eta)+ \scl{D v^i}{D
\eta}_
{\bsa}.
\end{align*}
Therefore, with
  $L = \EE^{\eps}[ F_{t_2} - F_{t_1}] =
  \Big(
  \overline{v^i(t_2,\cdot )}^{\psi} -
  \overline{v^i(t_1,\cdot )}^{\psi}\Big) \int \psi$ and the
  understanding that all space-time integrals $\iint$ in the rest of the proof
  are over $\gamma(\theta_2,4)\setminus \gamma(\theta_1,4)$
we have
    \begin{align}\label{equ:L}
    L = -
    \iint f^i \psi+
    \iint \Big( \peps v^i (\tpddt\eta + \sL \eta) +
    p_\epsilon \scl{Dv^i}{D\eta}_{\bsa} \Big).
    \end{align}
Since $\bsig$ is bounded
and globally Lipschitz, so is  $\bsa$, and the
infinitesimal generator $\sL$ can be written in a divergence form:
\[
 \sL = \sum_k \tilde{\bsb}_k D_k + \tot \sum_{j,k} D_j (\bsa_{jk} D_k),
\]
where $\tilde{\bsb}_k = \bsb_k - \tot \sum_j D_j(\bsa_{jk})$ is bounded and $D_j(\bsa_{jk})$ is the weak derivative of $\bsa_{jk}$.
Another consequence of the (regularity and ellipticity) assumptions imposed
on $\bsig$ is the fact that the
transition density $\peps$ is smooth for $t> t_0$ and satisfies the
forward Kolmogorov equation
\begin{equation}\label{equ:adj-gen}
  \tpds{t} \peps = \sL^* \peps = - \sum_k D_k (\tilde{\bsb}_k p_\epsilon) + \tot \sum_{j,k} D_k(\bsa_{jk} D_k p_\epsilon).
\end{equation}
Furthermore, since $\psi$ does not depend on $t$, we have
$\peps \tpddt \eta + \eta \tpddt \peps =0$, and, so,
\begin{align}
\label{equ:fi}
  \iint \peps v^i \tpddt \eta
   = -\iint \eta v^i \tpddt \peps = - \iint \eta v^i
   \sL^ * \peps.
\end{align}
Using the divergence form
of $\sL^*$ in \eqref{equ:adj-gen} and the fact that $\eta(t,\cdot)$ is
supported in $\beta(4)$, we conclude that
\begin{align}
\label{equ:se}
  - \int_{\beta(4)} \eta v^i \sL^* \peps =  -\int_{\beta(4)} D(\eta
  v^i)\tilde{\bsb}\, p_\epsilon + \tot \int_{\beta(4)}\scl{ D (\eta v^i)}{D\peps}_
  {\bsa}.
\end{align}
Similarly,
\begin{align}
\label{equ:th}
   \int_{\beta(4)} \peps v^i \, \sL \eta = \int_{\beta(4)}  D \eta\,
   \tilde{\bsb}\, v^i\,  p_\epsilon
  - \tot \int_{\beta(4)}
 \scl{D (\peps v^i  )}{D \eta}_
 {\bsa}.
\end{align}
Finally, we integrate both \eqref{equ:se} and \eqref{equ:th}
over $t\in [t_1,t_2]$, and combine
them with \eqref{equ:fi}, to conclude that
\begin{multline}
\label{equ:parts-p}
    \iint  \Big( \peps v^i ( \tpddt\eta + \sL \eta) + \peps \scl{Dv^i}{D \eta}_
    {\bsa}\Big) = \\
    =\iint \Big(
 \tot \scl{D (\eta v^i)}{D\peps}_{\bsa}
 -\tot \scl{ D(\peps v^i)}{D\eta}_{\bsa} +
 \peps \scl{Dv^i}{D\eta}_{\bsa}  - D(\eta v^i) \tilde{\bsb} p_\epsilon + D \eta \tilde{\bsb} v^i p_\epsilon
\Big) = \\
=  \iint \Big(\tot \scl{Dv^i}{ D \psi}_{\bsa} - \psi D v^i \tilde{\bsb}\Big).
\end{multline}
Next, we multiply both sides of \eqref{equ:L} by $-L$ and use \eqref{equ:parts-p}
together with the uniform ellipticity of $\bsig$
and the fact
that $\abs {f^i} \leq C (\abs {D\bsv}^2 + k)$ to
obtain
    \begin{align*}
        L^2
     &\leq C \Big(
       \abs{L} \iint\abs{\psi} \abs{D\bsv}^2
      +
       \abs{L} \iint\abs{\psi} k +
    \iint (\abs{L D \psi} + \abs{L \psi \tilde{\bsb}}) \abs{D v^i}
    \Big).
    \end{align*}
H\"{o}lder's inequality, applied to the
third term on the right-hand side above, and use the fact that $\abs{L}\leq C
\nvl R^d$ for the first term yield
\[
 L^2 \leq C R^d \Big(\iint\ind{\psi>0} \abs {D\bsv}^2 + \iint \ind{\psi>0}\,
 {k}\Big)
       + \tot L^2 +
       \tot \iint (\abs {D\psi}+ \psi|\tilde{\bsb}|)^2 \iint \ind{\psi>0} \abs{D v^i}^2.
\]
To complete the proof, we use the first inequality in \eqref{equ:f}, apply
H\"older's inequality to the integral
$\iint \ind{\psi>0} k$, and  use the boundedness of $\tilde{\bsb}$ and $R\leq 1/4$ to obtain $\iint \psi^2\abs{\tilde{\bsb}}^2 \lec R^d$ and
  \[
  \int \abs{D\psi}^2 = R^{d-2} \int \abs{D \tpsi} \text{ so that }
 \iint \abs{D\psi}^2 \leq (\theta_2-\theta_1) R^{d} \int \abs{D
 \tpsi}^2\lec R^d.
 \qedhere
 \]
\end{proof}

Coming back to \eqref{equ:Dv est}, we will estimate different terms on the right-hand side using Lemmas \ref{lem:Por-2} and \ref{lem:Struwe}, together with a specific choice of $\bsc$.

\begin{lemma}
\label{lem:consts}
There exists universal constant $C$ such that, for all
$\eps \in (0,\eps_0/2]$ we have
\[
      \sup_{\gamma(4;4)\setminus \gamma(1,1)} R^d
      \abs{\overline{\bsv(t, \cdot)}^\psi - \bsc}^2
     \peps
      \lec   \iintl{\gamma(4;4)\setminus \gamma(4;1/2)}
      |D\bsv|^2 \Big(\tfrac{\peps}{\dl(\eps+\eps_0/2)}
      + \tfrac{\tdl(\eps+\eps_0/2)}{R^{d}}\Big) + R^{2-\tfrac{2+d}{q}},
\]
where $\bsc = \overline{\bsv(t_0+(\tot \eps_0\wedge \tT) R^2 , \cdot)}^\psi$.
\end{lemma}
\begin{proof}
We fix $(t,\bsx) \in \gamma(4;4)\setminus \gamma(1,1)$, set $\tt=(t-t_0)/R^2$ and
define  $\theta_1=\min(\tt,(\eps_0/2) \wedge \tT)$,  and $\theta_2=\max(
\tt,(\eps_0/2)\wedge \tT)$, so that $0\leq \theta_1\leq \theta_2\leq 4$, and
Lemma \ref{lem:Struwe} can be applied. We distinguish  the following two
cases:
% \begin{align}
% \label{equ:St}
%     R^d \babs{\ovrt - \bsc }^2
%     &\lec
%     \iint_{[t_1,t_2]\times
%     (\beta(4)\setminus \beta(1/2))} \abs{D \bsv}^2 + R^{\tfrac{2+d}{q'}},
%     \end{align}

\textbf{Case 1: $\eps_0/2 \leq \tt$.} In this case, $\theta_1= \epsilon_0/2$ and
the estimate \eqref{equ:G2} of Lemma \ref{lem:Gauss} applies. Lemma \ref{lem:Struwe}, with
\eqref{equ:G1} applied to the $p_\epsilon$ outside the integral in the second inequality, yields
\begin{align*} R^d\babs{\ovrt - \bsc }^2 \peps
    &\lec
      \peps \iint_{\gamma(\theta_2;4) \setminus( \gamma(\theta_1;4) \cup \gamma
      (4;1/2))}
     \abs{D \bsv}^2 \tfrac{R^d \peps}{\dl(\eps+\eps_0/2)}  + \peps R^{2+d-\tfrac
     {2+d} {q}} \\
     &\lec R^d \peps
    \oo{\dl(\eps+\eps_0/2)}
     \iint_{\gamma(4,4)\setminus \gamma(4,1/2)}
     \abs{D \bsv}^2 \peps + p_\epsilon R^{2+d-\tfrac{2+d}{q}}\\
     &\lec
  \oo{\dl(\eps+\eps_0/2)}
   \iint_{\gamma(4,4)\setminus \gamma(4,1/2)}
   \abs{D \bsv}^2 \peps + R^{2-\tfrac{2+d}{q}},
\end{align*}

\textbf{Case 2: $\eps_0/2 > \tt $.}
Since $\eps_0\leq 1$, we have $\tt\leq 1/2$. Therefore,
$(\tx)\in\gamma(1/2,4) \setminus \gamma(1/2,1)$, and, so $\bsx \not \in \beta(1)$.
Now that its conditions are met, inequlity \eqref {equ:G3} of Lemma \ref{lem:Gauss},
together with the fact that $\tdl$ is bounded from above and Lemma \eqref{lem:Struwe},
implies that
  \[R^d \abs{ \ovrt - \bsc} \peps \lec \tdl (\eps+\eps_0/2) R^{-d} \iint_{\gamma(4,4)\setminus \gamma(4,1/2)}
     \abs{D \bsv}^2  + R^{2-\tfrac{2+d}{q}}.\qedhere \]
\end{proof}
\begin{lemma}
\label{lem:consts-2}
There exists universal constant $C$ such that, for all $\eps \in (0,\eps_0/2]$,
and all
$t\in [t_0, t_0+4 R^2]$, we have
\begin{align*}
      \int_D \abs{\bsv(t,\cdot) - \ovrt}^2 \peps
      \lec   R^2 \int_{D'} \abs{D \bsv(t,\cdot)}^2
       \Big(\tfrac{\peps}{\dl(\eps+\eps_0/2)}
      + \tfrac{\tdl(\eps+\eps_0/2)}{ R^{d}}\Big),
\end{align*}
where either $(D,D') = (\beta(2)\setminus \beta(1), \beta(4)\setminus \beta(1/2))$,
or $(D,D') = (\beta(2), \beta(4))$.
\end{lemma}
\begin{proof} We omit the details, since the
same strategy as in the proof of Lemma \ref{lem:consts}, namely separating the cases
$\tt \leq \eps_0/2$ and $\tt \geq \eps_0/2$, where $\tt = (t-t_0)/R^2$,
and using the corresponding estimates from Lemma \ref{lem:Gauss},
but this time
together with Lemma \ref{lem:Por-2},  can be applied.
\end{proof}
\begin{lemma}
\label{lem:pre-hole}
Set $\alpha'= \min(\alpha, 1- \oo{q} (1+\tfrac{d}{2}) ) >0$.
 There exists a universal constant $C$ such that, for each $\eps \in (0,\eps_0/2]$
 we have
\begin{equation}\label{equ:hole}
\begin{split}
  &\iintl{\gamma(1;1/2)} \abs{D \bsv}^2 \peps\lec
  R^{2\alpha'}
  + \iintl{\gamma(4,4) \setminus \gamma(1,1/2)}
  \abs{D \bsv}^2 \Big(\tfrac{\peps}{\dl(\eps+\eps_0/2)}  + \tfrac{\tdl
  (\eps+\eps_0/2)}
  {R^{d}}\Big).
\end{split}
\end{equation}
\end{lemma}
\begin{proof}
With $\bsc = \ovps{v(t_0+\tot\eps_0 \wedge \tT) R^2,\cdot)}$, as in Lemma \ref
{lem:consts}, we start from  the inequality
\begin{align}
\label{equ:tri-sq}
\sabs{\bsv(\tx) - \bsc}^2 \peps  \leq 2 \sabs{\bsv(\tx) - \ovrt}^2 \peps  + 2
\sabs
{\ovrt -
  \bsc}^2 \peps.
\end{align}
With $\tt = (t-t_0)/R^2$, we integrate \eqref{equ:tri-sq}
   over $D(t) = \beta(2)\setminus \beta(1)$ when $\tt \in [0,1]$ and
  over
  $D(t)= \beta(2)$ when $\tt \in (1,4]$.
  Thanks to Lemmas \ref {lem:consts} and \ref
  {lem:consts-2} we get, for $t\in [t_0,t_0+4 R^2]$,
  \begin{multline}
  \label{equ:ppp}
R^{-2}\int_{D(t)} \abs{\bsv(t,\cdot) - \bsc}^2 \peps  \lec
\int_{D'(t)} \abs{D \bsv(t,\cdot)}^2
       \Big(\tfrac{\peps(t,\cdot)}{\dl(\eps+\eps_0/2)}
      + \tfrac{\tdl(\eps+\eps_0/2)}{R^{d}}\Big)+ \\ +
  R^{-2}\iintl{\gamma(4;4)\setminus \gamma(1;1/2)}
      |D\bsv|^2 \Big(\tfrac{\peps}{\dl(\eps+\eps_0/2)}
      + \tfrac{\tdl(\eps+\eps_0/2)}{R^{d}}\Big) + R^{-\tfrac{2+d}{q}},
 \end{multline}
 where $D'(t) = \beta(4)\setminus \beta(1/2)$ for $\tt \in [0,1]$ and
 $D'(t) = B (4)$, for $\tt \in (1,4]$. Then we
 integrate \eqref {equ:ppp}
 over $t\in [t_0,t_0+4R^2]$ to obtain
\begin{align}
\label{equ:int-v-c}
R^{-2}\iintl{\gamma(4,2) \setminus \gamma(1,1)} \abs{\bsv - \bsc}^2 \peps
\lec
  \iintl{\gamma(4;4)\setminus \gamma(1;1/2)}
      |D\bsv|^2 \Big(\tfrac{\peps}{\dl(\eps+\eps_0/2)}
      + \tfrac{\tdl(\eps+\eps_0/2)}{ R^{d}}\Big) + R^{2-\tfrac{2+d}{q}}.
\end{align}

Consider, now, the case when $\tilde{T}\leq 4$, i.e., $t_0\geq T - 4 R^2$.
Since $\bsg$ is $\alpha$-H\"{o}lder and
$\supp \psi \subseteq \beta(4)$, we have
$\sup_{\beta(4)} |\bsg-\ogr| \leq \max_{\bsx, \bsx'\in B (4)} |\bsg(x) -\bsg(x')| \lec
R^\alpha$. This inequality and Lemma \ref {lem:Struwe} combined, together with $R\leq 1/4$ together with \eqref{equ:G2} in Lemma \ref{lem:Gauss} applied to the last inequality, imply that
\begin{equation}\label{equ:int-g-c}
\begin{split}
 \int_{\beta(2)} |\bsg-\bsc|^2 \peps(T, \cdot)   & \leq
\int_{\beta(2)}|\bsg-\ogr|^2 p_\epsilon(T,\cdot) +2 |\ogr -\bsc|^2 \\
& \lec R^{2\alpha} + |\ovps{\bsv
 (T,\cdot)}-\ovps{\bsv(t_0+\eps_0 R^2/2,\cdot)}|^2   \\
  &\lec
 R^{2\alpha} + R^{-d}\iintl {[t_0\!+\!\eps_0 R^2/2, T] \times (B (4)\! \setminus\! B
 (1/2))}
\abs{D \bsv}^2 + R^
 {2-\tfrac{d+2}{q}} \\
 &\lec R^{2\alpha'} + \oo{\dl(\eps+\eps_0/2)} \iintl {\gamma
 (4,4)\setminus \gamma (4,1/2)}
\abs{D\bsv}^2  \peps.
\end{split}
\end{equation}
Finally, we combine the estimates \eqref{equ:int-v-c} and
\eqref{equ:int-g-c} with \eqref{equ:Dv est}
(shrinking and extending the domains of integration
appropriately) and use $R\leq 1/4$ to obtain \eqref{equ:hole}.
\end{proof}

\subsection{Hole-filling}
The following technique is so called ``hole-filling" which was first applied to parabolic systems by \cite{Str81}. In the previous subsections, $(t_0, x_0)$ and $R$ are fixed. Now they will be varied in $[0,T]\times B_{n-1}$ and $(0,1/4]$. However the centre of the ball $B_{n-1}$ is still fixed at $b_0$. It is important to note that none of constants $C$ below depends on $(t_0, x_0)$ and $R$.
\begin{lemma} There exists an universal constant $C$
\label{lem:hol-fil}
such that
 \begin{align}
 \label{equ:fil-hol}
   \sup_{\eps\in (0,\eps_0/4]} \iint_{\gamma(1/4;1/2)} \abs{D \bsv}^2 \peps \leq
   \kappa
   (\eps_0) \sup_{\eps' \in (0,16\eps_0]}
  \iint_{\gamma(16;4)}\abs{D\bsv}^2 p_{\eps'}
 +C R^{2\alpha'}
 \end{align}
where $\alpha'= \min(\alpha, 1- \oo{q} (1+\tfrac{d}{2}) ) >0$.
  and $\kappa(\eps) = (1 + \tot \dl(\eps))/(
  1+\dl(\eps))$.
  \end{lemma}
\begin{proof}
Inequality \eqref
 {equ:G2} of Lemma \ref{lem:Gauss}
 yields $p_{\eps_0} \gec R^{-d}$ on $\gamma (4;4)$.
  Given $\eps \in (0,\eps_0/2]$,
 this inequality, combined with
 \eqref{equ:hole}, yields
 \begin{align}
 \label{equ:hf}
   &\iintl{\gamma(1;1/2)} \abs{D \bsv}^2 \peps \lec R^{2\alpha'} +
   \tfrac{1}{\dl(\eps+\eps_0/2)} \iintl{\gamma(4,4) \setminus \gamma(1,1/2)} \abs{D
  \bsv}^2  \peps +
  \tdl(\eps+\eps_0/2) \iintl{\gamma(4;4)}\abs{D\bsv}^2 p_{\eps_0},
 \end{align}
Let  $C_0$ denote the constant $C$ from \eqref{equ:hf}; we assume, without
loss of generality, that $C_0\geq 1$.
 Adding $\tfrac{C_0}{\dl(\eps+\eps_0/2)} \iint_{\gamma(1;1/2)} \abs{D \bsv}^2 \peps$
 to both sides of \eqref{equ:hf} and
 dividing throughout by $1+\tfrac{C_0}{\dl(\eps+\eps_0/2)}$ yields
 \begin{align}
 \label{equ:hf2}
   \iintl{\gamma(1;1/2)} \abs{D \bsv}^2 \peps &\leq C_0 R^{2\alpha'} + \kappa'
   (\eps+\eps_0/2) \Big(
    \iintl{\gamma(4,4)} \abs{D \bsv}^2  \peps +
  \iintl{\gamma(4;4)}\abs{D\bsv}^2 p_{\eps_0} \Big)
 \end{align}
where $\kappa'(\eps) = (1 + \dl \tdl)/(C_0 + \dl)$. Our choice of the constant
$\eps_0$ implies that $\kappa'(\eps) \leq \kappa(\eps)$, for $\eps\leq \eps_0$.
Moreover, $\kappa$ is strictly decreasing on $[0,\eps_0]$, so
$\kappa(\eps_0/2) \geq \kappa(\eps+\eps_0/2)$, for
$\eps \in (0,\eps_0/2]$. Therefore, extending domains on the right-hand side and shrink domains on the left-hand side, we obtain
 \begin{align*}
   \iintl{\gamma(1/4;1/2)} \abs{D \bsv}^2 \peps &\leq
   \iintl{\gamma(1;1/2)} \abs{D \bsv}^2 \peps \leq  C_0 R^{2\alpha'} + \kappa
   (\eps_0/2) \sup_{\eps' \in (0,16\eps_0]}
  \iint_{\gamma(16;4)}\abs{D\bsv}^2 p_{\eps'}.
 \end{align*}
Maximizing over $\eps \in (0,\eps_0/4]$ on the left-hand side completes the argument.
 \end{proof}
 \begin{proposition}
 \label{pro:hol-fil-2}
 There exists universal constants $C,\alpha_0>0$ such
 that
 \begin{align*}
 \sup_{R\in (0,1/4]} R^{-d-2\alpha_0} \iint_{\gamma(1;1)} \abs{D \bsv}^2 \leq C.
 \end{align*}
 \end{proposition}
\begin{proof}
 In this proof, we need to vary $R$ and do not consider it fixed, while we
 still keep $(\txz)$ fixed.
 Hence, we include explicit dependence  on $R$
 in the notation as in, e.g., $\gamma_{R}(1,1)$.
 Given $\alpha_0 \in (0,\alpha']$, with $\alpha'$ as in Lemma
 \ref{lem:hol-fil}, we define
\begin{align*}
  \ld_{\alpha_0}(R) = \sup_{\eps \in (0,\eps_0]} R^{-2\alpha_0} \iint_
  {\gamma_{R}
  (1,1)} \abs {D \bsv}^2 p(t_0-\eps R^2, \bsx_0, \cdot,\cdot)
 \end{align*}
Lemma \ref{lem:hol-fil} implies that there exists a universal constant $C_0>0$
such that
 \begin{equation}\label{equ:ld}
 \ld_{\alpha_0} (\tot R) \leq C_0 + \nu(\alpha_0)\, \ld_{\alpha_0}
 (4R), \efor R\in(0,1/4],
 \end{equation}
 where, with $\kappa<1$ as in Lemma \ref{lem:hol-fil}, we have
$\nu(\alpha_0) = 8^{2\alpha_0} \kappa( \eps_0)$.
Choosing $0<\alpha_0\leq \alpha'$ small enough so that
$\nu_0= \nu
(\alpha_0)<1$, we obtain
\begin{align}
 \label{equ:varphi2}
 \ld_{\alpha_0} (\tot R) \leq C_0 + \nu_{0}\, \ld_{\alpha_0}(4R),
 \efor R\in(0, 1/4].
\end{align}
On the other hand, Proposition \ref{pro:testing} together with the boundedness of $\bsv$ imply that
$\ld_{\alpha_0}(\cdot)$ is bounded on compact segments of $(0,\infty)$. This and \eqref{equ:varphi2} combined
 yield
\begin{align}
\label{equ:rbigb}
\sup_{R\leq 1/4} \ld_{\alpha_0}(R) \leq C_1,
\end{align}
for some universal constant $C_1$.
 The statement then follows from specializing the supremum in the
 definition of $\vp$ to $\eps=\eps_0$ and estimating $p_{\eps_0}$ using
 \eqref {equ:G1} of
 Lemma \ref{lem:Gauss}.
\end{proof}

The following result finishes the proof of Theorem \ref{thm:abstract}.
\begin{corollary}[Uniform $C^{\alpha}$-bounds]
\label{cor:ca-nd}
  There exists a universal constants $C$ and $\alpha_0>0$ such that
  \[ [\bsv]_{\alpha_0; B_{n-1}} \leq C.\]
\end{corollary}
\begin{proof}
Keeping $(\txz) \in [0,T]\times B_{n-1}$ and $R\leq 1/4$ fixed,
we set
\[
\bsc= \tfrac{1}{R^2}\int_{t_0}^{t_0+R^2} \overline{\bsv(t, \cdot)}^\psi
dt,\]
so that
\begin{equation}\label{equ:v-c-cam}
 \int_{\beta(1)} \abs{\bsv(t, \cdot) - \bsc}^2 \lec \int_{\beta(1)} \abs{\bsv(t,
 \cdot) - \overline{\bsv(t, \cdot)}^\psi}^2 +
 R^d \abs{\overline{\bsv(t, \cdot)}^\psi - \bsc}^2.
\end{equation}
Applying Lemmas \ref{lem:Por-2} and \ref{lem:Struwe} to the two terms on the
right-hand side respectively, we obtain
\begin{align*}
 &\int_{\beta(1)} \abs{\bsv(t, \cdot) -\overline{\bsv(t,\cdot)}^\psi}^2 \lec
 R^2 \int_{\beta(4)} \abs{D\bsv(t, \cdot)}^2, \eand\\
 & R^d \abs{\overline{\bsv(t, \cdot)}^\psi-\bsc}^2 \lec \iint_{\gamma(1;4)}
 \abs{D \bsv}^2 + R^{d+ 2- \tfrac{2+d}{q}},
\end{align*}
so that, an integration of \eqref{equ:v-c-cam} over $[t_0, t_0+ R^2]$ yields
\[
 \iint_{\gamma(1;1)}\abs{\bsv-\bsc}^2 \lec R^2 \iint_{\gamma(1;4)} \abs{D \bsv}^2 + R^{2+d + (2-\tfrac{2+d}{q})}.
\]
Dividing both sides by $R^{d+2+2\alpha_0}$, where $\alpha_0$ is from  Proposition
\ref{pro:hol-fil-2}, and using the same proposition on
the right-hand side, we obtain a universal constant $C$ such that
 $\iint_{\gamma(1;1)} \abs{\bsv-\bsc}^2 \lec
R^{d+2+2\alpha_0}$ for all $R\leq 1/4$.
Finally, $\bsc= \overline{\bsv}_{\gamma(1;1)}$ minimizes the integral
$\iint_{\gamma(1;1)} |\bsv-\bsc|^2$, and, thus, we have
\begin{align}
\label{equ:Camp}
\sup_{R\in (0,1/4]} R^{-d-2-2\alpha_0}\iint_{\gamma(1;1)} \abs{\bsv-\overline{\bsv}_{\gamma(1;1)}}^2 \leq C.
\end{align}
The constant $C$ of \eqref{equ:Camp} above does not depend on
$(\txz)\in [0,T]\times B_{n-1}$, so $\bsv$ belongs to the ball of radius $\sqrt{C}$ in the
Campanato space
$\hat{C}^{\alpha_0}([0,T]\times B_{n-1})$, where
\[
 \hat{C}^{\alpha_0}([0,T]\times B_{n-1})
 := \big\{\bsv\in \ltwo\,:\hspace{-2em} \sup_{(t_0, \bsx_0)\in
 [0,T]\times B_{n-1}, R\in (0,1/4]}
\hspace{-2em}
 R^{-d-2-2\alpha_0}\iintl{\gamma_{\txz,R}(1;1)}
 \abs{\bsv-\overline{\bsv}_{\gamma_{\txz,R}(1;1)}}^2<\infty\big\}.
\]
The (topological) equivalence of the Campanato space
$\hat{C}^{\alpha_0}([0,T]\times B_{n-1})$ with the natural metric,
and the H\" older spaces $C^{\alpha_0}([0,T]\times B_{n-1})$ (see, e.g., \cite[IV.2, p. 49]{Lie96})
implies that $[\bsv]_{\alpha_0; B_{n-1}}$ admits a universal bound.
\end{proof}

\section{Additional proofs}

\subsection{Proof of Theorem \ref{thm:existence}}\label{subsec:convergence}
Thanks to our notational convention at the beginning of Section
\ref{sec:uniform est}, the index $m$ was suppressed in the statement
of Corollary \ref{cor:ca-nd}. The dependence on $k^m_n$ is through its $\mathbb{L}^q$-norm on $[0,T]\times B_n(b_0)$, which is assumed to be bounded uniformly in $m$. With the conditions (1)-(4) of Theorem
\ref{thm:abstract} holding uniformly in $m$,  we have a universal constant $C$ such that $[\bsv_m]_{\alpha_0;
B_{n}(b_0)} \leq C$, for all $m$. Combining this uniform H\"{o}lder estimate and
the uniform bound in condition (2) of Theorem \ref{thm:abstract}, we apply
Arzel\' a-Ascoli theorem on $[0,T]\times B_n(b_0)$ to extract a subsequence of
$\set{\bsv^m}$ which converges uniformly. A diagonal procedure then produces
another subsequence - still denoted by $\set{\bsv^m}$, as well as a continuous
function $\bsv:\TR\to \R^N$ such that
$\bsv^m \to \bsv$, locally uniformly. Thanks
to the preservation of H\" older
continuity under uniform convergence, the function $\bsv$ belongs to the
local H\" older space $\Caa_{loc, b_0}$, for some sequence $\seq{\alpha'}$ in $(0,1]$.

Having picked and fixed $n\in\N$ and the initial condition $(\tx)\in \TR$,
we set $\bsX=\bsX^{\tx}$ and define the exit time
\[ \tau_n = \inf\sets{ u\geq t}{ \bsX_u \not \in B_n(b_0)},\]
as well as the following two sequences of processes
\[ \bsYm_u = \bsv^m(u,\bsX^{\tau_n}_u) \quad \eand \quad \bsZm_u =
\bsw^m(u,\bsX_u)\inds{u < \tau_n}, \quad u\in [t,T].\]
Since $(\bsv^m, \bsw^m)$ is a Markovian solution to the system
\eqref{equ:BSDEm}, the process
$\bsY^m$ is a semimartingale whose finite-variation part is given by
\[-\int_t^{\cdot}
\bsf^m(u,  \bsX^{\tau_n}_u, \bsY^{(m)}_u, \bsZm_u) \inds{u\leq \tau_n}
\, du.\]
Condition (3) in Theorem \ref{thm:abstract} and Lemma
\ref{lem:Z-bound} imply that these, finite-variation, parts admit a uniform
bound in total variation, i.e.,
\begin{multline}
\label{equ:bdvar}
\EE^{\tx}\int_t^{T} \Big[\abs{\bsf^m(u,\bsX^{\tau_n}_u, \bsY^{(m)}_u, \bsZm_u)} \inds{u\leq
\tau_n}
\, du\Big] \leq_{C(n)} \\ \leq_{C(n)}
\EE^{\tx}\Big[\int_t^T \Big(\abs{\bsZm_u}^2 + k_n(u,\bsX_u)\Big) \inds{\bsX_u \in
B_n(b_0)}\, du\Big] \leq C(n).
\end{multline}
Moreover, the uniform convergence of
$\bsv^m$ on $B_n(b_0)$ implies that the convergence
$\bsYm \to \bsY^n=\bsv(\cdot, \bsX^{\tau_n})$ is also uniform.

Uniform ellipticity of $\bsig$ and
\itos{} formula applied to $|\bsYm-\bsYmp|^2$ yield
\begin{multline}
\label{equ:cau}
   \EE^{\tx}\Big[ \int_t^{T} \abs{\bsZm_u- \bsZmp_u}^2\, du\Big] \lec
   \norm{\bsYm_{\tau_n}-\bsYmp_{\tau_n}}_{\linf}^2  + \\ +
   \big\| \sup_{u\in[0,\tau_n]}|\bsYm_{u}-\bsYmp_{u}|\big\|_{\linf} \EE^{\tx}\Big[ \int_t^T
   \abs{\bsf^m_u}+
   \abs{\bsf^{m'}_u}\, du\Big],
\end{multline}
where $\bsf^m_u = \bsf^m(u,\bsX^{\tau_n}_u, \bsY^{(m)}_u, \bsZ^{(m)}_u)
   \inds{u\leq \tau_n}$ for all $m,m' \in \N$.  The uniform bound in
   \eqref{equ:bdvar} implies now that the sequence $\set{\bsZm}$ is Cauchy in
   $\ltwo$ uniformly for $(t,\bsx)\in [0,T]\times B_n(b_0)$, with limit
   $\bsZ^n$.  A subsequence, still labeled $\set{\bsZm}$,
   converges $\text{Leb}\otimes \PP$-a.e.~towards the same limit,
   \[ \bsw^m(u,\bsX_u)\inds{u\leq \tau_n} \to \bsZ^n_u,\]
    uniformly for $(t,\bsx)\in [0,T]\times B_n(b_0)$.
   So it follows that
   \[ \bsZ^n_u = \bsZ_u \inds{u\leq \tau_n},
   \text{Leb}\otimes \PP-\text{a.s.,}\]
 where
   \[ \bsZ_u = \bsw(u,\bsX_u) \quad \text{and}\quad
 \bsw(t', \bsx') = \liminf_m \bsw^m(t', \bsx') \text{ componentwise.}\]
 Therefore, by \itos{} isometry, for almost all $t'\geq t$ we have
 \[
 \int_t^{\tau_n \wedge t'} \bsZ^{(m)}_u \bsig(u,\bsX_u)\, d\bsW_u \to
 \int_t^{\tau_n \wedge t'} \bsZ_u \bsig(u,\bsX_u)\, d\bsW_u,\text{
   a.s.}\]
Next, we show that
\begin{align}
\label{equ:fbc}
\int_t^{\tau_n\wedge t'} \bsf^m(u, \bsX_u, \bsY^{(m)}_u, \bsZ^{(m)}_u)\, du \to
\int_t^{\tau_n\wedge t'} \bsf(u, \bsX_u, \bsY_u, \bsZ_u)\, du, \text{ a.s.}
\end{align}
as $m\to\infty$, for almost all $t'\geq t$. For that, we first observe that,
thank to the assumptions placed on the convergence $\bsf^m\to\bsf$, we have
\[ \bsf^m(u,\bsX_u, \bsY^{(m)}_u, \bsZ^{(m)}_u) \inds{u\leq \tau_n} \to \bsf(u,\bsX_u, \bsY_u,
\bsZ_u) \inds{u \leq \tau_n}, \ld\otimes \PP-\text{ a.e.}\]
This is, however, enough to ensure the $\ld\otimes \PP$-convergence, which,
in turn, implies \eqref{equ:fbc}. Indeed, we have
\[
\abs{\bsf^m(u,\bsX_u,\bsY^{(m)}_u, \bsZ^{(m)}_u) - \bsf(u,\bsX_u,\bsY_u\bsZ_u)}\inds{u\leq
\tau_n} \lec \big( \abs{\bsZ^m}^2+\abs{\bsZ}^2+k_n(t,\bsX_u) \big)
\inds{u\leq \tau_n},
\]
with the right-hand side $\text{Leb}\otimes \PP$-uniformly integrable, thanks to
the $\ltwo(\text{Leb}\otimes \PP)$-convergence of $\bsZ^{(m)}$.

It is straightforward now to let $n\to\infty$ and conclude that the pair
$(\bsv,\bsw)$ is a Markovian solution to \eqref{equ:BSDE}. To show that
$\bsw = D \bsv$ in the weak sense, we simply note that the proof of Lemma
\ref{lem:weak-Jacobian} applies verbatim.

\subsection{Proof of Theorem \ref{thm:unique}}
We start with a uniform $\bmo$ estimate which will
also be used in the proof of uniqueness.
For a Borel
function $\bsw:\TR\to\R^{N\times d}$ and a constant $\delta>0$, we define its $\bmo(\delta)$-norm by
\[ 
\norm{\bsw}^2_{\bmo(\delta)} = \sup_{t\in [\delta, T]} \sup_{t-\delta \leq \tau \leq t}
\Big\|
\EE_\tau\Big[ \int^t_\tau \abs{\bsw(u,\bsX_u)}^2\, du\Big]\Big\|_{\linf}, \]
where $\tau$ is any stopping time taking value in $[t-\delta, t]$.
We say that $\bsw\in \umo$ if $\lim_{\delta\downto 0}
\norm{\bsw}_{\bmo(\delta)}=0$; this, stronger, notion of
$\bmo$-regularity will play a role in the uniqueness proof below. We start with a well-known
estimate whose proof we include for the reader's convenience:
\begin{lemma}
\label{lem:X-tau}
For any $t\in [\delta, T]$ and any stopping time $\tau$ taking value in $[t-\delta, t]$, 
and $\alpha\in (0,1]$, we have
\begin{align}
\label{equ:X-tau}
\EE_\tau[ \abs{\bsX_{t}-\bsX_\tau}^{\alpha}] \leq C\,
 \delta^{\alpha/2},
\end{align}
where $C$ depends only on $\alpha$, $d$, $\norm{\bsb}_{\linf}$ and $\norm{\bsig}_{\linf}$.
\end{lemma}
\begin{proof}
 Using the Burkholder-Davis-Gundy inequality, we obtain
 \begin{align*}
  \EE_\tau[\abs{\bsX_t-\bsX_\tau}] &\lec
  \EE_\tau\Big[\Big|\int_{\tau}^{t}\bsb(u,\bsX_u) du\Big|\Big]
  +  \EE_\tau \Big[\Big( \int_{\tau}^{t} \abs{\bsig(u,\bsX_u)}^2 du\Big)^{\tot}\Big]   \lec
    \delta^{1/2}.
   \end{align*}
The inequality \eqref{equ:X-tau} now follows from
the fact that
 $\EE_\tau[\abs{\bsX_{t}-\bsX_\tau}^{\alpha}]
 \leq \EE_\tau[ \abs{\bsX_{t}-\bsX_\tau}]^{\alpha}.$
\end{proof}
\begin{proposition}
\label{pro:unif-bmo} Suppose that, for some $c>0$,
there exists $(h,k)\in
\Ly(\bsf,c)$ with $k\in\linf$.
Then $\bsw\in \umo$ for
any locally H\" olderian solution $(\bsv,\bsw)$
to \eqref{equ:BSDE1} with $\bsv\in \Caa_{loc}$ and 
$\norm{\bsv}_{\linf}\leq c$.
%Then, for all pairs $\tau\leq
%\tau'$ of stopping times in $[0,T]$, we have
%\[ \EE_{\tau} \left[ \int_{\tau}^{\tau'} \abs{\bsw(u,\bsX_u)}^2\,
%du \right] \leq C \EE_{\tau}[\tau'-\tau]^{\alpha/2},\text{ a.s.,}\]
%where $\alpha$ is the H\" older constant of $\bsv$ and $C$ depends only on
%$\norm{\bsv}_{\Caa}$, $T$,
%$M$, $c$, $\alpha$ and $h$.
 \end{proposition}
\begin{proof}
Given $(h,k) \in \Ly(\bsf,c)$ and $t-\delta \leq \tau \leq t$, we
apply \itos{} formula to $h(\bsY_u)$, where $\bsY_u=\bsv(u,\bsX_u)$.
With the boundedness of $\bsv$ and a localization argument guaranteeing that the expectations of the
local-martingale parts vanish, we obtain
\begin{align*}
    \EE_\tau\Big[ h(\bsv(t,\bsX_{t})) - h(\bsv(\tau,\bsX_\tau))\Big] \geq
   \EE_\tau\left[ \int_{\tau}^{t} \abs{\bsw(u,\bsX_u)}^2\, du\right] -
   M \delta,
\end{align*}
where $M$ is an upper bound for $k$. 

To derive a $\umo$-estimate, let $L$ be the Lipschitz constant of the
function $h$ on $B_c$. Since $\bsv\in \Caa_{loc}$, for any given $n$, there exists a constant $C_n$ such that 
\[
 |\bsv(t,\bsx) - \bsv(t', \bsx')| \leq C_n \max\{|\bsx- \bsx'|^{\alpha_n}, |t'-t|^{\alpha_n/2}\}, 
\]
for any $t,t'\in [0,T]$ and $|\bsx -\bsx'|\leq n$. The Markov inequality
coupled with Lemma \ref{lem:X-tau} then imply that
\begin{multline*}
    \EE_\tau\Big[ h(\bsv(t,\bsX_{t})) - h(\bsv(t,\bsX_\tau))\Big]
    \leq L\,
 \EE_\tau[ \abs{\bsv(t,\bsX_{t})-\bsv( t, \bsX_\tau)}] \\
 \lec \EE_\tau[\abs{\bsv(t, \bsX_{t}) - \bsv(t, \bsX_\tau)}
 \ind{\abs{\bsX_{t}-\bsX_\tau}\leq n} ] + 2 \norm{\bsv}_{\linf}
 \PP_\tau[\abs{\bsX_{t} - \bsX_\tau} > n]\\
 \lec
 \EE_\tau[ \max( \abs{\bsX_{t}-\bsX_\tau}^{\alpha_n},
 (t-\tau)^{\alpha_n/2})] + (t-\tau)^{1/2}
 \lec \delta^{\alpha_n/2}.
 \end{multline*}

 The statement then follows from
 combining above displayed estimates. \qedhere
\end{proof}

The uniqueness part of the proof is based on a result \cite[Proposition
2.1]{Frei-splitting} of Frei, which, in turn extends \cite[Proposition
1]{Tevzadze} from BSDE whose generator does not depend on $\bsy$ and
the terminal condition is small in
$\mathbb{L}^\infty$-norm,  to those whose terminal condition is small in
the $\BMO$-norm (see,
also, \cite[Theorem A.1]{Kramkov-Pulido} for a similar result). We now work with $\bsf$ which does not depend on $\bsy$ and derive a consequence of Proposition \ref{pro:unif-bmo} above.

\begin{corollary}\label{cor:lin-syst}
Let $\bsF:\TR\to \R^N$ be continuous and bounded,
and let $\bsg \in \Ca_{loc} \cap \linf$.
The linear system
 \begin{equation}\label{equ:BSDE-linear}
  d\bsY_t = -\bsF(t, \bsX_t)\, dt + \bsZ_t
  \bsig_t d\bsW_t,\ \bsY_T =\bsg(\bsX_T),
 \end{equation}
 admits a solution $(\bsv,\bsw)$, which is unique in the class of bounded
 solutions.
 Furthermore, $\bsv\in \Caa_{loc}$ and $\bsw\in \umo$.
\end{corollary}

\begin{proof}
Let
$\set{\bsf^m,\bsg^m}$ be a sequence of smooth approximations obtained by
mollification of the functions $\bsF$ and $\bsg$, respectively. This
sequence of approximation does not depend on $(\bsy, \bsz)$ and can be constructed so that $\norm{\bsf^m}_{\linf}
+ \norm{\bsg^m}_{\linf} \leq 1 + \norm{\bsF}_{\linf} + \norm{\bsg}_{\linf}$
for all $m$, with $\set{\bsg^m}$ bounded in $\Caa_{loc}$ (cf. Proposition \ref{pro:mollification} below).

Thanks to their boundedness and independence of $\bsz$, these functions are
easily seen to satisfy the conditions of Theorem \ref{thm:abstract}. In
fact, they admit a common $c$-Lyapunov pair for any $c$ -
indeed, it is enough to choose a quadratic $h$ and large-enough constant $k$.
Thanks to the Lipschitz continuity of its coefficients,  the equation
 \begin{equation}\label{equ:BSDE-linear-m}
  d\bsY_t = -\bsf^m(t, \bsX_t)\, dt + \bsZ_t
  \bsig_t d\bsW_t,\quad \bsY_T =\bsg^m(\bsX_T),
 \end{equation}
admits a continuous Markovian solution $(\bsv^m, \bsw^m)$ for each $m$, with
$\set{\bsv^m}$ uniformly bounded.  Therefore, by
Theorem \ref{thm:existence}, there exists a locally H\" olderian solution to
\eqref{equ:BSDE-linear}, i.e., a Markovian solution $(\bsv,\bsw)$ with
$\bsv\in \Caap_{loc, b_0}$, for some $b_0\in \R^d$. Moreover, since a (global) Lyapunov pair $(h,k)$ exists with $k$ bounded, and the H\"{o}lder norm of $\bsg$ does not depend on $b_0$, the last statement of Theorem \ref{thm:abstract} implies that the H\"{o}lder norm of $\bsv$ does not depend on $b_0$, i.e., $\bsv\in \Caa_{loc}$.

It is straightforward to see that this solution is unique in the class of
all bounded solutions. Moreover, thanks to the existence of a
Lyapunov pair mentioned above, the conditions of Proposition
\ref{pro:unif-bmo} are satisfied, and, so, $\bsw\in \umo$.  \end{proof}

%For a measurable stochastic process $\bsZ$, we define
%\[ \norm{\bsZ}_{\bmo[t,T]} = \esssup_{\tau \in [t,T]}
%\EE_{\tau}\left[ \int_{\tau}^T \abs{\bsZ_u}^2\, du\right], \]
%where the essential supremum is taken over all stopping times with values
%in $[t,T]$.
%Our second result, namely Proposition \ref{pro:unif-BMO} provides a uniform BMO ``tail'' estimates on H\" older
%Markovian solutions. Before we state it, we need the two simple
%preparatory lemmas:

To complete the proof of Theorem \ref{thm:unique}, we pick that
we pick two bounded-$\bsv$ continuous solutions $(\bsv, \bsw)$
and $(\bsv', \bsw')$. By Remark \ref{rem:a priori} part (2), both of them
are locally H\"{o}lderian. Since a (global) Lyapunov pair $(h,k)$
exists with $k$ bounded, and the H\"{o}lder norm of $\bsg$ does not depend
on $b_0$, the last statement of Theorem \ref{thm:abstract} implies that the
H\"{o}lder norms of $\bsv$ and $\bsv'$ do not depend on $b_0$, i.e., that
$\bsv, \bsv'\in \Caa_{loc}$. (We can assume, without loss of generality,
that they
both belong to some $\Caa_{loc}$, with the same exponent sequence
$\seq{\alpha}$.)  We define $t_0\in [0,T]$ by
\[ t_0= \inf\sets{t \in [0,T]}{ \bsv(u,\bsx)=\bsv'(u,\bsx)\text{ for all }
\bsx\in\ \R^d, u \in [t,T]}. \]
Let us assume - contrary to the conclusion of the theorem - that $t_0>0$.
When restricted to $[0,t_0]$, both $(\bsv,\bsw)$ and $(\bsv', \bsw')$ are
bounded Markovian solutions to \eqref{equ:BSDE1} with the terminal
condition $\bsg=\bsv(t_0,\cdot)= \bsv'(t_0,\cdot)$. They differ, however,
on each interval of the form $[t_0-\delta,t_0]$, $\delta>0$.

Let be
$(\bsv_f,
\bsw_f)$,
be the unique bounded solution
to the auxiliary equation
\eqref{equ:BSDE-linear} on $[0,t_0]$ with
$\bsF= \bsf(\cdot, \cdot, 0)$ and
$\bsg = \bsv(t_0,\cdot)$.
The conditions of Corollary \ref{cor:lin-syst} above
are satisfied, so  we have
\begin{align}
\label{equ:Frei-cond}
 \lim_{\delta\downto 0} \sup_{t_0 -\delta \leq \tau \leq t_0} \EE_\tau
\Big\|\Big[ \int_{\tau}^{t_0}
\abs{\bsw_f(u,\bsX_u)}^2\, du\Big] \Big\|_{\linf}=0.
\end{align}
At this point, everything is ready for the application of the
aforementioned local uniqueness result
of Frei, which we
summarize for the reader's convenience:
when the quantity \[\sup_{t_0-\delta \leq \tau \leq t_0} \Big\|\EE_\tau \Big[\int_{\tau}^{t_0} |\bsw_f(u, \bsX_u)|^2 du\Big]\Big\|_{\linf},\] which is the $\BMO$ norm of the terminal condition $\bsg$ on $[t_0-\delta, t_0]$, is small, solution $(\bsv, \bsw)$ to \eqref{equ:BSDE1} is unique in a class $\sC$ of $(\hat{\bsv}, \hat{\bsw})$ with the $\bmo$-norm of $\hat{\bsw}$ on $[t_0-\delta, t_0]$, i.e., the quantity \[\sup_{t_0-\delta \leq \tau \leq t_0} \Big\|\EE_\tau \Big[\int_\tau^{t_0} |\hat{\bsw}(u, \bsX_u) \sigma(u, \bsX_u)|^2 du\Big]\Big\|_{\linf},\] is sufficiently small.

Thanks to \eqref{equ:Frei-cond}, Frei's result applies when $t \in [0,t_0)$
is chosen close enough to $t_0$. By making it even closer, if necessary, we
can use Proposition \ref{pro:unif-bmo} to make sure that both of our
solutions $(\bsv,\bsw)$ and $(\bsv',\bsw')$ belong to the class $\sC$.
Therefore, thanks to the fact that $\bsX_u$ has a full support under
$\PP^{\tx}$ for $t < u \leq t_0$,
we conclude that $\bsv(u,\cdot) =
\bsv'(u,\cdot)$ for each $t<u\leq  t_0$ - a contradiction with our
definition of $t_0$. To show that $\bsw=\bsw'$, a.e., we simply appeal to
Lemma \ref{lem:weak-Jacobian} above.

\subsection{Proof of Theorem \ref{thm:sufficient}}
Our proof of Theorem \ref{thm:sufficient} proceeds in two steps. In the
first step, we construct a sequence of Lipschitz approximations to the generator
$\bsf$ and the terminal condition $\bsg$, making sure there is enough
uniformity for the construction of a uniform Lyapunov pair.
Next, we observe that those approximation satisfy the condition
(AB) or (wAB), producing a uniform, a-priori bound in $\linf$. Lastly, we
apply the approximation Theorem \ref{thm:abstract}.

\subsubsection{Lipschitz approximations} We start by outlining a
Lipschitz-approximation procedure that will be used in the sequel.  We
extend slightly the notation for the class of functions satisfying the
condition (BF) from Definition \ref{def:BF}, by including a general, but small,
quadratic term; its significance is explained in Remark \ref{rem:BF} and
the additional term corresponds to the `error' in \eqref{equ:BF-approx}.
If a function $\bsf: [0,T]\times \R^d\times \times \R^{N} \times \R^{N\times
d}\rightarrow \R^N$ admits the following decomposition
\begin{align}
\label{equ:BF-decomp}
 \bsf(t,\bsx,\bsy, \bsz) = \text{diag}(\bsz \bsfl(t,\bsx,\bsy, \bsz)) +
 \bsfq(t,\bsx,\bsy, \bsz) + \bsfs(t,\bsx,\bsy,\bsz) + \bsfe(t,\bsx,\bsy,\bsz) +
 \bsfk(t,\bsx),
\end{align}
where $\bsfl, \bsfq, \bsfs, \bsfk$ satisfy the conditions of Definition
\ref{def:BF},
and for each $n\in\N$ we have \[ |\bsfe(t, \bsx,\bsy,\bsz)|\leq \epsilon_n
(1+|\bsz|^2), \text{  for some } \epsilon_n>0
\text{ and all } (t,\bsx,\bsz)\in [0,T]\times B_n(b_0)\times \R^N\times \R^{N\times d}, \]
then we say that $\bsf$ satisfies
the \define{approximate condition (BF)}, and write
$\bsf\in \BFa(\seq{C},\seq{\kappa}, \seq{q}, \seq{\eps})$.

As Proposition \ref{pro:mollification} below shows, a pleasant feature of the
condition (BF) (and its approximate version) is that it allows for
approximation by more regular functions, in a uniform way. More precisely, $\bsf$ can be approximated by a sequence of regular functions $\{\bsf^m\}$, such that, even though the functions $\bsfl,
\bsfq, \bsfs, \bsfk$ and $\bsfe$ in the decomposition of $\bsf^m$ may depend on $m$, the constant sequences $(\seq{C}, \seq{\kappa}, \seq{q}, \seq{\eps})$ do not. This uniformity is essential to construct a sequence of universal Lyapunov functions $\seq{h}$ for the approximation sequence $\{\bsf^m\}$.

\begin{proposition}[Approximations preserving the approximate condition
(BF)]\label{pro:mollification}\
\begin{enumerate}
\item For each $\bsg\in \Caa_{loc}$, then there exists a sequence
$\{\bsg^m\}$, bounded in $\Caa_{loc}$, such that
each $\bsg^m$ is Lipschitz (globally in all arguments)
 and $\bsg^m \to \bsg$ everywhere.
\item There exists a constant $M$, which depends only on $d$ and $N$ such
that for each $\bsf\in \BFa(\seq{C}, \seq{\kappa}, \seq{q}, \seq{\eps})$,
there exists a sequence $\{\bsf^m\}$ and a subquadratic sequence
$\{\kappa'_n\}$ (as in Definition \ref{def:BF}) such that
\begin{enumerate}
\item $\bsf^m\in \BFa(M C_n, \kappa'_n, q_n, M \eps_n)$
 and is globally Lipschitz in all of its arguments,
\item $\bsf^m \to \bsf$ pointwise, locally uniformly in $\bsz$,
and
\item $\norm{\bsf^m(\cdot,\cdot,0)}_{\mathbb{L}^{q_n}([0,T]\times B_n(b_0))}$ is bounded uniformly
in $m$, for each $n\in \N$.
\end{enumerate}
\end{enumerate}
\end{proposition}
\begin{proof}
The idea is to mollify using smooth kernels with a compact
support and linearize the tails of the quadratic parts.
In this spirit, we define the $C^{\infty}(\R^d)$-function
 \[ \eta(\bsx) = C e^{\oo{|\bsx|^2-1}} \ind{\abs{\bsx}<1},\]
 with the $C$ is chosen so that $\int \eta(\bsx) d\bsx =1$. We use the same
 notation $\eta$ (and the same formula) for its $C^{\infty}(\R)$ and
 $C^{\infty}(\R^{N\times d})$ versions.

 For each $m\in \N$, we set $\eta^m(\bsx)= m^d \eta(m\bsx)$ and
 $\eta^m(t,\bsx, \bsz) = m^{1+d+N\times d} \eta(m t) \eta(m \bsx) \eta(m
 \bsz)$, and use the standard notation for mollification, namely,
 \[ (\bsg \ast \eta)(\bsx) = \int \bsg(\bsx-\bar{\bsx}) \eta(\bar{\bsx})\,
 d\bar{\bsx},\]
 as well as for its $\eta(\tx,\bsz)$-version.  We refer the reader to
 \cite[Appendix C.4, Theorem 6]{evans} for standard properties of
 mollification.

 We also define the partial-truncation function
 $\Pi^m(\bsw)= \tfrac{|\bsw|\wedge m}{|\bsw|} \bsw$, with $\Pi^m(0)=0$;
 clearly,
 $\Pi^m$ is Lipschitz and $|\Pi^m(\bsw)|= |\bsw|\wedge m$.
 Most of the approximations in this proof will be of the form
 \[ \bsg^m(\bsx) = (\bsg \ast \eta^m)
 (\Pi^m(\bsx)) \]
 in the $\R^d$ case and
 \begin{equation}\label{equ:fm}
 \bsf^m (\tx,\bsz) = ( \bsf \ast \eta^m)(t, \Pi^m(\bsx), \Pi^m(\bsz))\end{equation}
 in the $\R^{1+d+N\times d}$-case after extending the
 domain of $\bsf$ via $\bsf(t, \bsx, \bsz) = \bsf(0, \bsx, \bsz)$ for $t< 0$, and
 $\bsf(t, \bsx, \bsz) = \bsf(T, \bsx, \bsz)$ when $t> T$. The same, superscript-$m$, notation will
 be used
 without explicit mention,
 when this operation is applied to other functions below.

\medskip

 (1) For $\bsg\in \Caa_{loc, b_0}$,
 one easily checks that each $\bsg^m$ is Lipschitz, bounded, the
 sequence $(\bsg^m)$ is bounded in $\Caa_{loc, b_0}$, and the convergence
 $\bsg^m \to \bsg$
 follows from the standard properties of
 mollification.

\medskip

 (2) Given a function $\bsf$ satisfying the assumptions in part (2),
  since the convolution $\bsf\ast \eta^m$ is smooth, and $\Pi^m$ is
 Lipschitz and bounded, each approximation $\bsf^m$ is globally
 Lipschitz in all of its variables.
 Furthermore, when $\bsz^m \rightarrow
 \bsz$, we have $\Pi^m(\bsz^m) = \bsz^m$ for sufficiently large $m$ and,
 so, using the properties that mollifications of a continuous function converge locally uniformly, we have
 $\bsf^m \to \bsf$  locally uniformly in $\bsz$.

\smallskip

 To verify the condition $\BFa$ for $\bsf^m$, we fix $n\in\N$, and start with the
 quadratic-triangular component $\bsfq$. Thanks to the fact
 that $\eta$ is of compact support, all components $\bsfq^{m,i}$ of the approximation
 $\bsq^m$ have the following property
 \[|\bsfq^{m,i}(t, \bsx,
 \bsz)| \leq C_n (2+ \sum_{j=1}^i |\bsz^j|^2) \leq 2 C_n (1+\sum_{j=1}^i
 \abs{\bsz^j}^2) \text{  on }
 [0,T]\times B_n \times \R^{N\times d}. \]
 Therefore
 $\bsfq^m$ is quadratic-triangular as well, with $C'_n = 2 C_n$.
 A similar argument can be applied to $\bsfe$.
 For $\bsfk$, it
 follows from \cite[Appendix C.4, Theorem 6 (iv)]{evans} that $\bsfk^m$
 converges to $\bsfk$ in $\mathbb{L}^{q_n}([0,T]\times B_n)$ for each $n$;
 in particular,
 the sequence
 $\norm{\bsfk^m}_{\mathbb{L}^{q_n}([0,T]\times B_n)}$, which
 differs from
 $\norm{\bsf^m(\cdot,\cdot,0)}_{\mathbb{L}^{q_n}([0,T]\times B_n)}$ only by
 a constant, is bounded
 in $m$.

 For $\bsfl$, a direct approximation of
 $\bsfd(\tx,\bsz) = \diag (\bsz \bsfl(\tx,\bsz))$
 does not produce the function in the same class; it needs an adjustment by
 a subquadratic term.
To see that we note that
 \[
 \bsfd^m(\tx,\bsz)
 = \diag ( \bsz\,
 \bsfl^m(\tx,\bsz)) - \bsfL^m(\tx,\bsz),\]
 where
 $\bsfd^m(\tx,\bsz)=( \bsfd  \ast \eta^m)  (t,\Pi^m(\bsx), \Pi^m(\bsz))$
 and
 \[ (\bsfL^i)^m(\tx,\bsz)= \sum_{j=1}^d (\bsfl_{ji} \ast
 \hat{\eta}^m_{ij} )(t, \Pi^m(\bsx), \Pi^m(\bsz))
 \eand
 (\hat{\eta}^m)_{ij}(\tx,\bsz) = z_{ij} \eta^m(\tx,\bsz).\]
 The function $\bsfl^m$ grows at most linearly, with the constant $C'_n$
 bounded from above by $C_n$ multiplied by a constant which depends only on
 $d$ and $N$. On the other hand,
 the components of $\bsfL^m$ are mollifications of
 linearly-growing functions by kernels $\hat{\eta}^m_{ij}$, all of which are
 dominated by $\eta^m$ in absolute
 value. Therefore,
 the functions $\bsfL^m$
 are of subquadratic growth, uniformly in $m$.

The subquadratic growth of $\bsfs$ ensures the same property for
$\bsfs^m$, uniformly in $m$, perhaps with a different growth bound
$\kappa'_n$. Therefore, each $\bsf^m$ admits a
decomposition as in \eqref{equ:BF-decomp} into the functions $\bsfd^m +
\bsfL^m$,
$\bsq^m$, $\bss^m- \bsfL^m$, $\bsfe^m$ and $\bsfk^m$ which have all the
required properties.
\end{proof}

\subsubsection{Existence of Lyapunov pairs}
Proposition \ref{pro:h-exists} below confirms Proposition \ref{thm:BF-Lyap} and Remark \ref{rem:BF} part (1). Its proof is partially based on a construction in
\cite[Proposition 3.1, p.~174]{Bensoussan-Frehse}.

\begin{proposition}
\label{pro:h-exists} Let $\{\bsf^m\}$ be a sequence in $\BFa(\{C_n\}, \{\kappa_n\},
\{q_n\}, \{\eps_n\})$ with
 $\norm{\bsf^m(\cdot,\cdot,0)}_{\mathbb{L}^{q_n}([0,T]\times B_n(b_0))}$
 bounded, for each $n\in\N$.
Then, for each sequence $\seq{c}$ of positive numbers, there exists a
sequence $\seq{\bar{\eps}}$ in $(0,\infty)$ such that if
$\eps_n\leq \bar{\eps}_n$ for all $n$,
there exists families $\seq{h}$ and $\seq{k^m}$ such that, for each $m$,
$(\{h_n\}, \{k^m_n\})$ is
a local $\seq{c}$-Lyapunov pair for $\bsf^m$. In particular, when $\seq{C}, \seq{\kappa}, \seq{q}$ and $\seq{\eps}$ are constants in $n$, there exists a $c$-Lyapunov pair for $\bsf$, for any $c>0$.
\end{proposition}

\begin{proof}
We restrict the spatial domain to $B_n(b_0)$ and suppress the subscript $n$
throughout the proof. When $\seq{C}, \seq{\kappa}, \seq{q}$ and $\seq{\eps}$ are constants in $n$, the spatial domain is $\R^d$. Since $\{\bsf^m\}$ satisfies the approximate condition
$(BF)$
with uniform growth sequences $\seq{C}$, $\seq{\kappa}$, $\seq{q}$ and
$\seq{\eps}$, we suppress the superscript $m$ as well.
For $\bsy=(y_1,\dots, y_N)\in\R^N$ and $k=1, \dots, N$, we define
 \[ G_k(\bsy) = \cosh(\alpha_k y_k) \quad\eand\quad S_k(\bsy) = \sinh (\alpha_k
 y_k),\]
with $\alpha_1,\dots, \alpha_N > 0$ to be determined later.
 Recursively, we set
 \[ H_{N+1}= 0 \quad\eand \quad H_k = \exp(G_k + H_{k+1}),  \efor k=1,\dots, N,\]
as well as \[ P_0=1, \quad P_k=\textstyle\prod_{i=1}^k H_i,\quad k=\ft{1}{N},\]
noting that $1\leq P_1\leq P_2\leq \dots \leq P_N$.
The (linear combinations) of functions $G_k$ and $S_k$ play the
role of $\beta$, and
$H_{k}$ the role $X_{\nu}$, in the notation
of \cite[equation (3.6), p. 174]{Bensoussan-Frehse}. 
With $A_i=\alpha_i
S_i P_i$ we compute
\begin{align*}
D_i H_k &= A_i P_{k-1}^{-1} \inds{i \geq k},
\end{align*}
where $D_i$ stands for $\tfrac{\partial}{\partial y^i}$, so that
\begin{align*} D_j P_i &= P_i \sum_{k=1}^i (H_k)^{-1} D_j H_k = P_i A_j
\sum_{k=1}^{i\wedge j} (P_k)^{-1}.  \end{align*}
Setting $h=H_1$ and $\tA_i = \alpha^2_i G_i P_i$,  for $1\leq
i,j\leq N$, we obtain
\begin{align}
\label{equ:dth}
D_i h = A_i \quad \eand \quad
 D_{ij} h   = \tA_i \inds{i=j}+ A_i A_j \sum_{k=1}^{i\wedge j} P_k^{-1}.
\end{align}

To prove \eqref{equ:Lyapunov}, we pick
an $N\times d$-matrix $\bsz$ and set
$\bszeta^i= \bsz^i \bsig$ where $\bsz^i$ is the $i$-th row of $\bsz$, so that
$\tfrac{1}{\Lambda} |\bsz^i|^2 \leq |\bszeta^i|^2 \leq \Lambda
|\bsz^i|^2$. Thanks to \eqref{equ:dth}, we obtain
\begin{align*}
    D^2 h : \scl{\bsz}{\bsz}_{\bsa} &= \sum_{ij} D_{ij} h\,
    \bszeta^i (\bszeta^j)^{\top} =
    \sum_{i=1}^N \tA_i \abs{\bszeta^i}^2
+  \sum_{i=1}^N P_i^{-1} \abs{\bseta_i(\bszeta)}^2,
\end{align*}
where
$\bseta^i(\bszeta) = \sum_{j=i}^N A_j \bszeta^j$.

To deal with the $Dh \bsf$-part in \eqref{equ:Lyapunov},
we consider various
constituents in  \eqref{equ:BF} separately. We reuse the letter $C$ for
any constant - possibly differing from place to place - which depends only
on the sequences $\seq{C}$ and $(\kappa_n)$ from statement, or the
universal constants.

- \emph{The Quadratic-Linear part:} Let $\bsld^i$ denote the $i$-th row of
the matrix  $\bsfl^{\top} \sigma^{-1}$, so that
\[ \diag( \bsz \bsfl)_i = (\bsz \bsfl)_{ii} = \bszeta^i
(\bsld^i)^{\top}.\]
If we extend the definition of $\bseta=\bseta(\bszeta)$ and $\bsld$ by setting
$\bseta^{n+1}=\bsld^0=\boldsymbol{0}$, `summation by parts' implies that
\[ \sum_{i=1}^N A_i\,  \bszeta^i (\bsld^i)^{\top}
= \sum_{i=1}^N (\bseta^{i} - \bseta^{i+1})   (\bsld^i)^{\top} =
\sum_{i=1}^N \bseta^i  (\bsld^i - \bsld^{i-1})^{\top}.\]
The fact that $|\bsfl|\leq C(1+|\bsz|)$ and Young's inequality yield
\begin{align*}
  Dh \diag(\bsz \bsfl) &=
  \sum_{i=1}^N
 \bseta^i  (\bsld^i - \bsld^{i-1})^{\top}
 \leq  \sum_{i=1}^N C \abs{\bseta^i} (1+\abs{\bszeta})\\
 &\leq \sum_{i=1}^N  C P_i (1+\abs{\bszeta}^2) + \tot \sum_{i=1}^N P_i^{-1}
  \abs{\bseta^i(\bszeta)}^2.
\end{align*}

- \emph{The Quadratic-Triangular part:}
\begin{align*}
  Dh\, \bsfq  &=
  \sum_{i=1}^N
  A_i \sfq^i \leq C
  \sum_{j=1}^N (1+\abs{
    \bszeta^j}^2)
  \sum_{i=j}^N \abs{A_i} \leq  C
  \sum_{j=1}^N (1+\abs{\bszeta^j}^2)
  \sum_{i=j}^N  \abs{A_i}.
\end{align*}

- \emph{Choosing constants $\alpha_1, \dots, \alpha_N$:}
The inequality
$\sum_{j=1}^N P_i \leq N \sum_{j=i}^N P_j$ is valid for each $i$, so
\begin{align*}
 &\tot D^2 h : \scl{\bsz}{\bsz}_{\bsa} - Dh (
\diag (\bsz \bsfl) + \bsfq)\\
 & \quad \geq  \sum_{i=1}^N \Big( \tot \tilde{A}_i -  C \sum_{j=i}^N
 (P_j+\abs{A_j}) \Big)\abs{\bszeta^i}^2 - C \sum_{j=i}^N(P_j+\abs{A_j})\\
 & \quad \geq \tfrac{1}{\Lambda} \sum_{i=1}^N \Big( \tot \tilde{A}_i -  C \sum_{j=i}^N
 (P_j+\abs{A_j}) \Big)\abs{\bsz^i}^2 - C \sum_{j=i}^N(P_j+\abs{A_j}).
 \end{align*}
The choice $\alpha_N \geq 2C + 1$, together with $G_N \geq 1$ and
$|S_N|\leq G_N$ yields
 $\tot \tA_N > C (P_N+\abs{A_N})$. For $i\leq N-1$, we have
 \[
  \tot \tilde{A}_i - C \sum_{j=i}^N (P_j + |A_j|) \geq [\tot \alpha_i^2  - C(1+\alpha_i)]G_i P_i - \sum_{j=i+1}^N (P_j + |A_j|).
 \]
 Next, we observe the fact that $\sum_{j=i+1}^N (P_j
 + \abs{A_j})$ depends only on $\alpha_{i+1},\dots, \alpha_N$, and is
 bounded on the set $[-c,c]^N$ as a function of $y^{i+1},\dots, y^N$.
 Therefore, we
 can choose a sufficiently large $\alpha_i$ so that the left-hand side of the
 previous inequality is positive, and
 continue this process recursively down to $i=1$. This way, we
 obtain a constant $C_0$, depending on $C,c$ and the universal
 constants, so that
\begin{align*}
 \tot D^2 h : \scl{\bsz}{\bsz}_a - Dh (\diag (\bsz \bsfl) + \bsfq)
  \geq  C_0 \abs{\bsz}^2 -C_0.
\end{align*}

- \emph{The Subquadratic part:} With $\alpha_1, \dots, \alpha_N$ now fixed,
$\sum_i |A_i|$ is  bounded on $[-c, c]^N$, as a function of $\bsy$.
Therefore $\eps_0 = C_0/ \sup_{ [-c,c]^N} (\sum_i \abs{A_i})>0$. With
$\eps$ as in
\eqref{equ:BF-approx} assumed to be smaller than $\eps_0$, we
pick $\eps' < \eps_0 - \eps$ and set
\[ \kappa^*= \sup_{\iota\geq 0} ( \kappa(\iota) - \eps' \iota).\]
The sublinear growth of
$\kappa$ ensures that $\kappa^*$ is well-defined in
$[0,\infty)$ and
\begin{align*}
  Dh \, \bsfs  &=   \sum_{i=1}^{n}
  A_i \sfs^i \leq
   \sum_{i=1}^N \abs{A_i} \kappa(\abs{\bsz}^{2})
\leq \tfrac{C_0}{\eps_0} ( \kappa^* + \eps' \abs{\bsz}^2)=
\tfrac{C_0 \eps'}{\eps_0} \abs{\bsz}^2 +
\tfrac{C_0}{\eps_0} \kappa^*.
\end{align*}

Lastly, we combine all of the above estimates  to obtain
\begin{align*}
 \tot D^2 h : \scl{\bsz}{\bsz}_{\bsa} - Dh \bsf
 & \geq
 \tot D^2 h : \scl{\bsz}{\bsz}_{\bsa} -
 Dh\, (\diag (\bsz \bsfl) + \bsfq + \bsfs + \bsfk)
- \eps \abs{Dh} \abs{\bsz}^2\\
&\geq \big( C_0 \tfrac{\eps_0-\eps'}{\eps_0} - \eps \abs{Dh} \big) \abs{\bsz}^2 -C_0 - C_0 \eps_0 \kappa^*-
Dh\, \bsfk.
\end{align*}
Since $ \eps \abs{Dh} \leq  \eps C_0/\eps_0$, it suffices to define $k = C_0+
C_0 \eps_0 \kappa^* + C_1 \abs{\bsfk}$, for some $C_1\geq \abs{Dh}$,  and,
if necessary, scale both $h$ and
$\eps_0$ (yielding $\bar{\eps}$)
to make the coefficient in front of $\abs{\bsz}^2$ equal to $1$.
\end{proof}

\subsubsection{Conclusion of the Proof of Theorem \ref{thm:sufficient}}
Let $\bsg$ and $\bsf$ be two functions which satisfy the conditions
of Theorem \ref{thm:sufficient}, namely $\bsg$ is in $\Ca_{loc, b_0}$ and of
subquadratic growth, $\bsf$ satisfies the conditions (AB) and the
approximate condition (BF).
We start by picking a sequence $(\bsf^m, \bsg^m)$ of
Lipschitz approximations constructed in Proposition \ref{pro:mollification}.
Thanks to the Lipschitz property of all ingredients,
each approximate system  \eqref{equ:BSDEm}
admits a continuous Markovian solution $(\bsv^m, \bsw^m)$
(see, e.g.,  \cite[Theorem 4.1 and Corollary 4.1]{ElKaroui-Peng-Quenez}).
Moreover, the Lipschitz approximations $\bsf^m$ satisfy the condition (AB),
of Definition \ref{def:AB}, possibly with the same positively-spanning set
$\bssa_1,\dots, \bssa_K$, and a possibly different, but $m$-independent,
$\lone$-function $l$. The existence of the latter - as the supremum of a family of finer and
finer mollifications of an integrable function - is guaranteed by the Hardy-Littlewood
maximal theorem. Similarly, the approximations $\bsg^m$
of the terminal condition $\bsg$ satisfy the inequality
$|\bsg^m(\bsx)| \leq \zeta(\bsx)$ uniformly in $m$, for
some smooth function $\zeta$ with $\lim_{\abs{\bsx} \to\infty}
\zeta(\bsx)/\abs{\bsx}^2 = 0$.

Let us first show that each $\bsv^m$ is bounded. Let $\bssa_1, \dots, \bssa_K$ be a positive spanning set from condition (AB). Given $k\in \{1, \cdots, K\}$, we consider the following quadratic BSDE
\[
 d\bar Y^{k,m}_t = - \big[\ell(t) + \tfrac12 |\bar \bsZ^{k,m}_t|^2\big] dt + \bar \bsZ^{k,m}_t dW_t, \quad \bar Y^{k,m}_T= \bssa_k^{\top} \one_N \bsg^m(\bsX_T),
\]
where $\one_N$ denotes the vector of $1$s in $\R^N$. Since $\bsg^m$ is bounded by its construction in Proposition \ref{pro:mollification}, the previous BSDE admits a bounded solution $(\bar Y^{k,m}, \bar \bsZ^{k,m})$. Recall that $(\bsv^m, \bsw^m)$ is a Markovian solution to a Lipschitz BSDE. The comparison theorem for Lipschitz BSDEs (see eg. \cite[Theorem 2.2]{ElKaroui-Peng-Quenez}, whose proof only needs one generator to be Lipschitz) implies that $\bssa_k^{\top} \bsv^m(\cdot, X)\leq \bar Y^{k,m}$, hence $\bssa_k^{\top} \bsv^m(\cdot, \bsX)$ is bounded from above. It remains to use
the following fact: a sequence $\set{\bsv^m}$ in $\R^N$ for
which the sequence $\{ \bssa^{\top} \bsv^m\}$ is bounded from above for each $\bssa$ in some
positive spanning set of $\R^N$, is itself bounded in $\R^N$. Indeed, since the $i$-th canonical base $\mathfrak{e}_i$ for $\R^N$ can be positively spanned by $\mathfrak{e}_i = \lambda^i_1 \bsa_1 + \cdots \lambda^i_K \bsa_K$, we have that the $i$-th component of $\bsv^m$ has the decomposition $\mathfrak{e}_i^\top \bsv^m = \lambda^i_1 \bsa_1^\top \bsv^m + \cdots \lambda^i_K \bsa_K^\top \bsv^m$, which is bounded from above. Similarly argument applied to $-\mathfrak{e}_i$ also shows that the $i$-th component of $\bsv^m$ is bounded from below.

We now show $\{\bsv^m\}$ is bounded uniformly in compacts of $[0,T]\times \R^N$.
Given $k \in \set{1,\dots, K}$, we define
\[ e^m_k(t,\bsx) = \exp\Big( \bssa_k^{\top}
\bsv^m(t,\bsx) - \int_t^T l(s)\, ds\Big).\]
A direct computation yields that for each $(\tx)$, the  drift term in  the
$\PP^{\tx}$-semimartingale decomposition of the process $e^m_k(\cdot,\bsX)$
is given as the integral of
\[    e^m_k (s,\bsX_s) \Big( -\bssa_k^{\top}
\bsf^m(s,\bsX_s, \bsZ^m_s) + l(s) + \tot \abs{\bssa_k^{\top} \bsZ^m_s}^2 \Big), \]
Therefore, by the condition (AB), the process $e^m_k(\cdot,\bsX)$ is a nonnegative local
submartingale. Thanks to the boundedness of $\bsv^m$, it is, in fact, a
uniformly bounded submartigale and we can use the  Markov property to
conclude that
\[ e^m_k (t, \bsx) \leq \bar{e}(t,\bsx), \quad \text{ where }
\bar{e}_k(t,\bsx):=\EE^{\tx}[ \exp( \bssa_k^{\top} \one_N  \zeta(\bsX_T))], \]
Smoothness and the subquadratic growth of $\zeta$ imply that the function
$\bar{e}_k$ is smooth (see \cite[Theorem 12, p.~25]{Friedman}).
We then conclude that the sequence $\{ e^m_k\}$ - and, therefore, also,
$\{\bssa_k^{\top} \bsv^m\}$ -
is bounded from above on compact subsets of $\TR$, uniformly in $m$. Since $\{\bssa_k\}$ positively span $\R^N$, $\{\bsv^m\}$ is uniformly bounded on the same compact subset as well.

When $\bsg$ is bounded, but only the weaker version (wAB) of the condition
(AB) is satisfied, we have from the construction of $\bsf^m$ that
\[
 \bssa^\top_k \bsf^m(t,\bsx,\bsy,\bsz) \leq l(t) +\tfrac12 |\bssa^\top_k \bsz|^2 + \bssa^\top_k \bsz \bsL^m_k(t,\bsx,\bsz) - \bssa^\top_k (\bsL_k\ast \hat{\eta}^m) (t,\Pi^m(\bsx), \Pi^m(\bsz)),
\]
for any $m\in \N$, $(t,\bsx)\in [0,T]\times \R^d$, and $\bsz\in \R^{N\times d}$. Here  $\bsL^m_k$ is defined similarly as in \eqref{equ:fm}, $\hat{\eta}^m(t,\bsx,\bsz) = \bsz \eta^m(t,\bsx,\bsz)$. Since the $\text{supp }\eta^m\subseteq B_{1/m}$ and $\bsL_k$ has at most linear growth, there exists a constant $C$ such that
\[
 \left|\bssa^\top_k(L_k\ast \hat{\eta}^m)(t,\Pi^m(\bsx),\Pi^m(\bsz))\right|\lec \tfrac{1}{m} (1+ m\wedge |\bsz|) \leq C, \quad \text{for all } m.
\]
Combining the previous two estimates, we have
\[
  \bssa^\top_k \bsf^m(t,\bsx,\bsy,\bsz) \leq l(t) + \tfrac12 |\bssa^\top_k \bsz|^2 + \bssa^\top_k \bsz \bsL^m_k(t,\bsx,\bsz),
\]
for a different $l\in \mathbb{L}^1[0,T]$.

Argue as before that the
drift of the process $e^m_k(\cdot,\bsX)$
is given as the integral of
\[    e^m_k (u,\bsX_u) \Big( p^m_k(u,\bsX_u) - \bssa_k^{\top} \bsZ^m_u \bsL^m_k(u, \bsX_u, \bsZ^m_u)\Big),\]
for some functions $p^m_k\geq 0$,
while its martingale-part admits the $dW$-integrand of the form
\[    e^m_k (u,\bsX_u)\, \bssa_k^{\top} \bsZ^m_u.\]
Boundedness of $e^m_k$ and $\bsL^m_k$ allows us to conclude that
$e^m_k(\cdot,\bsX)$ is a submartingale under an equivalent measure (given
as the Girsanov transformation with drift $-\bsL^m_k$),  with
the terminal value bounded  uniformly in  $m$.
The rest is as before, and leads to a similar
conclusion, except that now the boundedness is uniform in $(\tx)$.

Moreover,
using only the function $l$ and the positive spanning set $\bssa_1,\dots,
\bssa_K$ of condition (wAB), we well as the $\linf$-bounds on $\bsg$, one
can produce an a-priori bound $c$ on $\{\bsv^m\}$. This way we obtain a
sequence - namely $\seq{c}$ with each
$c_n=c$ -  independently of
the other constants $\seq{C}$, $\seq{q}$
and $\seq{\kappa}$ appearing in the approximate condition
(BF). This way, we can avoid circularity in the
definition the sequence $\seq{\bar{\eps}}$  of
Proposition \ref{pro:h-exists}, that enforces the `smallness' condition on the
`error' term in the (BF)-decomposition of $\bsf$.

Whether $\bsg$ is bounded or unbounded, we have produced a sequence
$\seq{c}$ of a-priori bounds that can be used together with the sequences
$\{\bsf^m\}$ and $\{\bsg^m\}$ in Proposition \ref{pro:h-exists} to
construct local $\seq{c}$-Lyapunov pairs $\{h_n, k^m_n\}$ for $\bsf^m$ with
uniformly $\el^{q_n}$-bounded $k$-parts.
By Theorem \ref{thm:existence}, this is enough to guarantee the
existence of locally H\"{o}lderian solution $(\bsv,\bsw)$.

To establish uniqueness, we note that the sequence $\{k^m_n\}$ from
Proposition \ref{pro:h-exists} will be constant (both in $n$ and $m$) under
the condition (a) of Theorem \ref{thm:sufficient}. Moreover, due to
the absence of dependence on $n$ in (b), a  $c$-Lyapunov pair $(h,k)$
with constant $k$ can be constructed for any $c$.
The fact that any bounded continuous solution
is a-priori bounded by the constant $c$ constructed above, together with
the local Lipschitz condition in (c), is enough to apply the abstract
uniqueness result Theorem \ref{thm:unique}, part (2).

\subsection{Proofs for examples}

\subsubsection{Proof of Theorem \ref{thm:equilibrium}}
Consider the following system of BSDE
\begin{equation}\label{equ:equil-trans}
 d\tilde{\bsY}_t = \tilde{\bsf}(t, X_t, \tilde{\bsZ}_t) \, dt + \tilde{\bsZ}_t \sigma(t, X_t) \, dW_t, \quad \tilde{Y}_T= \bsg(X_T),
\end{equation}
where
\[
 \tilde{\bsf}(t, x, \tilde{\bsz}) = \bsf(\tilde{\bsz} \sigma(t,x)) \quad \text{ and } \quad \bsf(\bsz) = - \tfrac{1}{2} \bnu^2 + \tfrac{1}{2} \bsA[\bmu]^2 - \bsA[\bmu] \bmu \quad \text{with } \bsz= (\bmu, \bnu).
\]
We will use Theorem \ref{thm:sufficient} to establish the existence of a H\"{o}ldearian solution $(\tilde{\bsv}, \tilde{\bsw})$. Then $(\tilde{\bsv}, \tilde{\bsw}\sigma)$ is a H\"{o}ldearian solution of \eqref{equ:equilibrium}.

Let us first verify the condition (AB).
Let $\sigma_i$, $i=1,2$, be the $i$-th column vector of $\sigma$. Denote $\tilde{\mu} = \tilde{\bsz}\sigma_1$ and $\tilde{\nu} = \tilde{\bsz}\sigma_2$. Then
\[
 \tilde{\bsf} = - \tfrac{1}{2} \tilde{\bnu}^2 + \tfrac{1}{2} \bsA[\tilde{\bmu}]^2 - \bsA[\tilde{\bmu}]\tilde{\bmu}.
\]
Let $(\mathfrak{e}_1, \dots, \mathfrak{e}_N)$ be the standard Euclidean basis of $\R^N$, and $\bssa_{N+1}= (\alpha^1,\dots, \alpha^N)$ where the sequence of constants $\{\alpha^i\}$ appears after \eqref{equ:equilibrium}. Then the set $(-\mathfrak{e}_1, \dots, -\mathfrak{e}_N, \bssa_{N+1})$ positively span $\R^N$. Moreover,
\begin{align*}
 -\mathfrak{e}_i^\top \tilde{\bsf}\leq \tfrac12 (\tilde{\mu}^i)^2 +\tfrac12 (\tilde{\nu}^i)^2 \quad \text{ for } 1\leq i\leq N \quad \text{and} \quad \bssa_{N+1}^\top \tilde{\bsf} = -\tfrac{1}{2} \bsA[\tilde{\bnu}^2]-\tfrac{1}{2} \bsA[\tilde{\bmu}]^2\leq 0.
\end{align*}
In order to verify the condition (BF), let us introduce an invertible linear transformation on $\R^N$ via
\begin{equation}\label{equ:equ-oY}
 \overline{Y}^i = \tilde{Y}^i -\tilde{Y}^N, \quad i=1,\cdots, N-1,  \quad \overline{Y}^N= \tilde{Y}^N,
\end{equation}
and $\overline{\bmu}, \overline{\bnu}$, and $\overline{\bsg}$ in a similar manner.
A simple calculation reveals  the dynamics of
$\overline{\bsY}$ as
\[
 d \overline{\bsY}_t = \overline{\bmu}_t dB_t + \overline{\bnu}_t dB^\bot_t -
 \overline{\bsf}(\overline{\bmu}_t, \overline{\bnu}_t)\, dt, \quad \overline{\bsY}_T =
 \overline{\bsg}(B_T, W_T),
\]
where  $\overline{\bsf}$ is given by
\begin{align*}
 \overline{f}^i &= -\tfrac{1}{2} \overline{\nu}^i (\overline{\nu}^i + 2 \overline{\nu}^N) - \overline{\mu}^i \big( \sum_{j=1}^{N-1} \alpha^j \overline{\mu}^j + \overline{\mu}^N\big), \quad i=1,\dots, N-1, \text{ and}\\
 \overline{f}^N &= -\tfrac{1}{2} (\overline{\nu}^N)^2 +\tfrac12 \big(\sum_{j=1}^{N-1} \alpha^j \overline{\mu}^j + \overline{\mu}^N\big)^2 - \overline{\mu}^N \big(\sum_{j=1}^{N-1} \alpha^j \overline{\mu}^j + \overline{\mu}^N\big).
\end{align*}
Using this, explicit, expression, one easily checks that $\overline{\bsf}$ satisfies the condition (BF) of Definition \ref{def:BF}. On the other hand, since $\tilde{\bsf}$ already satisfies the condition (AB), after the linear transformation of $\R^N$, $\overline{\bsf}$ satisfies (AB) as well (cf. Remark \ref{rem:wAB}).
Therefore the existence and uniqueness of a bounded continuous solution to \eqref{equ:equ-oY} (hence \eqref{equ:equil-trans} and \eqref{equ:equilibrium}) follows from Theorem \ref{thm:sufficient}. Finally, when the terminal condition is bounded, combining Theorem \ref{thm:unique} part (1) and \cite[Theorem 1.6 (2)$\implies$(1)]{Kardaras-Xing-Zitkovic}, we confirm the existence of an equilibrium. Conversely any equilibrium with continuous certainty equivalence functions corresponds to a continuous Markovian solution of \eqref{equ:equilibrium}, which is already proven to be unique.

\subsubsection{Proof of Proposition \ref{pro:Darling}}
We define the approximated driver $\bsf^m$ by
\[(\bsf^m(\bsy, \bsz))^k  = \tfrac12 \sum_{i,j}
{\Gamma}^k_{ij} (\bsy) \Big( \Pi^m(\bsz)^{\top} \Pi^m(\bsz) \Big)_{ij},
\text{ for $\bsy\in \R^N$, $\bsz\in
\R^{N\times d}$},\]
where $\Pi^m(\bsz) = \tfrac{|\bsz|\wedge m}{|\bsz|}\bsz$ for $m\in
\mathbb{N}$.
Also, we construct a sequence
$(\bsg^m)$ of
Lipschitz approximation of $\bsg$ as in Proposition
\ref{pro:mollification}. Mollification does not increase the
$\linf$-norm and the set
sub-level set $M_0 = \phi^{-1}((-\infty,0])$
is convex; therefore, the images of all $\bsg^m$ remain inside $M_0$,
i.e.,  $\phi^{-1}(\bsg^m(x))\leq 0$, for all $x$ and $m$.

The globally-Lipschitz structure of its ingredients
implies that
the approximated system
\[
 d\bsY^m_t = -\bsf^m(\bsY^m_t, \bsZ^m_t) \, dt + \bsZ^m_t \, d\bsW_t, \quad
 \bsY^m_T = \bsg^m(W_T)
\]
admits a unique H\" olderian solution $(\bsv^m, \bsw^m)$.
To show that  the sequence $\set{\bsv^m}$ is uniformly
bounded, more precisely, that $\bsv^m \in M_0$ for all $m$, we define
the stopping time
$\tau=\inf\sets{s\geq t}{ \phi(\bsY^m_s)\leq \eps }\wedge T$,
for given
$t \in [0,T]$ and $\eps>0$.
By \itos{} formula, we obtain
\begin{multline*}
 \EE_{t }[\phi(\bsY^m(\tau)) ] - \phi(\bsY^m_{t}) = \\
 =\EE_{t} \big[\int_t^{\tau} \tfrac12
 D^2\phi(\bsY^m_s):  (\bsZ^m_s)^{\top} \bsZ^m_s - \tfrac12 D\phi(\bsY^m_s)
 \Gamma(\bsY^m_s)
 \Pi^m(\bsZ^m_s)^{\top} \Pi^m(\bsZ^m_s) \, ds\big]\\
 =
 \tot \EE_{t} \left[ \int_{t}^{\tau}
 \left(  \Hess \phi( \Pi^m(\bsZ^m_s), \Pi^m(\bsZ^m_s))
 + \tot \tfrac{|\bsZ_s| - |\bsZ_s|\wedge m}{|\bsZ_s|} D^2 \phi(\bsY^m_s)
 (\bsZ^m_s)^{\top} \bsZ^m_s\right)\, ds\right],
\end{multline*}
where the local martingale term can be dealt with by stopping, using the
fact that $D\phi$ is bounded on compacts and
$\bsZ^m \in H^2$.
Double convexity of $\phi$ implies that
both terms inside the expectation above are nonnegative, so that
\[
 \phi(\bsY^m_{t}) \leq \EE_{t}[\phi( \bsY^m_{\tau})], \text{ a.s.}
\]
Since $\phi(\bsY^m_T)\leq 0<\eps$
the stopping time $\tau$ gets realized strictly before $T$.
 Therefore,
the right-hand side above is bounded from
above by $\eps$, immediately implying that $\phi(\bsY^m_t)\leq \eps$, a.s.,
and establishing the claim about boundedness of $\bsv^m$.

Once the a-priori boundedness of the approximating sequence $\{\bsv^m\}$ is established, we can use the strict
geodesic convexity of $\phi$ on some neighborhood of $M_0$ to conclude that
$\phi$ can be suitably redefined on a complement of a neighborhood of $M_0$
to serve, together with $k=0$, as a $c$-Lyapunov pair, for large enough
$c$.  Then the existence of a locally
H\"{o}lderian solution to \eqref{equ:BSDE-Darling}, with $\bsY_T=\bsg(W_T)$,
readily follows from
Theorem \ref{thm:existence}.

\subsubsection{Proof of Proposition \ref{pro:BF-game}}

Let us first argue that a solution $(\bsY, \bsZ)$ with $(\hat{\bmu}(\bsZ), \hat{\bnu}(\bsZ)) \in \bmo^2$
to \eqref{equ:BSDE-BF} corresponds to a Nash
equilibrium.  For a given $\bmu\in \bmo$, we consider the process
\[
 \tilde{Y}^1_t = \EE^{\mu, \hat{\nu}}_t \Big[\int_t^T \big(h^1(\bsX_u) +
 \tfrac12 |\bmu_u|^2 + \theta \bmu^\top_u \hat{\bnu}_u \big)\, du + g^1(\bsX_T)\Big].
\]
Thanks to the fact that
$h^1$ and $g^1$ are both bounded and $\bmu, \hat{\bnu}\in \bmo$, $\tilde{Y}^1$ is bounded. Since  $\bsX$ generates the
same filtration as $\bsW$, the martingale representation implies the existence of a process $\tilde{\bsZ}^1\in \sP^2$
such that
\[
 \tilde{Y}^1_t = g^1(\bsX_T) + \int_{t}^T \big(h^1(\bsX_u) + \tfrac12
 |\bmu_u|^2 + \theta \bmu^\top_u \hat{\bnu}_u\big) \, du - \int_{t}^T
 \tilde{\bsZ}^1_u \, d\bsW^{\mu, \hat{\nu}}_u.
\]
If we subtract the corresponding
component $Y^1$ of the solution
$\bsY$ of \eqref{equ:BSDE-BF} from it, we obtain
\[
 \tilde{Y}^1_0 - Y_0^1 = \int_0^T (L^1(\bmu_u, \hat{\bnu}_u,
 \bsZ_u)
 - L^1(\hat{\bmu}_u, \hat{\bnu}_u, \bsZ_u))\, du - \int_0^T
 (\tilde{\bsZ}^1_u - \bsZ^1_u) \, d\bsW^{\mu, \hat{\nu}}_u.
\]
Since $\hat{\bmu}$ is the minimizer of $L^1(\cdot, \hat{\bnu}, \bsZ)$, and
both $Y^1$ and $\tilde{Y}^1$ are bounded, a localization argument yields
$\tilde{Y}^1_0 \geq \hat{Y}^1_0$, confirming the first inequality in
\eqref{equ:Nash-equ}. A similar argument applies to the cost of the second
player as well.

For the existence and uniqueness of $(\bsY, \bsZ)$,
we verify all conditions in Theorem \ref{thm:sufficient}.
Introduce an invertible linear transformation on $\R^2$ via $\tilde{y}^1 = y^1 - y^2$ and $\tilde{y}^2 = y^2$. Define $\tilde{\bsz}$, $\tilde{\bsg}$ similarly, and consider the BSDE
\begin{equation}\label{equ:BF-trans}
 d\tilde{\bsY}_t = - \tilde{\bsf}(X_t, \tilde{\bsZ}_t)\, dt + \tilde{\bsZ}_t \, dW_t, \quad \tilde{\bsY}_T = \tilde{\bsg}(X_T),
\end{equation}
where
\[\tilde{f}^1(\bsx, \tilde{\bsz})= \tfrac{2\theta-1}{2(1+\theta)(1-\theta)} \tilde{\bsz}^1 \cdot (\tilde{\bsz}^1+ 2 \tilde{\bsz}^2) + h^1(x) - h^2(x) \quad \text{and} \quad \tilde{f}^2= f^2.\]
Using this explicit expression, one easily checks that $\tilde{\bsf}$ satisfies the condition (BF).

Next we show that $\bsf$ satisfies the condition (wAB), hence $\tilde{\bsf}$ satisfies the same condition as well. To this end, calculation shows that
\begin{equation}\label{equ:L1-ubb}
\begin{split}
 L^1(\bsz) &= -\tfrac{\theta^2}{2(1+\theta)^2 (1-\theta)^2} |\bsz^1+ \bsz^2|^2 - \tfrac{1-2\theta}{2(1-\theta)^2 (1+\theta)} \bsz^1 \cdot ((1-\theta) \bsz^1 + 2\bsz^2)\\
 &\leq - \tfrac{1-2\theta}{2(1-\theta)^2 (1+\theta)} \bsz^1 \cdot ((1-\theta) \bsz^1 + 2\bsz^2).
\end{split}
\end{equation}
A similar inequality holds for $L^2(\bsz)$. Therefore using the fact that $\bsh$ is bounded, we obtain functions $\bsL_i$ such that $\mathfrak{e}_i^\top \bsf \leq h^i+\mathfrak{e}_i^\top \bsz \bsL_i$ for $i=1,2$.

When $\theta\leq 1/2$, using the first identity above, we obtain
\begin{align*}
 L^1(\bsz) + L^2(\bsz) &= \tfrac{\theta^2+\theta-1}{(1+\theta)^2 (1-\theta)^2} |\bsz^1+\bsz^2|^2 + \tfrac{1-2\theta}{2(1-\theta)^2}(|\bsz^1|^2 + |\bsz^2|^2)\\
 & \geq \tfrac{\theta^2+\theta-1}{(1+\theta)^2 (1-\theta)^2} |\bsz^1+\bsz^2|^2.
\end{align*}
Hence, for $\bssa_3=(-1, -1)$, we have
\[
 \bssa_3^\top \bsf \leq 2 \|h^1+h^2\|_{\linf} + \tfrac{|\theta^2+\theta-1|}{(1+\theta)^2 (1-\theta)^2} |\bsz^1+\bsz^2|^2.
\]
As a result, $\bsf$ satisfies the condition (wAB) with the positively spanning set $(\mathfrak{e}_1, \mathfrak{e}_2, \bssa_3)$.

When $\theta>1$, consider $\bssa_3=(-\theta, 1)$ and $\bssa_4=(1, -\theta)$. We have from \eqref{equ:munu*} that
\[
 \hat{\bmu} = \tfrac{1}{(1-\theta)(1+\theta)}(\theta \bsz^2 - \bsz^1) = \tfrac{1}{(\theta-1)(1+\theta)} \bssa_4^\top\bsz \quad \text{and} \quad \hat{\bnu}= \tfrac{1}{(1-\theta)(1+\theta)} (\theta \bsz^1 - \bsz^2) = \tfrac{1}{(\theta-1)(1+\theta)} \bssa_3^\top\bsz.
\]
On the other hand, for $\theta>1$, we have
\begin{align*}
 \bssa_3^\top \bsL &= -\theta L^1+L^2 = -\tfrac{\theta}{2} |\hat{\bmu}+ \hat{\bnu}|^2 + (\tfrac12-\theta) |\hat{\bnu}|^2 + (2\theta-1) \hat{\bnu}^\top (\hat{\bmu} + \tfrac{\theta+2}{2}\hat{\bnu})\\
 &\leq \tfrac{2\theta-1}{(\theta-1)^2(1+\theta)^2} (\bssa_3^\top\tilde{\bsz})^\top (\tfrac{\theta+2}{2} \bssa_3+ \bssa_4)^\top \bsz.
\end{align*}
A similar inequality holds for $\bssa_4^\top \bsL$. Combining the previous two estimates together with boundedness of $\bsh$, we confirm that $\bsf$ satisfies the condition (wAB) with the set of vectors $(\mathfrak{e}_1, \mathfrak{e}_2, \bssa_3, \bssa_4)$, which positively span $\R^2$ when $\theta>1$.

Finally we conclude from Theorem \ref{thm:sufficient} that the system \eqref{equ:BF-trans} (hence \eqref{equ:BSDE-BF}) admits a unique bounded continuous solution.

\subsubsection{Proof of Proposition \ref{pro:EH-game}}
It is clear that the generator of system \eqref{equ:EH-game} satisfies the condition (BF) and (wAB) when $\bsg$ is bounded ((AB) when $\bsg$ is unbounded). Then the existence (and uniqueness for bounded $\bsg$) readily follows from Theorem \ref{thm:sufficient}. Given the bounded continuous solution, $\bsZ=\bsw(\cdot, X)\in \bmo$, hence \cite[Proposition 5.1]{El-Karoui-Hamadene} concludes that $(\hat{\bmu}, \hat{\bnu})$ is a Nash equilibrium with value $(\exp(v^1), \exp(v^2))$.

\bibliographystyle{amsalpha}
\bibliography{XinZit}

\def\cprime{$'$} \def\cprime{$'$}
\providecommand{\bysame}{\leavevmode\hbox to3em{\hrulefill}\thinspace}
\providecommand{\MR}{\relax\ifhmode\unskip\space\fi MR }
\providecommand{\MRhref}[2]{%
  \href{http://www.ams.org/mathscinet-getitem?mr=#1}{#2}
}
\providecommand{\href}[2]{#2}
\begin{thebibliography}{CHKP14}

\bibitem[Aro67]{Aro67}
D.~G. Aronson, \emph{Bounds for the fundamental solution of a parabolic
  equation}, Bull. Amer. Math. Soc. \textbf{73} (1967), 890--896.

\bibitem[Aub98]{Aub98a}
T.~Aubin, \emph{Some nonlinear problems in {R}iemannian geometry}, Springer
  Monographs in Mathematics, Springer-Verlag, Berlin, 1998.

\bibitem[BE13]{Briand-Elie}
P.~Briand and R.~Elie, \emph{A simple constructive approach to quadratic
  {BSDE}s with or without delay}, Stochastic Process. Appl. \textbf{123}
  (2013), 604--618.

\bibitem[BEK13]{Barrieu-El-Karoui}
P.~Barrieu and N.~El~Karoui, \emph{Monotone stability of quadratic
  semimartingales with applications to unbounded general quadratic {BSDE}s},
  Ann. Probab. \textbf{41} (2013), 1831--2853.

\bibitem[BF00]{Bensoussan-Frehse-Nplayer}
A.~Bensoussan and J.~Frehse, \emph{Stochastic games for {$N$} players}, J.
  Optim. Theory Appl. \textbf{105} (2000), no.~3, 543--565, Special Issue in
  honor of Professor David G. Luenberger.

\bibitem[BF02]{Bensoussan-Frehse}
\bysame, \emph{Smooth solutions of systems of quasilinear parabolic equations},
  ESAIM Control Optim. Calc. Var. \textbf{8} (2002), 169--193.

\bibitem[BH06]{Briand-Hu}
P.~Briand and Y.~Hu, \emph{{BSDE} with quadratic growth and unbounded terminal
  value}, Probab. Theory and Related Fields \textbf{136} (2006), 604--618.

\bibitem[BH08]{BriHu08}
\bysame, \emph{Quadratic {BSDE}s with convex generators and unbounded terminal
  conditions}, Probab. Theory Related Fields \textbf{141} (2008), no.~3-4,
  543--567.

\bibitem[Bis73]{Bis73}
J.-M. Bismut, \emph{Conjugate convex functions in optimal stochastic control},
  J. Math. Anal. Appl. \textbf{44} (1973), 384--404.

\bibitem[BL97]{Barles-Lesigne}
G.~Barles and E.~Lesigne, \emph{S{DE}, {BSDE} and {PDE}}, Backward stochastic
  differential equations ({P}aris, 1995--1996), Pitman Res. Notes Math. Ser.,
  vol. 364, Longman, Harlow, 1997, pp.~47--80.

\bibitem[Bla05]{Bla05}
F.~Blache, \emph{Backward stochastic differential equations on manifolds},
  Probab. Theory Related Fields \textbf{132} (2005), no.~3, 391--437.

\bibitem[Bla06]{Bla06}
\bysame, \emph{Backward stochastic differential equations on manifolds. {II}},
  Probab. Theory Related Fields \textbf{136} (2006), no.~2, 234--262.

\bibitem[BM01]{Bally-Matoussi}
V.~Bally and A.~Matoussi, \emph{Weak solutions for {SPDE}s and backward doubly
  stochastic differential equations}, J. Theoret. Probab. \textbf{14} (2001),
  no.~1, 125--164.

\bibitem[CD15]{Cetin-Danilova}
U.~Cetin and A.~Danilova, \emph{Markovian nash equilibrium in financial markets
  with asymmetric information and related forward-backward systems}, to appear
  in Ann. App. Prob., 2015.

\bibitem[CDY92]{Chang-Ding-Ye}
K.-C. Chang, W.~Y. Ding, and R.~Ye, \emph{Finite-time blow-up of the heat flow
  of harmonic maps from surfaces}, J. Differential Geom. \textbf{36} (1992),
  no.~2, 507--515.

\bibitem[CHKP14]{Cheridito-et.al}
P.~Cheridito, U.~Horst, M.~Kupper, and T.~A. Pirvu, \emph{Equilibrium pricing
  in incomplete markets under translation invariant preferences}, to appear in
  Math. Oper. Res., 2014.

\bibitem[CL15]{ChoLar14}
J.H. Choi and K.~Larsen, \emph{{T}aylor approximation of incomplete {R}adner
  equilibrium models}, Finance Stoch. \textbf{19} (2015), no.~3, 653--679.

\bibitem[CM97]{ChiMan97}
R.~Chitashvili and M.~Mania, \emph{On functions transforming a {W}iener process
  into a semimartingale}, Probab. Theory Related Fields \textbf{109} (1997),
  no.~1, 57--76.

\bibitem[CN14]{Cheridito-Nam}
P.~Cheridito and K.~Nam, \emph{{BSDE}s with terminal conditions that have
  bounded {M}alliavin derivative}, J. Funct. Anal. \textbf{266} (2014), no.~3,
  1257--1285.

\bibitem[CN15]{CheNam15}
\bysame, \emph{{Multidimensional quadratic and subquadratic BSDEs with special
  structure}}, to appear in Stochastics, 2015.

\bibitem[Dar95]{Darling}
R.W. Darling, \emph{Constructing {G}amma-martingale with prescribed limit,
  using backward {SDE}}, Ann. Probab. \textbf{23} (1995), no.~3, 1234--1261.

\bibitem[Dav54]{Dav54}
C.~Davis, \emph{Theory of positive linear dependence}, Amer. J. Math.
  \textbf{76} (1954), 733--746.

\bibitem[Del03]{Del03}
Fran\c{c}ois Delarue, \emph{Estimates of the solutions of a system of
  quasi-linear {PDE}s. {A} probabilistic scheme}, S\'eminaire de
  {P}robabilit\'es {XXXVII}, Lecture Notes in Math., vol. 1832, Springer,
  Berlin, 2003, pp.~290--332.

\bibitem[DHB11]{DelHuBao11}
F.~Delbaen, Y.~Hu, and X.~Bao, \emph{Backward {SDE}s with superquadratic
  growth}, Probab. Theory Related Fields \textbf{150} (2011), no.~1-2,
  145--192.

\bibitem[EKH03]{El-Karoui-Hamadene}
N.~El~Karoui and S.~Hamad{\`e}ne, \emph{B{SDE}s and risk-sensitive control,
  zero-sum and nonzero-sum game problems of stochastic functional differential
  equations}, Stochastic Process. Appl. \textbf{107} (2003), no.~1, 145--169.

\bibitem[EKPQ97]{ElKaroui-Peng-Quenez}
N.~El~Karoui, S.~Peng, and M.~C. Quenez, \emph{Backward stochastic differential
  equations in finance}, Math. Finance \textbf{7} (1997), no.~1, 1--71.

\bibitem[{\'E}me89]{Eme89}
M.~{\'E}mery, \emph{Stochastic calculus in manifolds}, Universitext,
  Springer-Verlag, Berlin, 1989, With an appendix by P.-A. Meyer.

\bibitem[ES64]{EelSam64}
J.~Eells and J.~H. Sampson, \emph{Harmonic mappings of {R}iemannian manifolds},
  Amer. J. Math. \textbf{86} (1964), 109--160.

\bibitem[ET15]{Espinosa-Touzi}
G.-E. Espinosa and N.~Touzi, \emph{Optimal investment under relative
  performance concerns}, Math. Finance \textbf{25} (2015), no.~2, 221--257.

\bibitem[Eva98]{evans}
L.~C. Evans, \emph{Partial differential equations}, Graduate Studies in
  Mathematics, vol.~19, American Mathematical Society, Providence, RI, 1998.

\bibitem[FdR11]{FreRei11}
C.~Frei and G.~dos Reis, \emph{A financial market with interacting investors:
  does an equilibrium exist?}, Math. Financ. Econ. \textbf{4} (2011), no.~3,
  161--182.

\bibitem[Fre88]{Fre88}
Jens Frehse, \emph{Remarks on diagonal elliptic systems}, Partial differential
  equations and calculus of variations, Lecture Notes in Math., vol. 1357,
  Springer, Berlin, 1988, pp.~198--210.

\bibitem[Fre14]{Frei-splitting}
C.~Frei, \emph{Splitting multidimensional {BSDE}s and finding local
  equilibria}, Stochastic Process. Appl. \textbf{124} (2014), 2654--2671.

\bibitem[Fri64]{Friedman}
A.~Friedman, \emph{Partial differential equations of parabolic type},
  Prentice-Hall, Inc., Englewood Cliffs, N.J., 1964.

\bibitem[HP06]{HuPen06}
Y.~Hu and S.~Peng, \emph{On the comparison theorem for multidimensional
  {BSDE}s}, C. R. Math. Acad. Sci. Paris \textbf{343} (2006), no.~2, 135--140.

\bibitem[Hsu02]{Hsu02}
E.~P. Hsu, \emph{Stochastic analysis on manifolds}, Graduate Studies in
  Mathematics, vol.~38, American Mathematical Society, Providence, RI, 2002.

\bibitem[HT15]{HuTan15}
Y.~Hu and S.~Tang, \emph{Multi-dimensional backward stochastic differential
  equations of diagonally quadratic generators}, to appear in Stochastic
  Process. Appl., 2015.

\bibitem[IKO62]{Ilin-et-al}
A.~M. Il{\cprime}in, A.~S. Kala{\v{s}}nikov, and O.~A. Ole{\u\i}nik,
  \emph{Second-order linear equations of parabolic type}, Russian Math. Surveys
  \textbf{17} (1962), no.~3 (105), 3--146.

\bibitem[Joh78]{Joh78}
F.~John, \emph{Partial differential equations}, third ed., Applied Mathematical
  Sciences, vol.~1, Springer-Verlag, New York-Berlin, 1978.

\bibitem[Ken90]{Kendall}
W.~S. Kendall, \emph{Probability, convexity, and harmonic maps with small
  image. {I}. {U}niqueness and fine existence}, Proc. London Math. Soc. (3)
  \textbf{61} (1990), no.~2, 371--406.

\bibitem[Kob00]{Kobylanski}
M.~Kobylanski, \emph{Backward stochastic differential equations and partial
  differential equations with quadratic growth}, Ann. Probab. \textbf{28}
  (2000), no.~2, 558--602.

\bibitem[KP16]{Kramkov-Pulido}
D.~Kramkov and S.~Pulido, \emph{A system of quadratic {BSDE}s arising in a
  price impact model}, To appear in Ann. Appl. Probab., 2016.

\bibitem[KX{\v{Z}}15]{Kardaras-Xing-Zitkovic}
C.~Kardaras, H.~Xing, and {\v{Z}}itkovi{\' c}, \emph{Incomplete stochastic
  equilibria with exponential utilities: close to {P}areto optimality}, Working
  paper, 2015.

\bibitem[Lej02]{Lejay}
A.~Lejay, \emph{B{SDE} driven by {D}irichlet process and semi-linear parabolic
  {PDE}. {A}pplication to homogenization}, Stochastic Process. Appl.
  \textbf{97} (2002), no.~1, 1--39.

\bibitem[Lie96]{Lie96}
G.~M. Lieberman, \emph{Second order parabolic differential equations}, World
  Scientific Publishing Co. Inc., River Edge, NJ, 1996.

\bibitem[LSM97]{Lepeltier-SanMartin}
J.~P. Lepeltier and J.~San~Martin, \emph{Backward stochastic differential
  equations with continuous coefficient}, Statist. Probab. Lett. \textbf{32}
  (1997), no.~4, 425--430.

\bibitem[LSU67]{LadSolUra67}
O.~A. Lady{\v{z}}enskaja, V.~A. Solonnikov, and N.~N. Ural{\cprime}ceva,
  \emph{Linear and quasilinear equations of parabolic type}, Translated from
  the Russian by S. Smith. Translations of Mathematical Monographs, Vol. 23,
  American Mathematical Society, Providence, R.I., 1967.

\bibitem[MX08]{Matoussi-Xu}
A.~Matoussi and M.~Xu, \emph{Sobolev solution for semilinear {PDE} with
  obstacle under monotonicity condition}, Electron. J. Probab. \textbf{13}
  (2008), no. 35, 1035--1067.

\bibitem[Pen99]{Pen99}
S.~Peng, \emph{Open problems on backward stochastic differential equations},
  Control of distributed parameter and stochastic systems ({H}angzhou, 1998),
  Kluwer Acad. Publ., Boston, MA, 1999, pp.~265--273.

\bibitem[PP90]{Pardoux-Peng}
{\'E}.~Pardoux and S.~Peng, \emph{Adapted solution of a backward stochastic
  differential equation}, Systems Control Lett. \textbf{14} (1990), no.~1,
  55--61.

\bibitem[PP92]{ParPen92}
\bysame, \emph{Backward stochastic differential equations and quasilinear
  parabolic partial differential equations}, Stochastic partial differential
  equations and their applications ({C}harlotte, {NC}, 1991), Lecture Notes in
  Control and Inform. Sci., vol. 176, Springer, Berlin, 1992, pp.~200--217.

\bibitem[Str81]{Str81}
M.~Struwe, \emph{On the {H}\"older continuity of bounded weak solutions of
  quasilinear parabolic systems}, Manuscripta Math. \textbf{35} (1981),
  no.~1-2, 125--145.

\bibitem[Sub91]{Sub91}
A.~Subrahmanyam, \emph{Risk aversion, market liquidity, and price efficiency},
  Rev. Financ. Stud. \textbf{4} (1991), no.~3, 417-- 441.

\bibitem[SV06]{StrVar06}
D.~W. Stroock and S.~R.~S. Varadhan, \emph{Multidimensional diffusion
  processes}, Classics in Mathematics, Springer-Verlag, Berlin, 2006, Reprint
  of the 1997 edition.

\bibitem[Tan03]{Tan03}
S.~Tang, \emph{General linear quadratic optimal stochastic control problems
  with random coefficients: linear stochastic {H}amilton systems and backward
  stochastic {R}iccati equations}, SIAM J. Control Optim. \textbf{42} (2003),
  no.~1, 53--75 (electronic).

\bibitem[Tev08]{Tevzadze}
R.~Tevzadze, \emph{Solvability of backward stochastic differential equations
  with quadratic growth}, Stochastic Process. Appl. \textbf{118} (2008), no.~3,
  503--515.

\bibitem[Udr94]{Udr94}
C.~Udri{\c{s}}te, \emph{Convex functions and optimization methods on
  {R}iemannian manifolds}, Mathematics and its Applications, vol. 297, Kluwer
  Academic Publishers Group, Dordrecht, 1994.

\bibitem[Wid71]{Widman1971}
K.-O. Widman, \emph{H\"older continuity of solutions of elliptic systems},
  Manuscripta Math. \textbf{5} (1971), 299--308.

\bibitem[Zha12]{Zha12}
Y.~Zhao, \emph{Stochastic equilibria in a general class of incomplete brownian
  market environments}, Ph.D. thesis, The University of Texas at Austin, 2012.

\bibitem[{\v{Z}}it12]{Zit12}
G.~{\v{Z}}itkovi{\'{c}}, \emph{An example of a stochastic equilibrium with
  incomplete markets}, Finance and Stochastics \textbf{16} (2012), no.~2,
  177--206.

\end{thebibliography}
\end{document}